\documentclass[12pt,a4paper]{article}

\usepackage[utf8]{inputenc}
\usepackage[english]{babel}

\usepackage{amsmath,amssymb,amsfonts,mathrsfs}

\usepackage[amsmath,thmmarks]{ntheorem}

\numberwithin{equation}{section}

\newtheorem{thm}{Theorem}[section]

\newtheorem{lem}[thm]{Lemma}
\newtheorem{prop}[thm]{Proposition}
\newtheorem{rem}[thm]{Remark}
\newtheorem{defn}[thm]{Definition}
\newtheorem{axm}[thm]{Axiom}
\newtheorem{exl}[thm]{Example}
\newtheorem{ques}[thm]{Question}
\newtheorem{notn}[thm]{Notation}
\newtheorem{conv}[thm]{Convention}

\theoremstyle{nonumberplain}
\theorembodyfont{\normalfont}
\theoremsymbol{\ensuremath{\square}}
\newtheorem{proof}{Proof}

\usepackage{hyperref}

\title{On a Completion of Cohomological Functors Generalising Tate Cohomology I}
\date{3 November 2026}
\author{Max Gheorghiu}

\usepackage[margin=2.5cm]{geometry}

\usepackage{tikz-cd}
\usepackage{stmaryrd}
\usepackage{graphicx}

\newcommand{\bigast}{\mathop{\scalebox{2}{\raisebox{-0.1ex}{$\ast$}}}}%
\newcommand{\triangleleftneq}{%
  \mathrel{\ooalign{$\lneq$\cr\raise.22ex\hbox{$\lhd$}\cr}}}

\begin{document}

%avoiding indentation at the beginning of each paragraph
\setlength{\parindent}{0cm}

%create a title page from the preamble info
\maketitle

\begin{abstract}
Tate cohomology has been generalised by several authors using different constructions that have applications in group theory, ring theory and homotopical algebra. Therefore, there is a need for a uniform account that explains why their underlying approaches all lead to the same conclusions. The key notion in such a uniform theory is a specific completion of cohomological functors that is constructed under mild assumptions. This completion takes Tate cohomology to settings where it has never been introduced such as in condensed mathematics. Through the latter, one can define Tate cohomology for any $T1$ topological group.
\end{abstract}

\tableofcontents

\section{Introduction}\label{sec:intro}

Originally, Tate cohomology $\widehat{H}^{\bullet}(G, -)$ was developed only for any finite group $G$ by Tate in the paper~\cite{tat52} motivated by class field theory. As it unites group homology and group cohomology of finite groups in a convenient manner, it attracted the attention of group theorists~\cite[p.~128]{bro82}. It was generalised to groups of finite virtual cohomological dimension in Farrell's paper~\cite{far77} and then to all groups in Benson and Carlson's paper~\cite{ben92}, Goichot's paper~\cite{goi92} and Mislin's paper~\cite{mis94}. The different approaches of the above papers have been used in generalised constructions by several authors with applications to group theory, ring theory and homotopical algebra. There is a need for a uniform account that explains why the approaches underlying these constructions all lead to the same conclusions: a theory that does not only work for discrete groups or modules over a ring, but also in greater generality. The main goal of this paper to provide a uniform treatment of the entire theory, putting in place detailed proofs that establish the equivalence of different definitions of the different authors. \\

This is the first in a series of two papers where we establish the fundamental properties of our resulting uniform theory in the second paper~\cite{ghe24}. There, we prove that our theory detects whether a group object in a category (such as a topological group) has finite cohomological dimension. It satisfies a form of dimension shifting, meaning that one can express a Tate cohomology group of any degree isomorphically in terms of any other degree. It satisfies an Eckmann–Shapiro Lemma, meaning that Tate cohomology of finite index (closed) subgroups of a discrete or profinite group can be described by Tate cohomology of the entire group. Cup products and Yoneda products are developed in Tate cohomology under certain conditions and their key properties are established. We refer to the paper for the details. \\

\textbf{Uniformity of the theory.} The key notion for the development of a uniform theory are connecting homomorphisms, one of the most important structures in cohomology that is particularly relevant in group cohomology as a computational tool. Connecting homomorphisms ensure that cohomology is more than a mere countable collection of functors. Thus far, authors have only demonstrated that the cohomology groups $\widehat{H}^{\bullet}(G, A)$ resulting from the above approaches are isomorphic. Even when they consider connecting homomorphisms $\widehat{\delta}^{\bullet}: \widehat{H}^{\bullet}(G, C) \rightarrow \widehat{H}^{\bullet+1}(G, A)$ for short exact sequences $0 \rightarrow A \rightarrow B \rightarrow C \rightarrow 0$, they do not devise any explicit formulae for them. This gave rise to the following two issues. Mislin in~\cite{mis94} provides a connecting homomorphism for Benson and Carlson's approach that is incorrect (Remark~\ref{rem:mislinsmistake}). S.\! Guo and L.\! Liang in~\cite{guo23} provide a connecting homomorphism for Vogel's approach that might render it incompatible to the other approaches (Question~\ref{ques:comparewithguoliang}). To remedy these issues, we provide the very first explicit formulae for the connecting homomorphisms in terms of each approach (Theorem~\ref{thm:satelliteconnmorphisms}, Definition~\ref{defn:resolconnmorph}, Theorem~\ref{thm:naiveconnmorph}, Definition~\ref{defn:hyperconnmorph}). \\

This does not yet ensure a uniform treatment of the entire theory. To prove that the resulting morphisms $\widehat{\delta}^{\bullet}: \widehat{H}^{\bullet}(G, -) \rightarrow \widehat{H}^{\bullet+1}(G, -)$ are indeed connecting homomorphisms, one needs to demonstrate that they are natural. Naturality means here that they commute with morphisms of short exact sequences. This in turn necessitates explicit formulae for induced morphisms in the functors $\widehat{H}^{\bullet}(G, -)$. As it has not been done in the literature either, we develop for any morphism $f: A \rightarrow B$ the first explicit formulae for induced morphisms $\widehat{H}^n(G, f): \widehat{H}^n(G, A) \rightarrow \widehat{H}^n(G, B)$ in terms of each construction (Theorem~\ref{thm:inducedsatellitemorphs}, Definition~\ref{defn:resolindmorph}, Definition~\ref{defn:inducednaivemorph}, Definition~\ref{defn:extchainmaps}). Our novel explicit formulae for connecting homomorphisms and induced morphisms build the foundation of our uniform theory which proves that all approaches by the above mentioned authors are equivalent. \\

Specifically, we formulate this uniform theory by constructing a completion of cohomological functors. Cohomological functors are a generalisation of group cohomology only preserving the connecting homomorphisms. More specifically, if $\mathcal{C}$, $\mathcal{D}$ are abelian categories, a family of additive functors $(T^n: \mathcal{C} \rightarrow \mathcal{D})_{n \in \mathbb{Z}}$ is a cohomological functor if there are connecting homomorphisms $(\delta^n: T^n \rightarrow T^{n+1})_{n \in \mathbb{Z}}$ satisfying two straightforward axioms~\cite[p.~201--202]{kro95}. If $G$ is a discrete group, $R$ a discrete ring and $A$ a discrete $R$-module, then setting $H_R^n(G, -) = 0$ and $\mathrm{Ext}_R^n(A,-) = 0$ for $n < 0$ renders group cohomology and Ext-functors into cohomological functors~\cite[p.~201]{kro95}, \cite[p.~295]{mis94}. Mislin proves in~\cite{mis94} that the generalisations of Tate cohomology by his two approaches satisfy a completion of cohomological functors. Following the convention from~\cite[p.~197/202]{kro95}, a Mislin completion of a cohomological functor $T^{\bullet}: \mathcal{C} \rightarrow \mathcal{D}$ is another cohomological functor $\widehat{T}^{\bullet}: \mathcal{C} \rightarrow \mathcal{D}$ together with a morphism $\Phi^{\bullet}: T^{\bullet} \rightarrow \widehat{T}^{\bullet}$ such that $\widehat{T}^n(P) = 0$ for any projective $P \in \mathrm{obj}(\mathcal{C})$ and $n \in \mathbb{Z}$. It satisfies the universal property that any morphism $T^{\bullet} \rightarrow V^{\bullet}$ to a cohomological functor $V^{\bullet}: \mathcal{C} \rightarrow \mathcal{D}$ also vanishing on projectives factors uniquely through $\Phi^{\bullet}$~\cite[Definition~2.1]{mis94}. By its universal property, any Mislin completion is unique up to isomorphism~\cite[p.~202]{kro95}. Mislin completions bear the advantage that they relate group cohomology to the resulting generalised Tate cohomology in a unique manner. \\

In summary, Mislin completions together with the formulae for their connecting homomorphisms and induced morphisms form our uniform theory that generalises Benson, Carlson, Vogel and Mislin's approaches and thus Tate cohomology. \\

\textbf{Generality of the theory.} We outline the above approaches in order to detail the extent to which we generalise them below. In the way they are originally phrased, these approaches pertain to all groups and more generally, to modules over a ring. Two approaches can be found in Mislin's paper~\cite{mis94}. In the first, Tate cohomology groups $\widehat{H}^{\bullet}(G, A)$ are obtained by direct limits of ordinary cohomology groups through connecting homomorphisms. Direct limits are a particular type of colimits~\cite[p.~14]{rib10} which are also called filtered colimits in other parts of the literature. In the second, Tate cohomology groups $\widehat{H}^{\bullet}(G, A)$ are obtained as direct limits of left satellite functors of ordinary cohomology groups. Left satellite functors were treated as an alternative foundation of homological algebra in the third chapter of Cartan and Eilenberg's book on homological algebra~\cite{car56}. Goichot in~\cite{goi92} defines Tate cohomology groups $\widehat{H}^{\bullet}(G, A)$ using a hypercohomology construction which he attributes to Vogel. Thus, we shall refer to this approach as Vogel's approach from now on. Benson and Carlson in~\cite{ben92} obtain their Tate cohomology groups $\widehat{H}^{\bullet}(G, A)$ by direct limits of (stable) Hom-functors using a stabilisation process, making their approach akin to stable module theory. \\

We generalise these approaches by Mislin, Vogel, Benson and Carlson by phrasing them in terms of Mislin completions. As Mislin's concern was group cohomology, he only constructed in~\cite{mis94} Mislin completions of functors $T^{\bullet}: \mathcal{C} \rightarrow \mathcal{D}$ in case $\mathcal{C}$ and $\mathcal{D}$ are categories of modules over a ring. We generalise his two approaches under the assumptions that $\mathcal{C}$ has enough projective objects and that in $\mathcal{D}$ all countable direct limits exist and are exact (Theorem~\ref{thm:satellitemislin} and Theorem~\ref{thm:resolsatellites}). These assumptions are mild: a standard way of developing homological algebra is via projectives where exact countable direct limits exist in any category of modules. In proving that his two approaches are equivaent, we establish the unprecedented fact that there are countably many distinct constructions of Mislin completions (Definition~\ref{defn:distinctconstrs}, Theorem~\ref{thm:countmanyconstrs}). In particular, that implies that there are countably many distinct generalisations of Tate cohomology. If $\mathbf{Ab}$ is the category of abelian groups, then we generalise Vogel’s approach from~\cite{goi92} and Benson and Carlson’s approach from~\cite{ben92} so that they give rise to the Mislin completion of Ext-functors $\widehat{\mathrm{Ext}}_{\mathcal{C}}^{\bullet}(A, -): \mathcal{C} \rightarrow \mathbf{Ab}$ (Theorem~\ref{thm:resolbenson} and Theorem~\ref{thm:hyperresol}). \\

Our theory generalises constructions from the literature that we outline in the following. It generalises Farrell's approach from~\cite{far77} as we demonstrate in~\cite[Lemma~1.11]{ghe24}. Farrell formulated it originally for groups of finite virtual cohomological dimension where Symonds extended it to profinite groups of finite virtual cohomological dimension in~\cite[p.~34]{sym07}. Our theory applies to all groups and to all profinite groups. Avramov and Veliche generalised the above mentioned approaches to Ext-functors of modules over a commutative Noetherian ring $R$ and used them to characterise when $R$ is regular, complete intersection or Gorenstein. Our theory encompasses their ring theoretic work. Our theory encompasses their ring theoretic work. Finally, our theory takes Tate cohomology to a new setting that is of current interest: for any totally disconnected locally compact group $G$ one can define the Mislin completion of its continuous cohomology $\widehat{H}_{\mathbb{Q}}^{\bullet}(G, -)$ with discrete $\mathbb{Q}[G]$-modules as coefficients using~\cite[Proposition~3.2]{cas16}. At the end of the introduction we compare how our theory relates to general constructions of Tate cohomology in the realm of homotopical algebra. \\

The generality of our theory prompts us to introduce the following terminology. Whenever considering groups or topological groups, we denote the Mislin completion of their cohomology as complete cohomology, in accordance with the convention from~\cite{kro95}. Analogously, we denote Mislin completions of Ext-functors as completed Ext-functors. As we explain in the following, these can be enriched or unenriched. In various contexts one might wish to take enriched Hom-functors
\[ \mathrm{Hom}(-, -): \mathcal{C}^{\mathrm{op}} \times \mathcal{C} \rightarrow \mathcal{D} \]
where $\mathcal{D}$ is an abelian category. One context is the cohomology of profinite groups where $\mathcal{D}$ is a category of profinite modules. Another context is sheaf theory where $\mathcal{D}$ is a category of sheaves of modules. As one important context, condensed mathematics is a powerful theory, which is based on a kind of sheaves of sets termed condensed sets. There, $\mathcal{D}$ is a category of condensed modules, of which one can think as sheaves of modules. The derived functors of enriched Hom-functors are called enriched Ext-functors. In case the category $\mathcal{D}$ is the category of abelian groups $\mathbf{Ab}$ or a category of modules over a ring, the derived functors are called unenriched Ext-functors. Therefore, one can take complete cohomology or completed Ext-functors to be enriched or unenriched.  \\

This creates an incentive for the following general setting in which to consider group cohomology and in which the properties of complete cohomology are formulated in~\cite{ghe24}. In as much as a group is a group object in the category of sets, a topological group is a group object in the category of topological spaces. One can thus realise profinite groups, algebraic groups and condensed groups as group objects in a respective category. To define group cohomology, one also requires a choice of coefficient ring, which is often implicitly assumed. In the case of profinite groups, these are profinite rings, namely ring objects in the category of profinite spaces. In condensed mathematics, one considers condensed rings, namely ring objects in the category of condensed sets. The coefficients for group cohomology are module objects in the respective category, such as profinite modules or condensed modules. More specifically, let $G$ be a group object acting on a ring object $R$ in a category $\mathcal{C}$. Then the coefficients for group cohomology $H_R^{\bullet}(G, -) = \mathrm{Ext}^{\bullet}(R, -)$ are module objects in $\mathcal{C}$ over both $G$ and $R$ such that the $G$- and $R$-actions commute. To ensure that $H_R^{\bullet}(G, -)$ possesses connecting homomorphisms, we tacitly assume that the category of these module objects is abelian. \\

\textbf{Application to topological groups.} As a remarkable feature, our generalisation of Tate cohomology is applicable to condensed mathematics and thus to most topological groups of interest. Condensed mathematics was developed by Clausen and Scholze in 2018~\cite{sch19}. In a nutshell, it provides a unified approach for studying topological groups, rings and modules~\cite[p.~6]{sch19}. In~\cite{sch20}, Scholze writes that he wants ``to make the strong claim that in the foundations of mathematics, one should replace topological spaces with condensed sets''. Here, a condensed set can be thought as a sheaf of sets. Let us substantiate Scholze's claim. In practice, most topological spaces (resp.\! groups, rings, etc) are $T1$, meaning that all their points are closed. These can be functorially turned into condensed sets (resp.\! groups, rings, etc), meaning that one can work only with the corresponding condensed object. \\

Condensed mathematics provides an example of a Mislin completion that cannot be treated through any previous framework from the literature. If $\mathcal{R}$ is a condensed ring, $\mathrm{Cond}(\mathrm{Mod}(\mathcal{R}))$ the category of condensed $\mathcal{R}$-modules and $\mathrm{Cond}(\mathbf{Ab})$ the category of condensed abelian groups, then the  Mislin completion of the enriched Ext-functor
\[ \underline{\widehat{\mathrm{Ext}}}_{\mathcal{R}}^{\bullet}(A, -): \mathrm{Cond}(Mod(\mathcal{R})) \rightarrow \mathrm{Cond}(\mathbf{Ab}) \]
can be constructed only by means of our uniform theory (Remark~\ref{rem:differencetoguoliang}). \\

One can define complete cohomology for all condensed groups and thus, for most topological groups as the following list demonstrates (Theorem~\ref{thm:condgroupcohom}).
\begin{itemize}
\item Complete cohomology can be defined for all Lie groups and in particular, for $p$-adic analytic groups.
\item Complete cohomology can be defined for any Galois group or equivalently, for any profinite group. A profinite group is an inverse limit of finite discrete groups, meaning that it can be assembled from finite groups in a particular manner. Profinite groups are of interest because they possess properties of infinite groups such as finite generation, but also properties of finite groups such as Sylow subgroups. They have recently featured in $3$-manifold topology via the study of profinite rigidity of their fundamental groups. Condensed mathematics is particularly well suited for investigating the cohomology of profinite groups (Lemma~\ref{lem:solidandctscohom}).
\item Complete cohomology can be defined for totally disconnected locally compact (tdlc) groups whose study constitutes an active research area. Tdlc groups can be defined as locally profinite groups, thus they generalise both discrete and profinite groups. One can interpret their profinite open subgroups as taking the role of finite subgroups in discrete groups. Tdlc groups play a crucial role in understanding the structure of the more general locally compact groups. Examples of tdlc groups also include $p$-adic analytic groups and automorphisms of locally finite trees with the compact-open topology.
\item However, complete cohomology cannot be defined for algebraic groups via condensed mathematics. Namely, any topological group can be turned into a condensed group because any cartesian product of two topological spaces results in a product of their corresponding condensed sets. In contrast, a product of two algebraic varieties with the Zariski topology does not result in a product of the corresponding condensed sets.
\end{itemize}

\textbf{Existing Calculations, questions.} Calculations of complete cohomology have been performed beyond finite groups, but are notoriously difficult. This is already indicated by the fact that the complete cohomology group $\widehat{H}^0(G, \mathbb{Z})$ for a group $G$ vanishes if and only if $G$ has finite cohomological dimension~\cite[Lemma~4.2.4]{kro95}. In~\cite{dem08}, F. Dembegioti calculates the zeroeth complete cohomology group of a class of polycyclic groups, but also explains why one cannot determine a general formula calculating the zeroeth complete cohomology groups for all polycyclic groups. On the one hand, the cohomology of the free abelian group of countable rank $\bigoplus_{n \in \mathbb{N}} \mathbb{Z}$ is isomorphic to its complete cohomology by~\cite[p.~297]{mis94}. On the other hand, all cohomology groups of the free product $\bigast_{n \in \mathbb{N}} \mathbb{Z}^n$ are distinct from its complete cohomology groups~\cite[p.~432]{jo09}. In~\cite[Corollary~2.4 and Theorem~A]{dem05}, F. Dembegioti calculates the zeroeth complete cohomology group of $\bigast_{n \in \mathbb{N}} \mathbb{Z}^n$. Moving from (discrete) groups to profinite groups, Symonds computes in~\cite{sym04} the complete cohomology of the Morava stabiliser group $S_{p-1}$ with coefficients in the moduli space $E_{p-1}$ for odd primes $p$. To our knowledge, this is the only calculation of complete cohomology for a non-discrete topological group. \\

Because our theory enables one to define complete cohomology for most topological groups, the legitimate question arises to what calculations it can be employed. As a starting point, we suggest to compute complete cohomology groups of profinite groups because there is already an extensive theory on their cohomology as can be seen in~\cite{rib10} and \cite{wil98}. Inspired by the above mentioned work of Dembegioti, we propose the following two questions.

\begin{ques}
Let $\prod_{\mathbb{N}} \mathbb{Z}_p$ denote the free abelian pro-$p$ group over $\mathbb{N}$ as in~\cite[Example~3.3.8(c)]{rib10}. Is its cohomology $H_{\mathbb{Z}_p}^{\bullet}(\prod_{\mathbb{N}} \mathbb{Z}_p, -)$ isomorphic to its complete cohomology $\widehat{H}_{\mathbb{Z}_p}^{\bullet}(\prod_{\mathbb{N}} \mathbb{Z}_p, -)$ as a cohomological functor?
\end{ques}

\begin{ques}
Set $G_n = \mathbb{Z}_p^n$ for $n \in \mathbb{N}$ and $G_{\infty} = \lbrace 1 \rbrace$. Define $G = \bigsqcup_{n \in \mathbb{N} \cup \lbrace \infty \rbrace} G_n$ as the free pro-$p$ product over the one-point compactification of the integers $\mathbb{N} \cup \lbrace \infty \rbrace$ using the definitions of~\cite[p.~137--140]{rib17}. Can one compute the zeroeth complete cohomology $\widehat{H}_{\mathbb{Z}_p}^0(G, A)$ for a $\mathbb{Z}_p\llbracket G \rrbracket$-module $A$?
\end{ques}

It is unlikely that one can compute the above complete cohomology groups using previous generalisations of Tate cohomology. Namely, in the case of discrete groups, the complete cohomology groups of the free abelian group $\bigoplus_{n \in \mathbb{N}} \mathbb{Z}$ and of the free product $\bigast_{n \in \mathbb{N}} \mathbb{Z}^n$ cannot be calculated using Farrell's approach~\cite[Example~5.3]{cor97}, \cite[p.~119--120]{dem05}. In particular, Symonds generalisation of Farrell's approach from~\cite[p.~34]{sym07} is unlikely to apply to the above questions. \\

\textbf{Comparison to the literature.} We compare how our work relates to general constructions of Tate cohomology in the realm of homotopical algebra. Recently, S. Guo and L. Liang have generalised in~\cite{guo23} the above constructions by Mislin, Vogel, Benson and Carlson for relative Ext-functors over any abelian category $\mathcal{C}$ with a special precovering subcategory $\mathcal{W}$. Although they prove that their completed Ext-functors $\widehat{\mathrm{Ext}}_{\mathcal{WC}}^n(A, -)$ form a cohomological functor, we do not know whether they also form a Mislin completion of the Ext-functors $\mathrm{Ext}_{\mathcal{WC}}^{\bullet}(A, -)$ (Question~\ref{ques:comparewithguoliang}). Further, a generalisation of Tate cohomology in the language of spectra can be found in~\cite{kle02}. Using a version of Vogel's approach, J.\! Hu et al.\! establish in~\cite{hu21} completed Ext-functors $\widehat{\mathrm{Ext}}_{\mathcal{C}}^{\bullet}(-, -)$ for an extriangulated category $\mathcal{C}$ where extriangulated categories are a generalisation of exact categories and triangulated categories. Using a version of Benson and Carlson's approach, Beligiannis and Reiten define in~\cite[Section~IX.2]{bel07} a completed relative Ext-functor $\widehat{\mathrm{Ext}}_{(\mathcal{X}, \mathcal{Y})}^{\bullet}(-, -)$ for an abelian category $\mathcal{C}$ containing a particular pair of subcategories $(\mathcal{X}, \mathcal{Y})$. They prove that $\mathrm{Ext}_{\mathcal{X} \cap \mathcal{Y}}^{\bullet}(A, -) \rightarrow \widehat{\mathrm{Ext}}_{(\mathcal{X}, \mathcal{Y})}^{\bullet}(A, -)$ forms a Mislin completion for any $A \in \mathrm{obj}(\mathcal{C})$. For an overview of more generalisations of Tate cohomology, we refer to the survey paper~\cite{pag16}. Instead of focussing on one particular approach, we use Mislin completions to create a uniform framework through which we prove that the approaches by Mislin, Vogel, Benson and Carlson yield the same conclusions. \\

Let us point out two variants of the constructions considered in this paper that are covered in the literature, but that we shall not pursue. One could perform the constructions of Mislin, Vogel, Benson and Carlson for \textit{homological functors} instead of cohomological functors. In~\cite{cel17}, Celikbas et al.\! generalise them for homological functors of modules over rings, but they turn out to be not equivalent. Hence, one cannot expect to encounter the same uniform theory as we develop it for cohomological functors. As was first observed by Nucinkis in~\cite{nuc98}, one could perform the above constructions dually by assuming that there are enough injectives instead of enough projectives. More specifically, a ``Nucinkis completion'' of a cohomological functor $T^{\bullet}: \mathcal{C} \rightarrow \mathcal{D}$ is a morphism $T^{\bullet} \rightarrow \widetilde{T}^{\bullet}$ such that $\widetilde{T}^{\bullet}$ vanishes on injectives and is universal with respect to this property. S.\! Guo and L.\! Liang perform all their constructions of completed relative Ext-functors in a dual manner in~\cite{guo23}. Together with X.\! Yang they provide in~\cite[Main Theorem]{guol23} a criterion for when $\widehat{\mathrm{Ext}}_{\mathcal{C}}^{\bullet}(-, -)$ is isomorphic to $\widetilde{\mathrm{Ext}}_{\mathcal{C}}^{\bullet}(-, -)$. Such criteria are also established by Beligiannis and Reiten in~\cite[Theorem~IX.4.4]{bel07} as well as by S.\! Hu et al.\! in~\cite[Theorem~4.4]{huz21}. This ``Nucinkis completion'' would not affect the cohomology of discrete and profinite groups because they already vanish on injectives~\cite[p.~61]{bro82},~\cite[p.~9]{ser94}. As this completion does not generalise Tate cohomology, we shall not consider it further. \\

Lastly, we demonstrate that complete cohomology à la Benson--Carlson--Mislin--Vogel does not agree with what in number theory is termed the Tate cohomology of a profinite group $G$. According to~\cite[Definition~I.1.9.3]{neu13}, the latter is defined for a discrete $G$-module $A$ as $\widehat{H}^i(G, A) := H^i(G, A)$ for $i \geq 1$ and as an inverse limit for $i \leq 0$. For $n \geq 1$ take $G = \mathbb{Z}_p^n$ to be the $n$-fold product of the $p$-adic integers and $A = \mathbb{Z}/p \mathbb{Z}$ with a trivial $G$-action. It follows from~\cite[Lemma 6.8.1]{rib10} that $H^1(\mathbb{Z}_p^n, \mathbb{Z}/p \mathbb{Z})$ is isomorphic to the set of continuous group homomorphisms $\mathrm{Hom}_{\mathrm{cts}}(\mathbb{Z}_p^n, \mathbb{Z}/p \mathbb{Z})$. Hence, $H^1(\mathbb{Z}_p^n, \mathbb{Z}/p \mathbb{Z}) = (\mathbb{Z}/p \mathbb{Z})^n$ does not vanish. Contrarily, complete cohomology $\widehat{H}^i(\mathbb{Z}_p^n, -)$ vanishes by~\cite[Theorem 1.2]{ghe24} because $\mathbb{Z}_p^n$ has finite cohomological dimension. \\

\textbf{Section summary.} We detail in Section~\ref{sec:constructions} the relevant constructions of Mislin, Vogel, Benson and Carlson that we are uniformising and generalising in this paper. The bulk of the paper consists of Section~\ref{sec:satellites}, Section~\ref{sec:resols}, Section~\ref{sec:naiveconstruction} and Section~\ref{sec:hypercohom} in which we develop the first explicit formulae for connecting homomorphisms and for induced morphisms of the resulting Mislin completions. More specifically, we prove in Section~\ref{sec:satellites} that one approach by Mislin gives rise to a Mislin completion in great generality. In each subsequent section, we demonstrate this for one approach each by constructing explicit isomorphisms of cohomological functors to another approach that already yields a Mislin completion. In this manner, we cover the other approach also due to Mislin in Section~\ref{sec:resols}, the approach due to Benson and Carlson in Section~\ref{sec:naiveconstruction} and the approach due to Vogel in Section~\ref{sec:hypercohom}. As a noteworthy feature, we implement the resulting uniform theory within condensed mathematics in Section~\ref{sec:condmaths}.

\subsection{Notation and terminology}

We adopt the convention that the natural numbers $\mathbb{N}$ include $1$, but do not include $0$. We write $\mathbb{N}_0$ for $\mathbb{N} \cup \lbrace 0 \rbrace$. Moreover, we adopt Poonen's convention from~\cite{poo18} that a ring is an abelian group together with a totally associative product, meaning a binary associative relation with an identity element. In particular, ring homomorphisms are understood to map identity elements to identity elements. We use the symbol $\varinjlim$ only to denote direct limits. If $\mathcal{D}$ is a category, then $\varinjlim_{\mathcal{D}}$ denotes a direct limit in this category and if $I$ is a directed set, then $\varinjlim_{i \in I}$ denotes a direct limit indexed over $I$. We write $\varinjlim_{k \in \mathbb{N}} (M_k, \mu_k)$ if $\mu_k: M_k \rightarrow M_{k+1}$ are the morphisms giving rise to the direct limit. Lastly, we use the same numbering to label diagrams and equations.

\section{Outline of constructions}\label{sec:constructions}

As the key notion of our uniform and general theory, we axiomatically define cohomological functors. For two abelian categories $\mathcal{C}$, $\mathcal{D}$ a family of additive functors $(T^n: \mathcal{C} \rightarrow \mathcal{D})_{n \in \mathbb{Z}}$ is a cohomological functor if it satisfies the following two axioms~\cite[p.~201--202]{kro95}.

\begin{axm}\label{axm:delta}
For every $n \in \mathbb{Z}$ and any short exact sequence $0 \rightarrow A \rightarrow B \rightarrow C \rightarrow 0$ in $\mathcal{C}$, there is natural connecting homomorphism $\delta^n: T^n(C) \rightarrow T^{n+1}(A)$.
\end{axm}

Being natural means in this context that for every commuting diagram in $\mathcal{C}$
\begin{center}
\begin{tikzcd}
0 \arrow[r] & A \arrow[r] \arrow[d, "f"] & B \arrow[r] \arrow[d] & C \arrow[r] \arrow[d, "g"] & 0 \\
0 \arrow[r] & A' \arrow[r] & B' \arrow[r] & C' \arrow[r] & 0
\end{tikzcd}
\end{center}
with exact rows there is a commuting diagram
\begin{center}
\begin{tikzcd}
T^n(A) \arrow[r, "\delta^n"] \arrow[d, "T^n(g)"] & T^{n+1}(A) \arrow[d, "T^{n+1}(f)"] \\
T^n(A') \arrow[r, "\delta^n"] & T^{n+1}(C')
\end{tikzcd}
\end{center}
in $\mathcal{D}$.

\begin{axm}\label{axm:les}
For every short exact sequence $0 \rightarrow A \xrightarrow{\iota} B \xrightarrow{\pi} C \rightarrow 0$ in $\mathcal{C}$ there is a long exact sequence
\[ \dots {} \xrightarrow{T^{n-1}(\pi)} T^{n-1}(C) \xrightarrow{\delta^{n-1}} T^n(A) \xrightarrow{T^n(\iota)} T^n(B) \xrightarrow{T^n(\pi)} T^n(C) \xrightarrow{\delta^n} T^{n+1}(A) \xrightarrow{T^{n+1}(\iota)} {} \dots \]
\end{axm}

We axiomatically define morphisms of cohomological functors as in~\cite[p.~202]{kro95}.

\begin{axm}\label{axm:morphisms}
Let $(T^{\bullet}, \delta^{\bullet})$ and $(U^{\bullet},\varepsilon^{\bullet})$ be cohomological functors from $\mathcal{C}$ to $\mathcal{D}$. Then a family of natural transformations $(\nu^n: T^n \rightarrow U^n)_{n \in \mathbb{Z}}$ is a morphism of cohomological functors if for every $n \in \mathbb{Z}$ and any short exact sequence $0 \rightarrow A \rightarrow B \rightarrow C \rightarrow 0$ in $\mathcal{C}$, the square
\begin{center}
\begin{tikzcd}
  T^n(C) \arrow[r, "\delta^n"] \arrow[d, "\nu^n"]
    &T^{n+1}(A) \arrow[d, "\nu^{n+1}"] \\
  U^n(C) \arrow[r, "\varepsilon^n"]
    &U^{n+1}(A)
\end{tikzcd}
\end{center}
commutes.
\end{axm}

We generalise Mislin's definition of Mislin completions from~\cite[Definition~2.1]{mis94} to a great extent.

\begin{defn}[Mislin completion]\label{defn:mislincompletion}
Let $(T^{\bullet}, \delta^{\bullet})$ be a cohomological functor from $\mathcal{C}$ to $\mathcal{D}$. Then its Mislin completion is a cohomological functor $(\widehat{T}^{\bullet}, \widehat{\delta}^{\bullet})$ from $\mathcal{C}$ to $\mathcal{D}$ together with a morphism $\nu^{\bullet}: T^{\bullet} \rightarrow \widehat{T}^{\bullet}$ satisfying the following universal property: \newline
\textnormal{1.} $(\widehat{T}^{\bullet}, \widehat{\delta}^{\bullet})$ vanishes on projectives, meaning that $\widehat{T}^n(P) = 0$ for every projective object $P$ in $\mathcal{C}$ and every $n \in \mathbb{Z}$. \newline
\textnormal{2.} If $(U^{\bullet}, \varepsilon^{\bullet})$ is any cohomological functor vanishing on projectives, then every morphism $T^{\bullet} \rightarrow U^{\bullet}$ factors uniquely as $T^{\bullet} \xrightarrow{\nu^{\bullet}} \widehat{T}^{\bullet} \rightarrow U^{\bullet}$.
\end{defn}

By virtue of their universal property, Mislin completions are unique up to isomorphism in the following sense. If $(U^{\bullet}, \varepsilon^{\bullet})$ is another Mislin completion of $(T^{\bullet}, \delta^{\bullet})$, then there is an isomorphism $\mu^{\bullet}: \widehat{T}^{\bullet} \rightarrow U^{\bullet}$, meaning that $\mu^n(M): \widehat{T}^n(M) \rightarrow U^n(M)$ is an isomorphism for any $n \in \mathbb{Z}$ and $M \in \mathrm{obj}(\mathcal{C})$~\cite[p.~202]{kro95}. The following thus generalises Mislin's definition of complete cohomology from~\cite[p.~297]{mis94} to a great extent.

\begin{defn}[Axiomatic definition, Mislin]\label{defn:axiomatic}
${} \quad {}$
\begin{itemize}
\item For any $A \in \mathrm{obj}(\mathcal{C})$ extend the (enriched or unenriched) Ext-functors to a cohomological functor by setting $\mathrm{Ext}_{\mathcal{C}}^n(A, -) = 0$ for $n < 0$. Define completed Ext-functors as the Mislin completion $(\widehat{\mathrm{Ext}}_{\mathcal{C}}^{\bullet}(A, -), \widehat{\delta}^{\bullet})$.
\item Accordingly, let $G$ be a group object that acts on a ring object $G$ in $\mathcal{C}$ such that the category of module objects in $\mathcal{C}$ with commuting $G$- and $R$-actions forms an abelian category. Then extend group cohomology to a cohomological functor $(H_R^{\bullet}(G, -), \delta^{\bullet})$ by imposing $H_R^n(G, -) = 0$ for $n < 0$. Define complete cohomology as the Mislin completion $(\widehat{H}_R^{\bullet}(G, -), \widehat{\delta}^{\bullet})$.
\end{itemize}
\end{defn}

By Theorem~\ref{thm:satellitemislin}, Mislin completions exist if the domain category $\mathcal{C}$ of the cohomological functor has enough projectives on the one hand. On the other hand, in the codomain category $\mathcal{D}$ all countable direct limits need to exist and be exact, where the latter is defined as

\begin{defn}
Let $\mathcal{D}$ be a category in which all direct limits exist. Then direct limits in $\mathcal{D}$ are called exact if for every direct system of short exact sequence
\[ \lbrace 0 \rightarrow A_i \rightarrow B_i \rightarrow C_i \rightarrow 0 \rbrace_{i \in I} \]
also
\[ 0 \rightarrow \varinjlim_{i \in I} A_i \rightarrow \varinjlim_{i \in I} B_i \rightarrow \varinjlim_{i \in I} C_i \rightarrow 0 \]
is a short exact sequence~\cite[\href{https://stacks.math.columbia.edu/tag/079A}{Tag 079A}]{stacks-project}.
\end{defn}

To clarify where the above assumptions on $\mathcal{C}$ and $\mathcal{D}$ are needed, we present what we term the satellite functor construction, which is due to Mislin. First, we introduce left satellite functors. Using the assumption that $\mathcal{C}$ has enough projective objects, there is for any $M \in \mathrm{obj}(\mathcal{C})$ a short exact sequence $0 \rightarrow K \rightarrow P \rightarrow M \rightarrow 0$ in $\mathcal{C}$ with $P$ projective. For a cohomological functor $(T^{\bullet}, \delta^{\bullet})$, we define the zeroeth left satellite functor of $T^n$ as $S^0T^n := T^n$, the first left satellite functor as
\[ S^{-1}T^n(M) := \mathrm{Ker}\big(T^n(K) \rightarrow T^n(P)\big) \]
and the $k^{\text{th}}$ left satellite functor as $S^{-k}T^n := S^{-1}(S^{-k+1}T^n)$ for $k \geq 2$~\cite[p.~36]{car56}. It is shown in~\cite[Section~III.1]{car56} that left satellite functors do not depend on the choice of short exact sequence. Since they have been defined as kernels, it follows from Axiom~\ref{axm:les} that $\delta^n: T^n \rightarrow T^{n+1}$ induces a morphism $\underline{\delta}^n: T^n(M) \rightarrow S^{-1}T^{n+1}(M)$ and therefore
\begin{equation}\label{eq:thefactorisation}
S^{-k}\underline{\delta}^{n+k}: S^{-k}T^{n+k}(M) \rightarrow S^{-k-1}T^{n+k+1}(M)
\end{equation}
for any $k \in \mathbb{N}$~\cite[p.~207--208]{kro95}. We extend Mislin's construction from~\cite[p.~293]{mis94} by explicitly using our assumption that all countable  direct limits exist in the codomain category $\mathcal{D}$ of $T^{\bullet}$.

\begin{defn}[Satellite functor construction, Mislin]
The Mislin completion of a cohomological functor $(T^{\bullet}, \delta^{\bullet})$ can be defined as
\[ \widehat{T}^n(M) := \varinjlim_{k \in \mathbb{N}_0} (S^{-k}T^{n+k}(M), S^{-k}\underline{\delta}^{n+k}) \]
for any $M \in \mathrm{obj}(\mathcal{C})$ and $n \in \mathbb{Z}$. Accordingly,
\[ \widehat{H}_R^n(G, M) := \varinjlim_{k \in \mathbb{N}_0} (S^{-k}H_R^{n+k}(G, M), S^{-k}\underline{\delta}^{n+k}) \]
is a definition of complete cohomology.
\end{defn}

In order that the above forms a cohomological functor, we require the very assumption that all countable direct limits in $\mathcal{D}$ are exact. \\

In the same generality, we present what we term the resolution construction that occurs in~\cite[Lemma~B.3]{cel17} and can be retrieved from Mislin's paper~\cite[p.~299 ]{mis94}. If $(M_n)_{n \in \mathbb{N}_0}$ is a projective resolution of $M \in \mathrm{obj}(\mathcal{C})$, define $\widetilde{M}_0 := M$ and $\widetilde{M}_k := \mathrm{Ker}(M_{k-1} \rightarrow \widetilde{M}_{k-1})$ for $k \in \mathbb{N}$. This is called the $k^{\text{th}}$ syzygy of $M_{\bullet}$~\cite[p.~89]{pag16}. The choice of our notation is meant to reflect that our syzygies do not necessarily arise from a specific choice of projective resolution as in \cite{kro95} and \cite{mis94}. Since for every $k \in \mathbb{N}_0$ we have the short exact sequence $0 \rightarrow \widetilde{M}_{k+1} \rightarrow M_k \rightarrow \widetilde{M}_k \rightarrow 0$, there is a connecting homomorphism $\delta^{n+k}: T^{n+k}(\widetilde{M}_k) \rightarrow T^{n+k+1}(\widetilde{M}_{k+1})$ for every $n \in \mathbb{Z}$. Then the following definition makes it more apparent why Mislin completions vanish on projective objects.

\begin{defn}[Resolution construction, Mislin]\label{defn:resolsinoutline}
The Mislin completion of a cohomological functor $(T^{\bullet}, \delta^{\bullet})$ can be defined as 
\[ \widehat{T}^n(M) := \varinjlim_{k \in \mathbb{N}_0} (T^{n+k}(\widetilde{M}_k), \delta^{n+k}) \]
for any $n \in \mathbb{Z}$ and $M \in \mathrm{obj}(\mathcal{C})$. Accordingly,
\[ \widehat{H}_R^n(G, M) := \varinjlim_{k \in \mathbb{N}_0} (H_R^{n+k}(G, \widetilde{M}_k), \delta^{n+k}) \]
is a definition of complete cohomology.
\end{defn}

The last two constructions give rise to completed Ext-functors which are, in our terminology, the naïve construction by Benson and Carlson and the hypercohomology construction by Vogel. Let $A_{\bullet}$, $B_{\bullet}$ be projective resolutions of $A, B \in \mathrm{obj}(\mathcal{C})$. Because there are enough projectives in $\mathcal{C}$, there is for any morphism $\widetilde{f}_{n+k}: \widetilde{A}_{n+k} \rightarrow \widetilde{B}_k$ with $n \in \mathbb{Z}$ and $n+k \geq 0$ a lift, meaning there is a commuting diagram of the form
\begin{center}
\begin{tikzcd}
    0 \arrow[r] & \widetilde{A}_{n+k+1} \arrow[r] \arrow[d, "\widetilde{f}_{k+1}"] & A_{n+k} \arrow[r] \arrow[d, "f_{k+1}"] & \widetilde{A}_{n+k} \arrow[r] \arrow[d, "\widetilde{f}_k"] & 0 \\
    0 \arrow[r] & \widetilde{B}_{k+1} \arrow[r] & B_{k+1} \arrow[r] & \widetilde{B}_k \arrow[r] & 0
\end{tikzcd}
\end{center}
Note that $\mathrm{Hom}_{\mathcal{C}}(\widetilde{A}_{n+k}, \widetilde{B}_k)$ is an abelian group by virtue of $\mathcal{C}$ being an abelian category. Define $\mathcal{P}_{\mathcal{C}}(\widetilde{A}_{n+k}, \widetilde{B}_k)$ to be the subgroup of $\mathrm{Hom}_{\mathcal{C}}(\widetilde{A}_{n+k}, \widetilde{B}_k)$ consisting of all morphisms factoring through a projective and write
\begin{equation}\label{eq:stablehom}
[\widetilde{A}_{n+k}, \widetilde{B}_k]_{\mathcal{C}} := \mathrm{Hom}_{\mathcal{C}}(\widetilde{A}_{n+k}, \widetilde{B}_k)/ \mathcal{P}_{\mathcal{C}}(\widetilde{A}_{n+k}, \widetilde{B}_k)
\end{equation}
for the quotient~\cite[p.~203]{kro95}. As in the case of modules over a ring~\cite[p.~204]{kro95}, one can prove that
\begin{align}
t_{\widetilde{A}_{n+k}, \widetilde{B}_k}: [\widetilde{A}_{n+k}, \widetilde{B}_k]_{\mathcal{C}} &\rightarrow [\widetilde{A}_{n+k+1}, \widetilde{B}_{k+1}]_{\mathcal{C}}, \label{eq:transition} \\
\widetilde{f}_k + \mathcal{P}_{\mathcal{C}}(\widetilde{A}_{n+k}, \widetilde{B}_k) &\mapsto \widetilde{f}_{k+1} + \mathcal{P}_{\mathcal{C}}(\widetilde{A}_{n+k+1}, \widetilde{B}_{k+1}) \nonumber
\end{align}
is a well defined homomorphism. Using this, we generalise the naïve construction from Benson and Carlson's paper~\cite[p.~109]{ben92}.

\begin{defn}[Naïve construction, Benson \& Carlson]\label{defn:naive}
For $n \in \mathbb{Z}$, the $n^{\text{th}}$ completed unenriched Ext-functor can be defined as
\[ \widehat{\mathrm{Ext}}_{\mathcal{C}}^n(A, B) := \varinjlim_{k \in \mathbb{N}_0, n+k \geq 0} ([\widetilde{A}_{n+k}, \widetilde{B}_k]_{\mathcal{C}}, t_{\widetilde{A}_{n+k}, \widetilde{B}_k}) \, . \]
Accordingly, if $R_{\bullet}$ is a projective resolution of $R$, then
\[ \widehat{H}_R^n(G, B) := \varinjlim_{k \in \mathbb{N}_0, n+k \geq 0} ([\widetilde{R}_{n+k}, \widetilde{B}_k]_{\mathcal{C}}, t_{\widetilde{R}_{n+k}, \widetilde{B}_k}) \]
is a definition of complete cohomology.
\end{defn}

Lastly, we present the hypercohomology construction due to Vogel. Define the chain complex $(A_n')_{n \in \mathbb{Z}}$ by $A_n' = A_n$ for $n \geq 0$ and $A_n' = 0$ for $n < 0$ where $(B_n')_{n \in \mathbb{Z}}$ is defined analogously~\cite[p.~209]{kro95}. Define the hypercohomology complex $(\mathrm{Hyp}_{\mathcal{C}}(A_{\bullet}', B_{\bullet}')_n, d^n)_{n \in \mathbb{Z}}$ by having $n$-cochains 
\[ \mathrm{Hyp}_{\mathcal{C}}(A_{\bullet}', B_{\bullet}')_n = \prod_{k \in \mathbb{Z}} \mathrm{Hom}_{\mathcal{C}}(A_{k+n}', B_k') \, . \]
To ease notation, we view abelian groups as $\mathbb{Z}$-modules. If $a_n: A_n' \rightarrow A_{n-1}'$ and $b_n: B_n' \rightarrow B_{n-1}'$ denote the differentials induced from the respective projective resolutions, define the differential of $\mathrm{Hyp}_{\mathcal{C}}(A_{\bullet}', B_{\bullet}')_{\bullet}$ for $n \in \mathbb{Z}$ as
\begin{align}
d^n: \mathrm{Hyp}_{\mathcal{C}}(A_{\bullet}', B_{\bullet}')_n &\rightarrow \mathrm{Hyp}_{\mathcal{C}}(A_{\bullet}', B_{\bullet}')_{n+1} \label{eq:hypercohombd} \\ (\varphi_{n+k})_{k \in \mathbb{Z}} &\mapsto (b_{k+1} \circ \varphi_{n+k+1} - (-1)^n \varphi_{n+k} \circ a_{n+k+1})_{k \in \mathbb{Z}} \nonumber
\end{align}
Define the bounded complex $\mathrm{Bdd}_{\mathcal{C}}(A_{\bullet}', B_{\bullet}')_{n \in \mathbb{Z}}$ as the subcomplex of $\mathrm{Hyp}_{\mathcal{C}}(A_{\bullet}', B_{\bullet}')_{n \in \mathbb{Z}}$ given by
\[ \mathrm{Bdd}_{\mathcal{C}}(A_{\bullet}', B_{\bullet}')_n = \bigoplus_{k \in \mathbb{Z}} \mathrm{Hom}_{\mathcal{C}}(A_{k+n}', B_k') \, . \]
Finally, define the Vogel complex as the quotient complex~\cite[p.~209]{kro95}
\[ \mathrm{Vog}_{\mathcal{C}}(A_{\bullet}', B_{\bullet}')_{n \in \mathbb{Z}} := \Big(\mathrm{Hyp}_{\mathcal{C}}(A_{\bullet}', B_{\bullet}')_n / \mathrm{Bdd}_{\mathcal{C}}(A_{\bullet}', B_{\bullet}')_n \Big)_{n \in \mathbb{Z}} \, . \]
We thus generalise the following construction from Goichot's paper~\cite[Definition~1.2]{goi92} where he attributes it to Vogel on page~39.

\begin{defn}[Hypercohomology construction, Vogel]\label{defn:vogel}
For $n \in \mathbb{Z}$, the $n^{\text{th}}$ completed unenriched Ext-functor can be defined as
\[ \widehat{\mathrm{Ext}}_{\mathcal{C}}^n(A, B) := H^n((\mathrm{Vog}_{\mathcal{C}}(A_{\bullet}', B_{\bullet}')_k, d^k)_{k \in \mathbb{Z}}) \, . \]
Accordingly, if $R_{\bullet}$ is a projective resolution of $R$, then
\[ \widehat{H}_R^n(G, M) := H^n((\mathrm{Vog}_R(R_{\bullet}', B_{\bullet}')_k, d^k)_{k \in \mathbb{Z}}) \]
is a definition of complete cohomology.
\end{defn}

Let us remark why it is unlikely that the previous two constructions could yield Mislin completions of enriched Ext-functors. For the naïve construction, one aims to find a morphism $\mathrm{Hom}_{\mathcal{C}}(\widetilde{A}_{n+k}, \widetilde{B}_k) \rightarrow \mathrm{Hom}_{\mathcal{C}}(\widetilde{A}_{n+k+1}, \widetilde{B}_{k+1})$. If one tries to lift along the morphisms induced by $A_{n+k} \rightarrow \widetilde{A}_{n+k}$ and $B_{n+k} \rightarrow \widetilde{B}_{n+k}$, one requires that $\mathrm{Hom}_{\mathcal{C}}(A_{n+k}, -)$ preserves epimorphisms. For unenriched Hom-functors, this role is exactly played by projective objects~\cite[Lemma~2.2.3]{wei94}. For the hypercohomology construction we require that the coproduct $\mathrm{Bdd}_{\mathcal{C}}(A_{\bullet}', B_{\bullet}')$ maps into the product $\mathrm{Hyp}_{\mathcal{C}}(A_{\bullet}', B_{\bullet}')$. As we shall see in Subsection~\ref{subsec:hyperindconn}, the cohomology groups of the Vogel complex $\mathrm{Vog}_{\mathcal{C}}(A_{\bullet}', B_{\bullet}')$ correspond exactly to almost chain maps modulo almost chain homotopy. However, this implies that the objects $\mathrm{Hom}_{\mathcal{C}}(A_p, B_q)$ describe morphisms of the form $A_p \rightarrow B_q$ in $\mathcal{C}$.

\section{The satellite functor construction and Mislin completions}\label{sec:satellites}

By generalising the satellite functor construction in this section, we establish Mislin completions in full generality. More specifically, Mislin constructs in the proof of \cite[Theorem~2.2]{mis94} Mislin completions through satellite functors only for modules over a ring. We generalise his work such that one can define Mislin completions whenever the domain category $\mathcal{C}$ is an abelian category with enough projective objects and in the codomain category $\mathcal{D}$ all countable direct limits exist and are exact. This section also sets the cornerstone of our uniform theory as it presents the very first explicit formulae for the resulting connecting homomorphisms and induced morphisms. Namely, these formulae form the base for analogous explicit formulae of all other constructions of Mislin completions $(\widehat{T}^{\bullet}, \widehat{\delta}^{\bullet})$. \\

For the reader's convenience, we summarise relevant results on left satellite functors from the third chapter of Cartan and Eilenberg's book~\cite{car56}. Because Grothendieck has generalised their satellite functors in his Tôhoku paper~\cite[p.~140--143]{gro57}, their arguments pertain to our categories $\mathcal{C}$ and $\mathcal{D}$. The first relevant result is an explicit description of induced morphisms for left satellite functors that we present in the following. Consider a morphism $f: C' \rightarrow C$ in $\mathcal{C}$. Using the assumption that the category $\mathcal{C}$ has enough projectives, we choose short exact sequences
\[ 0 \rightarrow M' \rightarrow P' \rightarrow C' \rightarrow 0 \quad \text{and} \quad 0 \rightarrow M \rightarrow P \rightarrow C \rightarrow 0 \]
with $P'$ and $P$ projective. Then there exists a lift $\overline{f}^{\ast}: P' \rightarrow P$ rendering the diagram
\begin{equation}\label{diag:satellitemorph}
\begin{tikzcd}
  0 \arrow[r]
    &M' \arrow[r] \arrow[d, "\overline{f}^{\ast}"]
    &P' \arrow[r] \arrow[d, "\overline{f}"]
    &C' \arrow[r] \arrow[d, "f"]
    &0 \\
 0 \arrow[r]
    &M \arrow[r]
    &P \arrow[r]
    &C \arrow[r]
    &0
\end{tikzcd}
\end{equation}
commutative. The morphism morphism $S^{-1}T^n(f): S^{-1}T^n(C') \rightarrow S^{-1}T^n(C)$ is induced via the diagram
\begin{center}
\begin{tikzcd}
    S^{-1}T^n(C') \arrow[r, "\varepsilon^{-1}"] \arrow[d, dashed, "S^{-1}T^n(f)"] & T^n(M') \arrow[d, "T^n(\overline{f}^{\ast})"] \\
    S^{-1}T^n(C) \arrow[r, "\varepsilon^{-1}"] & T^n(M)
\end{tikzcd}
\end{center}
where the morphisms $\varepsilon^{-1}$ are the canonical monomorphisms from the kernels. The induced morphisms $S^{-k}T^n(f): S^{-k}T^n(C') \rightarrow S^{-k}T^n(C)$ for higher satellite functors are defined analogously. \\

There are connecting homomorphisms between left satellite functors that are indispensable for our results. For any exact sequence $0 \rightarrow A \rightarrow B \rightarrow C \rightarrow 0$ in $\mathcal{C}$ and any $k \in \mathbb{N}_0$ these connecting homomorphisms take the form $\varepsilon^{-k-1}: S^{-k-1}T^n(C) \rightarrow S^{-k}T^n(A)$. If $B$ is projective, then $\varepsilon^{-k-1}$ is the canonical monomorphism of the kernel $S^{-k-1}T^n(C)$. Otherwise let $0 \rightarrow M \rightarrow P \rightarrow C \rightarrow 0$ be a short exact sequence with $P$ projective. Let $h: P \rightarrow B$ be a morphism lifting $\mathrm{id}_C: C \rightarrow C$ as in the commutative diagram
\begin{equation}\label{diag:delta}
\begin{tikzcd}
  0 \arrow[r]
    &M \arrow[r] \arrow[d, "h^{\ast}"]
    &P \arrow[r] \arrow[d, "h"]
    &C \arrow[r] \arrow[d, "\mathrm{id}_{C}"]
    &0 \\
 0 \arrow[r]
    &A \arrow[r]
    &B \arrow[r]
    &C \arrow[r]
    &0
\end{tikzcd}
\end{equation}
Then 
\[ \varepsilon^{-k-1}: S^{-k-1}T^n(C) \rightarrow S^{-k}T^n(M) \xrightarrow{S^{-k}T^n(h^{\ast})} S^{-k}T^n(A) \]
is a connecting homomorphism where the first morphism is the canonical monomorphism of the kernel $S^{-n-1}T^n(C)$. The term connecting homomorphism is justified because the morphisms $\varepsilon^{-k}$ form natural transformations giving rise to a long exact sequence
\[ \dots {} \rightarrow S^{-k-1}T^n(C) \xrightarrow{\varepsilon^{-k-1}} S^{-k}T^n(A) \rightarrow S^{-k}T^n(B) \rightarrow S^{-k}T^n(C) \xrightarrow{\varepsilon^{-k}} S^{-k+1}T^n(A) \rightarrow {} \dots \]
The starting point of the satellite functor construction of Mislin completions is the following result that can be directly generalised from \cite[Proposition~III.5.2]{car56}.

\begin{lem}\label{lem:uniqueness}
Let $T^{\bullet}$ and $U^{\bullet}$ be cohomological functors and $\Phi^0: T^0 \rightarrow U^0$ be a natural transformation.
\begin{enumerate}
\item Assume that for every short exact sequence $0 \rightarrow K \rightarrow P \rightarrow M \rightarrow 0$ with $P$ projective the sequence $0 \rightarrow U^{-n-1}(M) \rightarrow U^{-n}(K)$ involving the connecting homomorphism is exact for every $n \in \mathbb{N}$. Then $\Phi^0$ extends uniquely to a (partial) morphism of cohomological functors $(\Phi^n: T^n \rightarrow U^n)_{n \leq 0}$ only defined for $n \leq 0$.
\item Further, if the sequence $0 \rightarrow T^{-n-1}(M) \rightarrow T^{-n}(K)$ is exact for every $n \in \mathbb{N}$ and $\Phi^0$ is an equivalence, then the extension $\Phi^{\bullet}$ is an isomorphism.
\end{enumerate}
\end{lem}

Mislin uses this lemma to perform the satellite functor construction in~\cite[pp.~295--296]{mis94} which we generalise in the following. For any $m \in \mathbb{N}_0$ we can define the cohomological functor
\begin{equation}\label{eq:directsyscohomfunctor}
T^k \langle m \rangle (M) := \begin{cases} S^{k-m}T^m(M) &\text{if } k < m \\ T^k(M) &\text{if } k \geq m \end{cases}
\end{equation}
where the connecting homomorphisms are given by $\delta^k \langle m \rangle = \delta^k: T^k \rightarrow T^{k+1}$ for $k \geq m$ and $\delta^k \langle m \rangle = \varepsilon^{k-m}: S^{k-m}T^m \rightarrow S^{k-m+1}T^m$ for $k < m$. Consider the partial morphism $(\Phi_m^k := \mathrm{id}: T^k \rightarrow T^k \langle m \rangle)_{m \leq k}$ of cohomological functors defined only in degrees $m \leq k$. According to the previous lemma we can extend $\Phi_m^{\bullet}$ uniquely to a morphism of cohomological functors defined in all degrees. Using the same arguments we have for $m \leq n$ a unique morphism $\Phi_{m, n}^{\bullet}: T^{\bullet} \langle m \rangle \rightarrow T^{\bullet} \langle n \rangle$ that is equal to the identity morphism in any degree $n \leq k \in \mathbb{Z}$. Using the uniqueness of all these morphisms, we conclude for $m \leq n \leq o \in \mathbb{N}_0$ that $\Phi_n^{\bullet} = \Phi_{m,n}^{\bullet} \circ \Phi_m^{\bullet}$ and $\Phi_{m,o}^{\bullet} = \Phi_{n,o}^{\bullet} \circ \Phi_{m,n}^{\bullet}$. This yields a morphism to a direct system $T^{\bullet} \rightarrow (T^{\bullet} \langle m \rangle, \Phi_{m,n}^{\bullet})_{m \leq n \in \mathbb{N}_0}$. This is where the existence of countable direct limits $\varinjlim_{\mathcal{D}, k \in \mathbb{N}_0}$ in the codomain category $\mathcal{D}$ comes into play. For every $n \in \mathbb{Z}$ we define the $n^{\text{th}}$ term of the satellite functor construction as
\[ \widehat{T}^n(M) := \varinjlim_{k \in \mathbb{N}_0} T^n \langle k \rangle(M) \, . \]
Unravelling the definitions and noting that $\lbrace k \in \mathbb{N}_0 \mid k \geq n \rbrace$ is cofinal in $\mathbb{N}_0$, we see that
\begin{equation}\label{eq:expliciterms}
\widehat{T}^n(M) =\!\!\!\! \varinjlim_{k \in \mathbb{N}_0, k \geq n} \big(S^{n-k}T^k(M), \Phi_{k, k+1}^n(M) \big) = \varinjlim_{k \in \mathbb{N}_0} \big( S^{-k}T^{n+k}(M), \Phi_{n+k, n+k+1}^n(M) \big) \, .
\end{equation}
One obtains the connecting homomorphisms of the above satellite functor construction as follows. By Axiom~\ref{axm:les}, every short exact sequence $0 \rightarrow A \rightarrow B \rightarrow C \rightarrow 0$ in the category $\mathcal{C}$ yields a long exact sequence in $T^{\bullet}$ and $T^{\bullet} \langle m \rangle$ for any $m \in \mathbb{N}_0$. We obtain a long exact sequence in $\widehat{T}^{\bullet} = \varinjlim_{\mathcal{D}, m \in \mathbb{N}_0} T^{\bullet} \langle m \rangle$ and in particular, a connecting homomorphism $\widehat{\delta}^{\bullet}$ only if the direct limit functor $\varinjlim_{\mathcal{D}, m \in \mathbb{N}_0}$ is exact. This proves the following generalisation of~\cite[Theorem~2.2]{mis94}.

\begin{thm}\label{thm:satellitemislin}
Let $\mathcal{C}$ be an abelian category with enough projective object and let $\mathcal{D}$ be an abelian category in which all direct limits exist and are exact. Let $(T^{\bullet}, \delta^{\bullet}): \mathcal{C} \rightarrow \mathcal{D}$ be any cohomological functor.  Then
\begin{itemize}
\item a Mislin completion $(\widehat{T}^{\bullet}, \widehat{\delta}^{\bullet})$ of $(T^{\bullet}, \delta^{\bullet})$ exists and
\item can be defined via the the satellite functor construction.
\end{itemize}
\end{thm}

\begin{notn}
For the rest of the paper we assume that $\mathcal{C}$ is an abelian category with enough projective object and that $\mathcal{D}$ is an abelian category in which all direct limits exist and are exact.
\end{notn}

We now provide the first explicit formula for the connecting homomorphisms of the satellite functor construction. For this, note that one can extend the concept of left satellite functors from~\cite[Corollary~5.3]{car56} to natural transformations. Let $\varphi: F \rightarrow G$ be a natural transformation between two functors $F, G: \mathcal{C} \rightarrow \mathcal{D}$. One can define the natural transformation $S^{-1}\varphi(C): S^{-1}F(C) \rightarrow S^{-1}G(C)$ in the same manner as we have defined induced morphisms $S^{-1}T^n(f): S^{-1}T^n(C') \rightarrow S^{-1}T^n(C)$. Write $S^0{\varphi} := \varphi$ and
\[ S^{-k}\varphi := S^{-1}(S^{-k+1}\varphi): S^{-k}F \rightarrow S^{-k}G \]
for any $k \in \mathbb{N}$. By Equation~\ref{eq:thefactorisation}, connecting homomorphisms are natural transformation that factor as
\[ \delta^n: T^n \xrightarrow{\underline{\delta}^n} S^{-1}T^{n+1} \xrightarrow{\varepsilon^{-1}} T^{n+1} \, . \]

\begin{thm}\label{thm:satelliteconnmorphisms}
For a short exact sequence $0 \rightarrow A \rightarrow B \rightarrow C \rightarrow 0$ in $\mathcal{C}$ and $k \in \mathbb{N}_0$ denote by $\varepsilon^{-k}: S^{-k}T^{n+k}(C) \rightarrow S^{-k+1}T^{n+k}(A)$ the associated connecting homomorphism. Then the $n^{\text{th}}$ connecting homomorphism of the satellite functor construction is given by
\[ \widehat{\delta}^n = \varinjlim_{k \in \mathbb{N}} \varepsilon^{-k}: \widehat{T}^n(C) \rightarrow \widehat{T}^{n+1}(A) \]
where the $\varepsilon^{-k}$ are connected over the morphisms $S^{-k}\underline{\delta}^{n+k}: S^{-k}T^{n+k} \rightarrow S^{-k-1}T^{n+k+1}$ defined in Equation~\ref{eq:thefactorisation}.
\end{thm}

\begin{proof}
The proof is based on determining all terms of each morphism of the cohomological functors $\Phi_{k, k+1}^{\bullet}: T^{\bullet}\langle k \rangle \rightarrow T^{\bullet}\langle k+1 \rangle$ that give rise to the satellite functor construction. We already know that $\Phi_{k, k+1}^n = \mathrm{id}_{\mathcal{D}}$ for $k < n$. To describe $\Phi_{k, k+1}^k$, consider a short exact sequence $0 \rightarrow M \rightarrow P \rightarrow C \rightarrow 0$ with $P$ projective. We obtain from Diagram~\ref{diag:delta} the commutative diagram
\begin{equation}\label{diag:factorise}
\begin{tikzcd}
  {} \dots {} \arrow[r]
    &T^k(P) \arrow[r] \arrow[d, "T^k(h)"]
    &T^k(C) \arrow[r, "\delta^k"] \arrow[d, "\mathrm{id}_{T^k(C)}"]
    &T^{k+1}(M) \arrow[r] \arrow[d, "T^{k+1}(h^{\ast})"]
    &T^{k+1}(P) \arrow[r] \arrow[d, "T^{k+1}(h)"]
    &{} \dots {} \\
  {} \dots {} \arrow[r]
    &T^k(B) \arrow[r]
    &T^k(C) \arrow[r, "\delta^k"]
    &T^{k+1}(A) \arrow[r]
    &T^{k+1}(B) \arrow[r]
    &{} \dots {}
\end{tikzcd}
\end{equation}
Recall that $S^{-1}T^{k+1}(C)$ is a kernel and $\varepsilon^{-1}: S^{-1}T^{k+1}(C) \rightarrow T^{k+1}(M)$ its canonical monomorphism. Using Equation~\ref{eq:thefactorisation}, we can write the middle square as
\begin{center}
\begin{tikzcd}
  T^k(C) \arrow[r, "\underline{\delta}^k"] \arrow[d, "\mathrm{id}_{T^k(C)}"]
    &S^{-1}T^{k+1}(C) \arrow[r, "\varepsilon^{-1}"]
    &T^{k+1}(M) \arrow[d, "T^{k+1}(h^{\ast})"] \\
  T^k(C) \arrow[rr, "\delta^n"]
    &{}
    &T^{k+1}(A)
\end{tikzcd}
\end{center}
By definition, 
\[ \delta^{k}\langle k+1 \rangle = \varepsilon^{-1}: S^{-1}T^{k+1}(C) \xrightarrow{\varepsilon^{-1}} T^{k+1}(M) \xrightarrow{T^{k+1}(h^{\ast})} T^{k+1}(A) \, . \]
Tilting the last diagram, we obtain the desired commutative diagram
\begin{equation}\label{diag:twoconnmorphisms}
\begin{tikzcd}
  T^k(C) \arrow[r, "\delta^k"] \arrow[d, "\underline{\delta}^k"]
    &T^{k+1}(A) \arrow[d, "\mathrm{id}_{T^{k+1}(A)}"] \\
  S^{-1}T^{k+1}(C) \arrow[r, "\varepsilon^{-1}"]
    &T^{k+1}(A)
\end{tikzcd}
\end{equation}
By the proof of Lemma~\ref{lem:uniqueness} found in~\cite[p.~46--47]{car56}, $\underline{\delta}^k = \Phi_{k, k+1}^k: T^k \rightarrow S^{-1}T^{k+1}$ is a natural transformation. Thus, $\Phi_{k, k+1}^n = S^{n-k}\underline{\delta}^k: S^{n-k}T^k \rightarrow S^{n-k+1}T^{k+1}$ for any $n < k$. In particular, the square
\begin{equation}\label{diag:takesatellitetrsfs}
\begin{tikzcd}
    S^{n-k}T^k(C) \arrow[r, "\varepsilon^{n-k}"] \arrow[d, "S^{n-k}\underline{\delta}^k"] & S^{-k+1}T^n(A) \arrow[d, "S^{n-k+1}\underline{\delta}^k"] \\
    S^{n-k-1}T^{k+1}(C) \arrow[r, "\varepsilon^{n-k-1}"] & S^{n-k}T^{k+1}(A)
\end{tikzcd}
\end{equation}
commutes. In the direct limit, these squares give rise to the connecting homomorphism $\widehat{\delta}^n$ and the formula in the lemma follows from reindexing as in Equation~\ref{eq:expliciterms}.
\end{proof}

Diagram~\ref{diag:twoconnmorphisms} and the proof of Lemma~\ref{lem:uniqueness} from~\cite[p.~46--47]{car56} imply the following proposition, which is required to establish explicit formulae for the induced morphisms of the resolution construction.

\begin{prop}\label{prop:vardeltanat}
For any $n \in \mathbb{Z}$, $\underline{\delta}^n: T^n \rightarrow S^{-1}T^{n+1}$ from Equation~\ref{eq:thefactorisation} is a natural transformation.
\end{prop}

We end this section by providing an explicit formula for the induced morphisms in the satellite functor construction.

\begin{thm}\label{thm:inducedsatellitemorphs}
Let $f: A \rightarrow B$ be a morphism in $\mathcal{C}$. Then for any $n \in \mathbb{Z}$ the commuting squares
\begin{center}
\begin{tikzcd}
    S^{-k}T^{n+k}(A) \arrow[rrr, "S^{-k}T^{n+k}(f)"] \arrow[d, "S^{-k}\underline{\delta}^{n+k}"] & & & S^{-k}T^{n+k}(B) \arrow[d, "S^{-k}\underline{\delta}^{n+k}"] \\
    S^{-k-1}T^{n+k+1}(A) \arrow[rrr, "S^{-k-1}T^{n+k+1}(f)"] & & & S^{-k-1}T^{n+k+1}(B)
\end{tikzcd}
\end{center}
give rise to
\[ \widehat{T}^n(f) = \varinjlim_{\mathcal{D}, k \in \mathbb{N}_0} S^{-k}T^{n+k}(f): \widehat{T}(A) \rightarrow \widehat{T}(B) \, , \]
the induced morphism in the satellite functor construction.
\end{thm}

\begin{proof}
This follows from Equation~\ref{eq:expliciterms} and the description of the cohomological functors $\Phi_{n+k, n+k+1}^{\bullet}$ found in the proof of Theorem~\ref{thm:satelliteconnmorphisms}.
\end{proof}

\section{The resolution and the satellite functor construction}\label{sec:resols}

The main goal of this section is to demonstrate how the resolution construction gives rise to Mislin completions of cohomological functors. This is based on developing explicit formulae for the induced morphisms and connecting homomorphisms of this construction. This in turn is achieved by developing isomorphisms between the terms of the satellite functor construction and the ones of the resolution construction. As a consequence of these explicit formulae for induced morphisms and connecting homomorphisms, we demonstrate that there are countably many distinct constructions of Mislin  completions.

\subsection{Termwise isomorphisms}

First we set up all necessary notation for the treatment of the resolution construction. Let $M \in \mathrm{obj}(\mathcal{C})$, $(M_n)_{n \in \mathbb{N}_0}$ be a projective resolution of $M$ and $\widetilde{M}_k$ be the $k^{\text{th}}$ syzygy of $M_{\bullet}$ for $k \in \mathbb{N}_0$. For $n \in \mathbb{Z}$ we denote the $n^{\text{th}}$ term of the satellite functor construction by $\widehat{T}^n(M)$. If
\[ \delta^{n+k}: T^{n+k}(\widetilde{M}_k) \rightarrow T^{n+k+1}(\widetilde{M}_{k+1}) \]
are taken as in Section~\ref{sec:constructions}, we write the $n^{\text{th}}$ term of the resolution construction as
\[ T_{Res}^n(M) := \varinjlim_{k \in \mathbb{N}_0} \big(T^{n+k}(\widetilde{M}_k), \delta^{n+k}\big) \ . \]
A priori, the terms of the resolution construction depend on a choice of projective resolution, which is resolved by demonstrating

\begin{lem}\label{lem:resolisoms}
For any $n \in \mathbb{Z}$ there is an isomorphism $\omega_n(M): \widehat{T}^n(M) \rightarrow T_{Res}^n(M)$ in $\mathcal{D}$.
\end{lem}

In order to develop explicit formulae for the induced morphisms and connecting homomorphisms of the resolution construction, we require an explicit formula for such an isomorphism. A proof of this lemma can be extracted from~\cite[p.~299]{mis94} for modules over rings and can be found in~\cite[p.~20]{guo23} for relative Ext-functors. However, the resulting isomorphisms are not in an explicit form in both cases. We remedy this issue by providing our own proof the above lemma, which involves the following morphisms that are needed at the end of this section again.

\begin{defn}\label{defn:sigmafunctors}
Define $\Sigma^0{\delta^{n+k}} := \delta^{n+k}$ and
\begin{align*}
\Sigma^{-l}\delta^{n+k+l} &:= S^{-l+1}\underline{\delta}^{n+k+l} \circ \varepsilon^{-l}:  \\
&S^{-l}T^{n+k+l}(\widetilde{M}_k) \rightarrow S^{-l+1}T^{n+k+l}(\widetilde{M}_{k+1}) \rightarrow S^{-l}T^{n+k+l+1}(\widetilde{M}_{k+1})
\end{align*}
for any $l \in \mathbb{N}$ in analogy to left satellite functors.
\end{defn}

\begin{proof}[Lemma~\ref{lem:resolisoms}]
The proof consists of two steps where we introduce the relevant notation throughout its course. In the first step, we construct an isomorphism
\[ \eta_n(M): \varinjlim_{k \in \mathbb{N}} (S^{-1}T^{n+k}(\widetilde{M}_{k-1}), \Sigma^{-1}\delta^{n+k}) \rightarrow T_{Res}^n(M) \]
and in the second step we extend it to an isomorphism $\omega_n(M): \widehat{T}^n(M) \rightarrow T_{Res}^n(M)$. \\

{\textbf{Step 1}} Defining $\eta_n(M): \varinjlim_{k \in \mathbb{N}} (S^{-1}T^{n+k}(\widetilde{M}_{k-1}), \Sigma^{-1}\delta^{n+k}) \xrightarrow{\cong} T_{Res}^n(M)$ \\

For the short exact sequence $0 \rightarrow \widetilde{M}_{k+1} \rightarrow M_k \rightarrow \widetilde{M}_k \rightarrow 0$ take the morphism from Equation~\ref{eq:thefactorisation}
\[ \underline{\delta}^{n+k}: T^{n+k}(\widetilde{M}_k) \rightarrow S^{-1}T^{n+k+1}(\widetilde{M}_k) \]
and the canonical morphism from the kernel
\[ \varepsilon^{-1}: S^{-1}T^{n+k}(\widetilde{M}_{k-1}) \rightarrow T^{n+k}(\widetilde{M}_k) \, . \]
 Define the morphism
\[ \Sigma^{-1}\delta^{n+k} := \underline{\delta}^{n+k} \circ \varepsilon^{-1}: \; \; S^{-1}T^{n+k}(\widetilde{M}_{k-1}) \rightarrow T^{n+k}(\widetilde{M}_k) \rightarrow S^{-1}T^{n+k+1}(\widetilde{M}_k) \, . \]
If we denote the cokernel of $\varepsilon^{-1}$ by $q_{n+k}: T^{n+k}(\widetilde{M}_k) \rightarrow N_{n+k}$, we may use $\Sigma^{-1}\delta^{n+k}$ to extend Diagram~\ref{diag:twoconnmorphisms} commutatively as
\begin{equation}\label{diag:cokernels}
\begin{tikzcd}
  0 \arrow[r]
    & S^{-1}T^{n+k}(\widetilde{M}_{k-1}) \arrow[r, "\varepsilon^{-1}"] \arrow[d, "\Sigma^{-1}\delta^{n+k}" near start]
    & T^{n+k}(\widetilde{M}_k) \arrow[r, "q_{n+k}"] \arrow[dl, "\underline{\delta}^{n+k}" near start] \arrow[d, "\delta^{n+k}" near start]
    & N_{n+k} \arrow[r] \arrow[d, "m_{n+k}"]
    & 0 \\
  0 \arrow[r]
    & S^{-1}T^{n+k+1}(\widetilde{M}_k) \arrow[r, "\varepsilon^{-1}"]
    & T^{n+k+1}(\widetilde{M}_{k+1}) \arrow[r, "q_{n+k+1}"]
    & N_{n+k+1} \arrow[r]
    & 0
\end{tikzcd}
\end{equation}
Since
\[ m_{n+k} \circ q_{n+k} = q_{n+k+1} \circ \varepsilon^{-1} \circ \underline{\delta}^{n+k}: T^{n+k}(\widetilde{M}_k) \rightarrow N_{n+k+1} \]
is the zero morphism, we deduce that $m_{n+k} = 0$. Note that these diagrams form a direct system of short exact sequences where countable direct limits in $\mathcal{D}$ are exact. Thus, we obtain the required isomorphism
\[ \eta_n(M) = \varinjlim \varepsilon^{-1}: \varinjlim_{k \in \mathbb{N}} (S^{-1}T^{n+k}(\widetilde{M}_{k-1}), \Sigma^{-1}\delta^{n+k}) \rightarrow T_{Res}^n(M) \, . \]

{\textbf{Step 2}} Defining $\omega_n(M): \widehat{T}^n(M) \xrightarrow{\cong} T_{Res}^n(M)$ \\

We connect the terms of the direct system $(S^{-k}T^{n+k}(M), S^{-k}\underline{\delta}^{n+k})_{k \in \mathbb{N}_0}$ to the ones of $(T^{n+k}(\widetilde{M}_k), \delta^{n+k})_{k \in \mathbb{N}_0}$ to construct this isomorphism, which requires a larger commutative diagram. We can extend Diagram~\ref{diag:takesatellitetrsfs} to another commutative diagram of the form
\begin{equation}\label{diag:commtriangles}
\begin{tikzcd}
    S^{-l}T^{n+k+l}(\widetilde{M}_k) \arrow[r, "\varepsilon^{-l}"] \arrow[dr, "\Sigma^{-l}\delta^{n+k+l}"] \arrow[d, "S^{-l}\underline{\delta}^{n+k+l}" near end] & S^{-l+1}T^{n+k+l}(\widetilde{M}_{k+1}) \arrow[d, "S^{-l+1}\underline{\delta}^{n+k+l}"] \\
    S^{-l-1}T^{n+k+l+1}(\widetilde{M}_k) \arrow[r, "\varepsilon^{-l-1}"] & S^{-l}T^{n+k+l+1}(\widetilde{M}_{k+1})
\end{tikzcd}
\end{equation}
Because left satellite functors vanish on projective objects such as $M_k$, the morphism
\[ \varepsilon^{-l}: S^{-l}T^{n+k+l}(\widetilde{M}_k) \rightarrow S^{-l+1}T^{n+k+l}(\widetilde{M}_{k+1}) \]
is an isomorphism. Define the isomorphisms
\begin{align}
\widetilde{o}^k := \varepsilon^{-3} \circ {} \dots {} \circ \varepsilon^{-k+1} \circ \varepsilon^{-k}: \; \; {} &S^{-k}T^{n+k}(M) \rightarrow S^{-2}T^{n+k}(\widetilde{M}_{k-2}) \nonumber \\
\text{and} \quad o^{k} := \varepsilon^{-2} \circ \widetilde{o}^k: \; \; {} &S^{-k}T^{n+k}(M) \rightarrow S^{-1}T^{n+k}(\widetilde{M}_{k-1}) \, . \label{eq:themapsok}
\end{align}
Using the above morphisms together with Diagram~\ref{diag:twoconnmorphisms}, Diagram~\ref{diag:takesatellitetrsfs} and Diagram~\ref{diag:commtriangles}, we can write the commutative diagram
\begin{equation}\label{diag:finalisom}
\begin{tikzcd}
  {}
    & {}
    & {}
    & {\scriptstyle T^n(M)} \arrow[d, "{\scriptscriptstyle \delta^{n}}" near start] \arrow[dl, "{\scriptscriptstyle \underline{\delta}^{n}}" near start] \\
  {}
    & {}
    & {\scriptstyle S^{-1}T^{n+1}(M)} \arrow[r, "{\scriptscriptstyle \varepsilon^{-1}}"] \arrow[d, "{\scriptscriptstyle \Sigma^{-1}\delta^{n+1}}" near start] \arrow[dl, "{\scriptscriptstyle S^{-1}\underline{\delta}^{n+1}}" near start]
    & {\scriptstyle T^{n+1}(\widetilde{M}_1)} \arrow[d, "{\scriptscriptstyle \delta^{n+1}}" near start] \arrow[dl, "{\scriptscriptstyle \underline{\delta}^{n+1}}" near start] \\
  {}
    & {\scriptstyle S^{-2}T^{n+2}(M)} \arrow[r, "{\scriptscriptstyle \varepsilon^{-2}}"] \arrow[d, "{\scriptscriptstyle \Sigma^{-2}\delta^{n+2}}" near start] \arrow[dl, "{\scriptscriptstyle S^{-2}\underline{\delta}^{n+2}}" near start]
    & {\scriptstyle S^{-1}T^{n+2}(\widetilde{M}_1)} \arrow[r, "{\scriptscriptstyle \varepsilon^{-1}}"] \arrow[d, "{\scriptscriptstyle \Sigma^{-1}\delta^{n+2}}" near start] \arrow[dl, "{\scriptscriptstyle S^{-1}\underline{\delta}^{n+2}}" near start]
    & {\scriptstyle T^{n+2}(\widetilde{M}_2)} \arrow[d, "{\scriptscriptstyle \delta^{n+2}}" near start] \arrow[dl, "{\scriptscriptstyle \underline{\delta}^{n+2}}" near start] \\
  {\scriptstyle S^{-3}T^{n+3}(M)} \arrow[r, "{\scriptscriptstyle \widetilde{o}^3}"] \arrow[d, "{\scriptscriptstyle S^{-3}\underline{\delta}^{n+3}}" near start]
    & {\scriptstyle S^{-2}T^{n+3}(\widetilde{M}_1)} \arrow[r, "{\scriptscriptstyle \varepsilon^{-2}}"] \arrow[d, "{\scriptscriptstyle \Sigma^{-2}\delta^{n+3}}" near start]
    & {\scriptstyle S^{-1}T^{n+3}(\widetilde{M}_2)} \arrow[r, "{\scriptscriptstyle \varepsilon^{-1}}"] \arrow[d, "{\scriptscriptstyle \Sigma^{-1}\delta^{n+3}}" near start] \arrow[dl, "{\scriptscriptstyle S^{-1}\underline{\delta}^{n+3}}" near start]
    & {\scriptstyle T^{n+3}(\widetilde{M}_3)} \arrow[d, "{\scriptscriptstyle \delta^{n+3}}" near start] \arrow[dl, "{\scriptscriptstyle \underline{\delta}^{n+3}}" near start] \\
  {\scriptstyle S^{-4}T^{n+4}(M)} \arrow[r, "{\scriptscriptstyle \widetilde{o}^4}"] \arrow[d]
    & {\scriptstyle S^{-2}T^{n+4}(\widetilde{M}_2)} \arrow[r, "{\scriptscriptstyle \varepsilon^{-2}}"] \arrow[d]
    & {\scriptstyle S^{-1}T^{n+4}(\widetilde{M}_3)} \arrow[r, "{\scriptscriptstyle \varepsilon^{-1}}"] \arrow[d]
    & {\scriptstyle T^{n+4}(\widetilde{M}_4)} \arrow[d] \\
  {} {\scriptstyle \vdots} {} \arrow[r] \arrow[d]
    & {} {\scriptstyle \vdots} {} \arrow[r] \arrow[d]
    & {} {\scriptstyle \vdots} {} \arrow[r] \arrow[d]
    & {} {\scriptstyle \vdots} {} \arrow[d] \\
  {\scriptstyle S^{-k}T^{n+k}(M)} \arrow[r, "{\scriptscriptstyle \widetilde{o}^k}"] \arrow[rr, bend left = 15, crossing over, "{\scriptscriptstyle o^k}" near end] \arrow[d, "{\scriptscriptstyle S^{-k}\underline{\delta}^{n+k}}" near start]
    & {\scriptstyle S^{-2}T^{n+k}(\widetilde{M}_{k-2})} \arrow[r, "{\scriptscriptstyle \varepsilon^{-2}}"] \arrow[d, "{\scriptscriptstyle \Sigma^{-2}\delta^{n+k}}" near start]
    & {\scriptstyle S^{-1}T^{n+k}(\widetilde{M}_{k-1})} \arrow[r, "{\scriptscriptstyle \varepsilon^{-1}}"] \arrow[d, "{\scriptscriptstyle \Sigma^{-1}\delta^{n+k}}" near start] \arrow[dl, "{\scriptscriptstyle S^{-1}\underline{\delta}^{n+k}}" near start]
    & {\scriptstyle T^{n+k}(\widetilde{M}_k)} \arrow[d, "{\scriptscriptstyle \delta^{n+k}}" near start] \arrow[dl, "{\scriptscriptstyle \underline{\delta}^{n+k}}" near start] \\
  {\scriptstyle S^{-k-1}T^{n+k+1}(M)} \arrow[r, "{\scriptscriptstyle \widetilde{o}^{k+1}}"] \arrow[d]
    & {\scriptstyle S^{-2}T^{n+k+1}(\widetilde{M}_{k-1})} \arrow[r, "{\scriptscriptstyle \varepsilon^{-2}}"] \arrow[d]
    & {\scriptstyle S^{-1}T^{n+k+1}(\widetilde{M}_k)} \arrow[r, "{\scriptscriptstyle \varepsilon^{-1}}"] \arrow[d]
    & {\scriptstyle T^{n+k+1}(\widetilde{M}_{k+1})} \arrow[d] \\
  {} {\scriptstyle \vdots} {}
    & {} {\scriptstyle \vdots} {}
    & {} {\scriptstyle \vdots} {}
    & {} {\scriptstyle \vdots} {}
\end{tikzcd}
\end{equation}
We see that its right most column corresponds to the direct system $(T^{n+k}(\widetilde{M}_k), \delta^{n+k})_{k \in \mathbb{N}_0}$ whereas its top diagonal together with its left most column constitutes the direct system $(S^{-k}T^{n+k}(M), S^{-k}\underline{\delta}^{n+k})_{k \in \mathbb{N}_0}$. The right most squares in the diagram recover the isomorphism
\[ \eta_n(M) := \varinjlim \varepsilon^{-1}: \varinjlim_{k \in \mathbb{N}} (S^{-1}T^{n+k}(\widetilde{M}_{k-1}), \Sigma^{-1}\delta^{n+k}) \rightarrow T_{Res}^n(M) \, . \]
For $k \geq 3$, all the remaining squares yield together squares of the form
\begin{center}
\begin{tikzcd}
  S^{-k}T^{n+k}(M) \arrow[r, "o^k"] \arrow[d, "S^{-k}\underline{\delta}^{n+k}"]
    & S^{-1}T^{n+k}(\widetilde{M}_{k-1}) \arrow[d, "\Sigma^{-1}\delta^{n+k}"] \\
  S^{-k-1}T^{n+k+1}(M) \arrow[r, "o^{k+1}"]
    & S^{-1}T^{n+k+1}(\widetilde{M}_k)
\end{tikzcd}
\end{center}
As the morphisms $o^k$ are isomorphisms and countable direct limits in $\mathcal{D}$ are exact, this gives rise to the isomorphism
\[ \varinjlim o^k: \widehat{T}^n(M) \rightarrow \varinjlim_{k \in \mathbb{N}} (S^{-1}T^{n+k}(\widetilde{M}_{k-1}), \Sigma^{-1}\delta^{n+k}) \, . \]
Therefore,
\[ \omega_n(M) := (\varinjlim o^k) \circ \eta_n(M) = \varinjlim (o^k \circ \varepsilon^{-1}): \widehat{T}^n(M) \rightarrow T_{Res}^n(M) \]
is our desired isomorphism. Since the morphisms $o^k \circ \varepsilon^{-1}$ span the entire width of Diagram~\ref{diag:finalisom}, every single morphism was needed to construct this isomorphism.
\end{proof}

\subsection{Induced morphisms and connecting homomorphisms}

The subject of this subsection is to explicitly define the induced morphisms and connecting homomorphisms of the resolution construction such that it results in a Mislin completion in full generality. We start by presenting an explicit formula for the induced morphisms.

\begin{defn}\label{defn:resolindmorph}
Let $f: M \rightarrow N$ be a morphism in $\mathcal{C}$ and $M_{\bullet}$, $N_{\bullet}$ be projective resolutions of $M$, $N$. Define a sequence $(\widetilde{f}_k: \widetilde{M}_k \rightarrow \widetilde{N}_k)_{k \in \mathbb{N}_0}$ by setting $\widetilde{f}_0 := f$ and imposing that $\widetilde{f}_{k+1}$ is a lift $\widetilde{f}_k$, meaning that diagram
\begin{equation}\label{diag:complexlift}
\begin{tikzcd}
    0 \arrow[r] & \widetilde{M}_{k+1} \arrow[r] \arrow[d, "\widetilde{f}_{k+1}"] & M_k \arrow[r] \arrow[d, "f_k"] & \widetilde{M}_k \arrow[r] \arrow[d, "\widetilde{f}_k"] & 0 \\
    0 \arrow[r] & \widetilde{N}_{k+1} \arrow[r] & N_k \arrow[r] & \widetilde{N}_k \arrow[r] & 0
\end{tikzcd}
\end{equation}
commutes. Then the direct limit of the commutative diagrams
\begin{center}
\begin{tikzcd}
    T^{n+k}(\widetilde{M}_k) \arrow[rr, "T^{n+k}(\widetilde{f}_k)"] \arrow[d, "\delta^{n+k}"] & & T^{n+k}(\widetilde{N}_k) \arrow[d, "\delta^{n+k}"] \\
    T^{n+k+1}(\widetilde{M}_{k+1}) \arrow[rr, "T^{n+k+1}(\widetilde{f}_{k+1})"] & & T^{n+k+1}(\widetilde{N}_{k+1})
\end{tikzcd}
\end{center}
written as $T_{Res}^n(f) := \varinjlim_{k \in \mathbb{N}_0} T^{n+k}(\widetilde{f}_k)$ is the induced morphism in the resolution construction.
\end{defn}

Here, there is a similar issue as in the definition of the terms of the resolution construction. A priori, the morphism $T_{Res}^n(f): T_{Res}^n(M) \rightarrow T_{Res}^n(N)$ depends on choices of projective resolutions of $M$ and $N$. This is resolved by the following lemma, which forms the first step in proving that the resolution construction gives rise to Mislin completions.

\begin{lem}\label{lem:omeganatural}
The isomorphisms $\omega_n: \widehat{T}^n \rightarrow T_{Res}^n$ are natural. In other words, the square
\begin{center}
\begin{tikzcd}
    \widehat{T}^n(M) \arrow[r, "\omega_n(M)"] \arrow[d, "\widehat{T}^n(f)"] & T_{Res}^n(M) \arrow[d, "T_{Res}^n(f)"] \\
    \widehat{T}^n(N) \arrow[r, "\omega_n(N)"] & T_{Res}^n(N)
\end{tikzcd}
\end{center}
commutes for any morphism $f: M \rightarrow N$ in $\mathcal{C}$.
\end{lem}

\begin{proof}
It follows from Proposition~\ref{prop:vardeltanat} that all morphisms in Diagram~\ref{diag:finalisom} are natural transformations. Therefore, the morphism $\omega_n$ in the direct limit is natural.
\end{proof}

An explicit formula for the connecting homomorphisms of the satellite functor construction is given in the following definition.

\begin{defn}\label{defn:resolconnmorph}
Let $0 \rightarrow A \xrightarrow{f} B \xrightarrow{g} C \rightarrow 0$ be short exact sequence in $\mathcal{C}$. Taking this as the base case, construct inductively a sequence $(0 \rightarrow \widetilde{A}_k \xrightarrow{\widetilde{f}_k} \widetilde{B}_k \xrightarrow{\widetilde{g}_k} \widetilde{C}_k \rightarrow 0)_{k \in \mathbb{N}_0}$ of short exact sequences as follows. Let
\begin{equation}\label{diag:horseshoe}
\begin{tikzcd}[row sep = scriptsize]
  {}
    & 0 \arrow[d]
    & 0 \arrow[d]
    & 0 \arrow[d]
    & {} \\
  0 \arrow[r]
    & \widetilde{A}_{k+1} \arrow[r, "\widetilde{f}_{k+1}"] \arrow[d]
    & \widetilde{B}_{k+1} \arrow[r, "\widetilde{g}_{k+1}"] \arrow[d]
    & \widetilde{C}_{k+1} \arrow[r] \arrow[d]
    & 0 \\
  0 \arrow[r]
    & A_k \arrow[r, "f_k"] \arrow[d]
    & B_k \arrow[r, "g_k"] \arrow[d]
    & C_k \arrow[r] \arrow[d]
    & 0 \\
  0 \arrow[r]
    & \widetilde{A}_k \arrow[r, "\widetilde{f}_k"] \arrow[d]
    & \widetilde{B}_k \arrow[r, "\widetilde{g}_k"] \arrow[d]
    & \widetilde{C}_k \arrow[r] \arrow[d]
    & 0 \\
  {}
    & 0
    & 0
    & 0
    & {}
\end{tikzcd}
\end{equation}
be a commutative diagram as in the Horseshoe Lemma found in~\cite[p.~37]{wei94}. To wit: $A_k$, $B_k$ and $C_k$ are projective and all rows and columns are exact. By~\cite[Proposition~III.4.1]{car56}, the square
\begin{equation}\label{diag:anticomm}
\begin{tikzcd}[row sep = scriptsize]
  T^{n+k}(\widetilde{C}_k) \arrow[r, "\delta^{n+k}"] \arrow[d, "\delta^{n+k}"]
    & T^{n+k+1}(\widetilde{A}_k) \arrow[d, "\delta^{n+k+1}"] \\
  T^{n+k+1}(\widetilde{C}_{k+1}) \arrow[r, "\delta^{n+k+1}"]
    & T^{n+k+2}(\widetilde{A}_{k+1})
\end{tikzcd}
\end{equation}
anticommutes for any $n \in \mathbb{Z}$ and $k \in \mathbb{N}_0$. Then the direct limit of these anticommutatitve squares
\[ \widetilde{\delta}^n := \varinjlim_{k \in \mathbb{N}_0} (-1)^k \delta^{n+k}: T_{Res}^n(C) \rightarrow T_{Res}^{n+1}(A) \]
is the connecting homomorphism of the resolution construction.
\end{defn}

A priori, this definition depends on specific choices of projective resolutions of the terms $A$, $B$ and $C$. This issue is remedied by the following lemma, which forms the second step in proving that the resolution construction gives rise to Mislin completions.

\begin{lem}\label{lem:commwithconn}
The isomorphisms $\omega_n: \widehat{T}^n \rightarrow T_{Res}^n$ commute with connecting homomorphisms. In other words, if $\widehat{\delta}^{\bullet}$ is the connecting homomorphism from the satellite functor construction, then the diagram
\begin{center}
\begin{tikzcd}[row sep = small]
  \widehat{T}^n(C) \arrow[r, "\omega_n"] \arrow[d, "\widehat{\delta}^n"]
  & T_{Res}^n(C) \arrow[d, "\widetilde{\delta}^n"] \\
  \widehat{T}^{n+1}(A) \arrow[r, "\omega_{n+1}"]
  & T_{Res}^{n+1}(A)
\end{tikzcd}
\end{center}
commutes for any short exact sequence $0 \rightarrow A \rightarrow B \rightarrow C \rightarrow 0$ in $\mathcal{C}$ and any $n \in \mathbb{Z}$.
\end{lem}

\begin{proof}
The proof is based on the diagram found on the next page, which represents an infinite prism extending downwards indefinitely. First, we show that it commutes. The front side is a copy of Diagram~\ref{diag:finalisom} as well as the back side up to its commuting right most column. The right hand side of the diagram gives rise to $\widetilde{\delta}^n: T_{Res}^n(C) \rightarrow T_{Res}^{n+1}(A)$ in the direct limit and the top left hand side to $\widehat{\delta}^n: \widehat{T}^n(C) \rightarrow \widehat{T}^{n+1}(A)$. In particular, all squares and triangles from these sides commute. Using the cohomological functors defined in Equation~\ref{eq:directsyscohomfunctor} we see that the horizontal squares commute by~\cite[Proposition~III.4.1]{car56}. As in Diagram~\ref{diag:finalisom} we can insert commuting triangles of the form
\begin{center}
\begin{tikzcd}
  {} & S^{-l}T^{n+k}(\widetilde{M}_{k-l}) \arrow[d, "\Sigma^{-l}\delta^{n+k}"] \arrow[dl, "S^{-l}\underline{\delta}^{n+k}" near start] \\
  S^{-l-1}T^{n+k+1}(\widetilde{M}_{k-l}) \arrow[r, "\varepsilon^{-l-1}"]
  & S^{-l}T^{n+k+1}(\widetilde{M}_{k-l+1})
\end{tikzcd}
\end{center}
into the front and back side. We conclude by this and Diagram~\ref{diag:takesatellitetrsfs} that the vertical squares running from back to front also commute. In particular, Diagram~\ref{diag:largestdiag} commutes. \\

To complete the proof, we note that the right most column on the back side yields the identity morphism in the direct limit. Thus, the entire back side yields the morphism $\omega_{n+1}: \widetilde{T}^{n+1}(A) \rightarrow T_{Res}^{n+1}(A)$ while the front side yields $\omega_n: \widetilde{T}^n(A) \rightarrow T_{Res}^n(A)$. As mentioned before, the right hand side yields $\widetilde{\delta}^n: T_{Res}^n(C) \rightarrow T_{Res}^{n+1}(A)$ and left hand side $\widehat{\delta}^n: \widehat{T}^n(C) \rightarrow \widehat{T}^{n+1}(A)$. Therefore, the square in the statement of the lemma commutes.
\end{proof}

Finally, we conclude that the resolution construction $(T_{Res}^{\bullet}, \widetilde{\delta}^{\bullet})$ with the explicit formula for induced morphisms from Definition~\ref{defn:resolindmorph} and for connecting homomorphisms from Definition~\ref{defn:resolconnmorph} forms a Mislin completion of $(T^{\bullet}, \delta^{\bullet})$.

\begin{thm}\label{thm:resolsatellites}
The resolution construction $(T_{Res}^{\bullet}, \widetilde{\delta}^{\bullet})$ forms a Mislin completion of $(T^{\bullet}, \delta^{\bullet})$. More specifically, $(T_{Res}^{\bullet}, \widetilde{\delta}^{\bullet})$ forms a cohomological functor and $\omega_{\bullet}: (\widehat{T}^{\bullet}, \widehat{\delta}^{\bullet}) \rightarrow (T_{Res}^{\bullet}, \widetilde{\delta}^{\bullet})$ is an isomorphism of cohomological functors from the satellite functor construction.
\end{thm}

\begin{proof}
This follows from Theorem~\ref{thm:satellitemislin}, Lemma~\ref{lem:omeganatural} and Lemma~\ref{lem:commwithconn}.
\end{proof}\pagebreak

\begin{samepage}
\begin{center}
\rotatebox{90}{%
\begin{tikzcd}[column sep = tiny, ampersand replacement = \&]
    {} \& \& {} \& \& {} \& \& {} \& \& {\scriptscriptstyle T^{n+1}(A)} \arrow[dd, "{\scriptstyle \delta^{n+1}}" near start] \arrow[ddll, bend right = 24, "{\scriptstyle \mathrm{id}}"] \\
    \& {} \& \& {} \& \& {} \& \& {} \& \& {\scriptstyle T^n(C)} \arrow[ul, "{\scriptstyle \delta^n}" near start] \arrow[dd, "{\scriptstyle \delta^n}"] \\
    {} \& \& {} \& \& {} \& \& {\scriptstyle T^{n+1}(A)} \arrow[rr, "{\scriptstyle \delta^{n+1}}"] \arrow[dd, "{\scriptstyle \delta^{n+1}}" near start] \arrow[ddll, bend right = 24, "{\scriptstyle \underline{\delta}^{n+1}}"] \& \& {\scriptstyle T^{n+2}(\widetilde{A}_1)} \arrow[dd, "{\scriptstyle \delta^{n+2}}" near end] \\
    \& {} \& \& {} \& \& {} \& \& {\scriptstyle S^{-1}T^{n+1}(C)} \arrow[ul, "{\scriptstyle \varepsilon^{-1}}" near start] \arrow[rr, crossing over, "{\scriptstyle \varepsilon^{-1}}" near start] \arrow[from = uurr, crossing over, bend right = 24, "{\scriptstyle \underline{\delta}^n}" near start] \& \& {\scriptstyle T^{n+1}(\widetilde{C}_1)} \arrow[ul, "{\scriptstyle - \delta^{n+1}}" near start] \arrow[dd, "{\scriptstyle \delta^{n+1}}"] \\
    {} \& \& {} \& \& {\scriptstyle S^{-1}T^{n+2}(A)} \arrow[rr, "{\scriptstyle \varepsilon^{-1}}"] \arrow[dd, "{\scriptstyle \Sigma^{-1}\delta^{n+2}}" near start] \arrow[ddll, bend right = 24, "{\scriptstyle S^{-1}\underline{\delta}^{n+1}}"] \& \& {\scriptstyle T^{n+2}(\widetilde{A}_1)} \arrow[rr, "{\scriptstyle \delta^{n+2}}" near end] \arrow[dd, "{\scriptstyle \delta^{n+2}}" near end] \& \& {\scriptstyle T^{n+3}(\widetilde{A}_2)} \arrow[dd, "{\scriptstyle \delta^{n+3}}" near end] \\
    \& {} \& \& {} \& \& {\scriptstyle S^{-2}T^{n+2}(C)} \arrow[ul, "{\scriptstyle \varepsilon^{-2}}" near start] \arrow[rr, crossing over, "{\scriptstyle \varepsilon^{-2}}" near start] \arrow[from = uurr, crossing over, bend right = 24, "{\scriptstyle S^{-1}\underline{\delta}^{n+1}}" near start] \& \& {\scriptstyle S^{-1}T^{n+2}(\widetilde{C}_1)} \arrow[ul, "{\scriptstyle -\varepsilon^{-1}}" near start] \arrow[rr, crossing over, "{\scriptstyle \varepsilon^{-1}}" near start] \arrow[from = uu, crossing over, "{\scriptstyle \Sigma^{-1}\delta^{n+1}}" near end] \& \& {\scriptstyle T^{n+2}(\widetilde{C}_2)} \arrow[ul, "{\scriptstyle \delta^{n+2}}" near start] \arrow[dd, "{\scriptstyle \delta^{n+2}}"] \\
    {} \& \& {\scriptstyle S^{-2}T^{n+3}(A)} \arrow[rr, "{\scriptstyle \varepsilon^{-2}}"] \arrow[dd, "{\scriptstyle \Sigma^{-2}\delta^{n+3}}" near start] \arrow[ddll, bend right = 24, "{\scriptstyle S^{-2}\underline{\delta}^{n+3}}"] \& \& {\scriptstyle S^{-1}T^{n+3}(\widetilde{A}_1)} \arrow[rr, "{\scriptstyle \varepsilon^{-1}}" near end] \arrow[dd, "{\scriptstyle \Sigma^{-1}\delta^{n+3}}" near end] \& \& {\scriptstyle T^{n+3}(\widetilde{A}_2)} \arrow[rr, "{\scriptstyle \delta^{n+3}}" near end] \arrow[dd, "{\scriptstyle \delta^{n+3}}" near end] \& \& {\scriptstyle T^{n+4}(\widetilde{A}_3)} \arrow[dd, "{\scriptstyle \delta^{n+4}}" near end] \\
    \& {} \& \& {\scriptstyle S^{-3}T^{n+3}(C)} \arrow[ul, "{\scriptstyle \varepsilon^{-3}}" near start] \arrow[rr, crossing over, "{\scriptstyle \varepsilon^{-3} = \tilde{o}^3}" near start] \arrow[from = uurr, crossing over, bend right = 24, "{\scriptstyle S^{-2}\underline{\delta}^{n+2}}" near start] \arrow[dd, "{\scriptstyle \Sigma^{-3}\delta^{n+3}}" near start] \arrow[ddll, crossing over, bend right = 24, "{\scriptstyle S^{-2}\underline{\delta}^{n+3}}" near start] \& \& {\scriptstyle S^{-2}T^{n+3}(\widetilde{C}_1)} \arrow[ul, "{\scriptstyle -\varepsilon^{-2}}" near start] \arrow[rr, crossing over, "{\scriptstyle \varepsilon^{-2}}" near start] \arrow[from = uu, crossing over, "{\scriptstyle \Sigma^{-2}\delta^{n+2}}" near end] \arrow[dd, "{\scriptstyle \Sigma^{-2}\delta^{n+3}}" near end] \& \& {\scriptstyle S^{-1}T^{n+3}(\widetilde{C}_2)} \arrow[ul, "{\scriptstyle \varepsilon^{-1}}" near start] \arrow[rr, crossing over, "{\scriptstyle \varepsilon^{-1}}" near start] \arrow[from = uu, crossing over, "{\scriptstyle \Sigma^{-1}\delta^{n+2}}" near end] \arrow[dd, "{\scriptstyle \Sigma^{-1}\delta^{n+3}}" near end] \& \& {\scriptstyle T^{n+3}(\widetilde{C}_3)} \arrow[ul, "{\scriptstyle - \delta^{n+3}}" near start] \arrow[dd, "{\scriptstyle \delta^{n+3}}"] \\
    {} {\scriptstyle \dots} {} \& \& {} {\scriptstyle \dots} {} \& \& {} {\scriptstyle \dots} {}\& \& {} {\scriptstyle \dots} {}\& \& {} {\scriptstyle \dots} {}\& \\
    \& {} {\scriptstyle \dots} {}\& \& {} {\scriptstyle \dots} {}\& \& {} {\scriptstyle \dots} {}\& \& {} {\scriptstyle \dots} {}\& \& {} {\scriptstyle \dots} {} \\ 
\end{tikzcd}
}
\end{center}
{\vspace{-120mm}
\begin{equation}\label{diag:largestdiag}
{} \quad {}
\end{equation}
${} \quad {}$ \vspace{105mm}}
\end{samepage}

\subsection{Countably many distinct constructions of Mislin completions}

We demonstrate in this final subsection that there are countably many distinct constructions of Mislin completions, where the term ``distinct'' will be defined in due course. First, we present a collection of constructions of (proposed) Mislin completions whose terms we define as follows.

\begin{defn}
Let $a = (a_k)_{k \in \mathbb{N}} \in \lbrace 0, 1 \rbrace^{\mathbb{N}}$ be a sequence taking values in $0$ and $1$. Define
\[ (\rho(a)_k)_{k \in \mathbb{N}_0} := \begin{cases} 0 &\text{if } k = 0 \\ \sum_{i = 1}^k a_i &\text{if } k \in \mathbb{N} \end{cases} \quad \text{and} \qquad (\sigma(a)_k)_{k \in \mathbb{N}_0} := \begin{cases} 0 &\text{if } k = 0 \\ \sum_{i = 1}^k (1-a_i) &\text{if } k \in \mathbb{N} \end{cases} \]
For any $M \in \mathrm{obj}(\mathcal{C})$ and $n \in \mathbb{Z}$ form the direct system
\[ \big(S^{-\rho(a)_k}T^{n+k}(\widetilde{M}_{\sigma(a)_k}), \delta_a^{n+k} \big)_{k \in \mathbb{N}_0} \]
where
\[ \delta_a^{n+k} := \begin{cases} {\scriptstyle \Sigma^{-\rho(a)_k}\delta^{n+k}: \: S^{-\rho(a)_k}T^{n+k}(\widetilde{M}_{\sigma(a)_k}) \rightarrow S^{-\rho(a)_k}T^{n+k+1}(\widetilde{M}_{\sigma(a)_k+1})} &{\scriptstyle \text{if } \rho(a)_{k+1} = \rho(a)_k} \\ {\scriptstyle S^{-\rho(a)_k}\underline{\delta}^{n+k}: \: S^{-\rho(a)_k}T^{n+k}(\widetilde{M}_{\sigma(a)_k}) \rightarrow S^{-\rho(a)_k-1}T^{n+k+1}(\widetilde{M}_{\sigma(a)_k})} &{\scriptstyle \text{if } \rho(a)_{k+1} = \rho(a)_k+1} \end{cases} \]
with the morphisms $\Sigma^{-\rho(a)_k}\delta^{n+k}$ taken as in Definition~\ref{defn:sigmafunctors}. Write
\[ T_a^n(M) := \varinjlim_{\mathcal{D}, k \in \mathbb{N}_0} S^{-\rho(a)_k}T^{n+k}(\widetilde{M}_{\sigma(a)_k}) \]
for the $n^{\text{th}}$ term of the construction $T_a^{\bullet}$.
\end{defn}

As a first step in showing that these constructions form Mislin completions, we prove

\begin{lem}\label{lem:oneconstrofcompletion}
There exists an isomorphism $\omega_{a, n}: T_a^n(M) \rightarrow T_{Res}^n(M)$ for any $M \in \mathrm{obj}(\mathcal{C})$.
\end{lem}

\begin{proof}
For any $k \in \mathbb{N}_0$ we define the morphism
\begin{equation}\label{eq:oa}
\widehat{o}_a^k := \begin{cases} \varepsilon^{-1} \circ {} \dots {} \circ \varepsilon^{-\rho(a)_k}: S^{-\rho(a)_k}T^{n+k}(\widetilde{M}_{\sigma(a)_k}) \rightarrow T^{n+k}(\widetilde{M}_k) &\text{if } \rho(a)_k \geq 1 \\ \qquad \qquad \qquad \quad \mathrm{id}: T^{n+k}(\widetilde{M}_k) \rightarrow T^{n+k}(\widetilde{M}_k) &\text{if } \rho(a)_k = 0 \end{cases}
\end{equation}
similarly to the morphism $o^k$ from Equation~\ref{eq:themapsok}. By Diagram~\ref{diag:finalisom} we infer that the square
\begin{center}
\begin{tikzcd}
    S^{-\rho(a)_k}T^{n+k}(\widetilde{M}_{\sigma(a)_k}) \arrow[r, "\widehat{o}_a^k"] \arrow[d, "\delta_a^{n+k}"] & T^{n+k}(\widetilde{M}_k) \arrow[d, "\delta^{n+k}"] \\
    S^{-\rho(a)_{k+1}}T^{n+k+1}(\widetilde{M}_{\sigma(a)_{k+1}}) \arrow[r, "\widehat{o}_a^{k+1}"] & T^{n+k+1}(\widetilde{M}_{k+1})
\end{tikzcd}
\end{center}
commutes. We conclude as in the proof of Lemma~\ref{lem:resolisoms} that $\omega_{a, n} := \varinjlim_{k \in \mathbb{N}_0} \widehat{o}_a^k$ is an isomorphism.
\end{proof}

As another step in proving that the constructions $T_a^{\bullet}$ form Mislin completions, we provide explicit formulae for their induced morphisms.

\begin{defn}\label{defn:countmanyconstrsindmorph}
Let $f: M \rightarrow N$ be a morphism in $\mathcal{C}$ and $(\widetilde{f}_k: \widetilde{M}_k \rightarrow \widetilde{N}_k)_{k \in \mathbb{N}_0}$ be a sequence of morphism as in Definition~\ref{defn:resolindmorph}. Then the direct limit of the commutative diagrams
\begin{center}
\begin{tikzcd}
    S^{-\rho(a)_k}T^{n+k}(\widetilde{M}_{\sigma(a)_k}) \arrow[rrrr, "S^{-\rho(a)_k}T^{n+k}(\widetilde{f}_{\sigma(a)_k})"] \arrow[d, "\delta_a^{n+k}"] & & & & S^{-\rho(a)_k}T^{n+k}(\widetilde{N}_{\sigma(a)_k}) \arrow[d, "\delta_a^{n+k}"] \\
    S^{-\rho(a)_{k+1}}T^{n+k+1}(\widetilde{M}_{\sigma(a)_{k+1}}) \arrow[rrrr, "S^{-\rho(a)_{k+1}}T^{n+k+1}(\widetilde{f}_{\sigma(a)_{k+1}})"] & & & & S^{\rho(a)_{k+1}}T^{n+k+1}(\widetilde{N}_{\sigma(a)_{k+1}})
\end{tikzcd}
\end{center}
written as
\[ T_a^n(f) := \varinjlim_{k \in \mathbb{N}_0} S^{-\rho(a)_k}T^{n+k}(\widetilde{f}_{\sigma(a)_k}) \]
is the induced morphism of the construction $T_a^{\bullet}$.
\end{defn}

We can reiterate the proof of Lemma~\ref{lem:omeganatural} to demonstrate

\begin{lem}\label{lem:omegaanatural}
The isomorphisms $\omega_{a, n}: \widehat{T}^n \rightarrow T_{Res}^n$ are natural. In other words, the square
\begin{center}
\begin{tikzcd}
    T_a^n(M) \arrow[rr, "\omega_{a, n}(M)"] \arrow[d, "T_a^n(f)"] & & T_{Res}^n(M) \arrow[d, "T_{Res}^n(f)"] \\
    T_a^n(N) \arrow[rr, "\omega_{a, n}(N)"] & & T_{Res}^n(N)
\end{tikzcd}
\end{center}
commutes for any morphism $f: M \rightarrow N$ in $\mathcal{C}$.
\end{lem}

As a next step in proving that the constructions $T_a^{\bullet}$ form Mislin completions, we provide explicit formulae for their connecting homomorphisms.

\begin{lem}\label{lem:countmanyconstrsconnmorph}
Let $0 \rightarrow A \rightarrow B \rightarrow C \rightarrow 0$ be a short exact sequence in $\mathcal{C}$. Define for any $k \in \mathbb{N}_0$ the morphism
\begin{align*}
\delta_{a, k}^{n+k} &:= (-1)^{-\rho(a)_k-1} \varepsilon^{-\rho(a)_k-1} \circ S^{-\rho(a)_k}\underline{\delta}^{n+k}: \\
&S^{-\rho(a)_k}T^{n+k}(\widetilde{C}_{\sigma(a)_k}) \rightarrow S^{-\rho(a)_k-1}T^{n+k+1}(\widetilde{C}_{\sigma(a)_k}) \rightarrow S^{-\rho(a)_k}T^{(n+1)+k}(\widetilde{A}_{\sigma(a)_k}) \, .
\end{align*}
Then the diagram
\begin{equation}\label{diag:levelconnmorphone}
\begin{tikzcd}
    S^{-\rho(a)_k}T^{n+k}(\widetilde{C}_{\sigma(a)_k}) \arrow[r, "\delta_{a, k}^{n+k}"] \arrow[d, "\delta_a^{n+k}"] & S^{-\rho(a)_k}T^{n+k+1}(\widetilde{A}_{\sigma(a)_k}) \arrow[d, "\delta_a^{n+k+1}"] \\
    S^{-\rho(a)_{k+1}}T^{n+k+1}(\widetilde{A}_{\sigma(a)_{k+1}}) \arrow[r, "\delta_{a, k+1}^{n+k+1}"] & S^{-\rho(a)_{k+1}}T^{n+k+2}(\widetilde{A}_{\sigma(a)_{k+1}})
\end{tikzcd}
\end{equation}
commutes for every $k \in \mathbb{N}_0$. The resulting direct limit
\[ \widetilde{\delta}_a^n := \varinjlim_{k \in \mathbb{N}_0} \delta_{a, k}^{n+k}: T_a^n(C) \rightarrow T_a^{n+1}(A) \]
is the connecting homomorphism of the construction $T_a^{\bullet}$.
\end{lem}

\begin{proof}
We need to prove that Diagram~\ref{diag:levelconnmorphone} commutes. On the one hand, the morphism
\[ \delta_{a, k}^{n+k}: S^{-\rho(a)_k}T^{n+k}(\widetilde{C}_{\sigma(a)_k}) \rightarrow S^{-\rho(a)_k}T^{(n+1)+k}(\widetilde{A}_{\sigma(a)_k}) \]
takes us first down the front side of Diagram~\ref{diag:largestdiag} and then over to its back side. On the other hand, the morphism
\[ \delta_a^{n+k}: S^{-\rho(a)_k}T^{n+k}(\widetilde{C}_{\sigma(a)_k}) \rightarrow S^{-\rho(a)_{k+1}}T^{n+k+1}(\widetilde{A}_{\sigma(a)_{k+1}}) \]
only takes us down in Diagram~\ref{diag:largestdiag}. The desired diagram commutes because Diagram~\ref{diag:largestdiag} does so too.
\end{proof}

The following is an analogue of Lemma~\ref{lem:commwithconn} for the constructions $(T_a^{\bullet}, \widetilde{\delta}_a^{\bullet})$.

\begin{lem}\label{lem:commwithconna}
The isomorphisms $\omega_n: \widehat{T}^n \rightarrow T_{Res}^n$ commute with connecting homomorphisms. 
In other words, if $\widehat{\delta}^{\bullet}$ is the connecting homomorphism from the satellite functor construction, then the diagram
\begin{equation}\label{diag:limitconnmorph}
\begin{tikzcd}[row sep = small]
  T_a^n(C) \arrow[r, "\omega_{a, n}"] \arrow[d, "\widetilde{\delta}_a^n"]
  & T_{Res}^n(C) \arrow[d, "\widetilde{\delta}^n"] \\
  T_a^{n+1}(A) \arrow[r, "\omega_{a, n+1}"]
  & T_{Res}^{n+1}(A)
\end{tikzcd}
\end{equation}
commutes for any short exact sequence $0 \rightarrow A \rightarrow B \rightarrow C \rightarrow 0$ in $\mathcal{C}$ and any $n \in \mathbb{Z}$.
\end{lem}

\begin{proof}
Taking the morphisms $\widehat{o}_a^k$ from Equation~\ref{eq:oa}, it suffices to show that the diagrams
\begin{equation}\label{diag:levelconnmorphtwo}
\begin{tikzcd}
    S^{-\rho(a)_k}T^{n+k}(\widetilde{C}_{\sigma(a)_k}) \arrow[r, "\widehat{o}_a^k"] \arrow[d, "\delta_a^{n+k}"] & T^{n+k}(\widetilde{C}_k) \arrow[d, "\delta^{n+k}"] \\
    S^{-\rho(a)_k}T^{(n+1)+k}(\widetilde{A}_{\sigma(a)_k}) \arrow[r, "\widehat{o}_a^{k+1}"] & T^{(n+1)+k}(\widetilde{A}_k)
\end{tikzcd}
\end{equation}
commutes for all $k \in \mathbb{N}_0$ as we obtain Diagram~\ref{diag:limitconnmorph} in their resulting direct limit. The proof is thus based on tracing the appropriate morphisms in Diagram~\ref{diag:largestdiag}. The term $S^{-\rho(a)_k}T^{n+k}(\widetilde{C}_{\sigma(a)_k})$ found in Diagram~\ref{diag:levelconnmorphtwo} lies in the front side of Diagram~\ref{diag:largestdiag}. On the one hand, we reach the very right hand side of the latter diagram through the morphism
\[ \widehat{o}_a^k: S^{-\rho(a)_k}T^{n+k}(\widetilde{C}_{\sigma(a)_k}) \rightarrow T^{n+k}(\widetilde{C}_k) \, . \]
The morphism
\[ \delta^{n+k}: T^{n+k}(\widetilde{C}_k) \rightarrow T^{(n+1)+k}(\widetilde{A}_k) \]
leads us to the back side. We move by the identity morphism of $T^{(n+1)+k}(\widetilde{A}_k)$ across a square in the very right hand column of the back side. On the other had, the morphism
\[ \delta_{a, k}^{n+k}: S^{-\rho(a)_k}T^{n+k}(\widetilde{C}_{\sigma(a)_k}) \rightarrow S^{-\rho(a)_k}T^{(n+1)+k}(\widetilde{A}_{\sigma(a)_k}) \]
takes us first down the front side of Diagram~\ref{diag:largestdiag} and then over to the back side. Lastly, we reach the left hand side of the right most column of the back side through the morphism
\[ \widehat{o}_a^{k+1}: S^{-\rho(a)_k}T^{(n+1)+k}(\widetilde{A}_{\sigma(a)_k}) \rightarrow T^{(n+1)+k}(\widetilde{A}_k) \, . \]
Since Diagram~\ref{diag:largestdiag} commutes, we conclude that Diagram~\ref{diag:levelconnmorphtwo} also commutes.
\end{proof}

Finally, we conclude that the constructions $(T_a^{\bullet}, \widetilde{\delta}_a^{\bullet})$ with the explicit formula for induced morphisms from Definition~\ref{defn:countmanyconstrsindmorph} and for connecting homomorphisms from Lemma~\ref{lem:countmanyconstrsconnmorph} forms Mislin completions of $(T^{\bullet}, \delta^{\bullet})$.

\begin{lem}
The constructions $(T_a^{\bullet}, \widetilde{\delta}_a^{\bullet})$ form Mislin completions of $(T^{\bullet}, \delta^{\bullet})$. More specifically, $(T_a^{\bullet}, \widetilde{\delta}_a^{\bullet})$  form cohomological functors and $\omega_{a, \bullet}: (T_a^{\bullet}, \widetilde{\delta}_a^{\bullet}) \rightarrow (T_{Res}^{\bullet}, \widetilde{\delta}^{\bullet})$ are isomorphisms of cohomological functors to the resolution construction.
\end{lem}

\begin{proof}
This follows from Theorem~\ref{thm:resolsatellites}, Lemma~\ref{lem:omegaanatural} and Lemma~\ref{lem:commwithconna}.
\end{proof}

The following proposition shows that the satellite functor construction and the resolution construction are special cases of the Mislin completions $(T_a^{\bullet}, \widetilde{\delta}_a^{\bullet})$.

\begin{prop}
Let $e \in \lbrace 0, 1 \rbrace^{\mathbb{N}}$ denote the sequence all whose terms are set to $0$ and $f \in \lbrace 0, 1 \rbrace^{\mathbb{N}}$ the sequence all whose terms are set to $1$. Then $(\widehat{T}^{\bullet}, \widehat{\delta}^{\bullet}) = (T_f^{\bullet}, \widetilde{\delta}_f^{\bullet})$ and $(T_{Res}^{\bullet}, \widetilde{\delta}^{\bullet}) = (T_e^{\bullet}, \widetilde{\delta}_e^{\bullet})$.
\end{prop}

\begin{proof}
By definition, $\widehat{T}^n = T_f^n$ and $T_{Res}^n = T_e^n$ for every $n \in \mathbb{Z}$. It can be deduced from Diagram~\ref{diag:largestdiag} that $\widetilde{\delta}_f^n = \widehat{\delta}^n$
and that $\widetilde{\delta}_e^n = \widetilde{\delta}^n$.
\end{proof}

We define when two constructions $(T_a^{\bullet}, \widetilde{\delta}_a^{\bullet})$, $(T_b^{\bullet}, \widetilde{\delta}_b^{\bullet})$ are distinct, and prove that there are countably many distinct constructions of Mislin completions.

\begin{defn}\label{defn:distinctconstrs}
For $a, b \in \lbrace 0, 1 \rbrace^{\mathbb{N}}$ the constructions of the Mislin completions $(T_a^{\bullet}, \widetilde{\delta}_a^{\bullet})$ and $(T_b^{\bullet}, \widetilde{\delta}_b^{\bullet})$ are said to agree if there is an infinite (cofinal) subset $\Lambda \subseteq \mathbb{N}$ such that
\[ S^{-\rho(a)_{\lambda}}T^{n+{\lambda}}(\widetilde{M}_{\sigma(a)_{\lambda}}) = S^{-\rho(b)_{\lambda}}T^{n+{\lambda}}(\widetilde{M}_{\sigma(b)_{\lambda}}) \]
for every $\lambda \in \Lambda$, $n \in \mathbb{Z}$ and $M \in \mathrm{obj}(\mathcal{C})$. If this is not the case, the constructions $(T_a^{\bullet}, \widetilde{\delta}_a^{\bullet})$ and $(T_b^{\bullet}, \widetilde{\delta}_b^{\bullet})$ are said to be distinct. 
\end{defn}

Here, the basic idea is that the constructions agree whenever their terms
\[ T_a^n(M) = \varinjlim_{k \in \mathbb{N}_0} S^{-\rho(a)_k}T^{n+k}(\widetilde{M}_{\sigma(a)_k}) \text{ and } T_b^n(M) = \varinjlim_{k \in \mathbb{N}_0} S^{-\rho(b)_k}T^{n+k}(\widetilde{M}_{\sigma(b)_k}) \]
are isomorphic in a straightforward manner. If this is not the case, they are distinct.

\begin{thm}\label{thm:countmanyconstrs}
There are countably many distinct constructions of Mislin completions of the form $(T_a^{\bullet}, \widetilde{\delta}_a^{\bullet})$.
\end{thm}

\begin{proof}
Define the binary relation `$\sim$' on $\lbrace 0, 1 \rbrace^{\mathbb{N}}$ by setting $a \sim b$ if for every $k \in \mathbb{N}$ there is $K \geq k$ such that $\rho(a)_K = \rho(b)_K$. If $\sim'$ denotes the transitive closure of $\sim$, then the proof reduces to showing that there are countably many elements in $\lbrace 0, 1 \rbrace^{\mathbb{N}}/\sim'$. Hence, define for any $m \in \mathbb{N}$ the sequences $e^m, f^m \in \lbrace 0, 1 \rbrace^{\mathbb{N}}$ by
\[ e_k^m := \begin{cases} 1 &\text{if } k < m \\ 0 &\text{if } k \geq m \end{cases} \qquad \text{and} \qquad f_k^m := \begin{cases} 0 &\text{if } k < m \\ 1 &\text{if } k \geq m \end{cases} \]
For any distinct $m, p \in \mathbb{N}$ we note that $e_m \nsim' e_p$, $e_m \nsim' f_p$ and $f_m \nsim' f_p$. In particular, there are at least countably many elements in $\lbrace 0, 1 \rbrace^{\mathbb{N}}/\sim'$. If $a \in \lbrace 0, 1 \rbrace^{\mathbb{N}}$ takes either the value $0$ or the value $1$ finitely many times, then there is $m \in \mathbb{N}$ such that $a \sim e_m$ or $a \sim f_m$. If $b, c \in \lbrace 0, 1 \rbrace^{\mathbb{N}}$ each take the values $0$ and $1$ infinitely often, one can construct $d \in \lbrace 0, 1 \rbrace^{\mathbb{N}}$ such that $b \sim d$ and $c \sim d$. Therefore, $\lbrace 0, 1 \rbrace^{\mathbb{N}}/\sim'$ is countable.
\end{proof}

\begin{notn}
In the rest of the paper, we write $(\widehat{T}^{\bullet}, \widehat{\delta}^{\bullet})$ for the Mislin completion of a cohomological functor $(T^{\bullet}, \delta^{\bullet})$ independently from which construction it arises. When required, we specify the construction or use a special notation for it.
\end{notn}

\section{The naïve and the resolution construction}\label{sec:naiveconstruction}

The main goal of this section is to demonstrate that the naïve construction gives rise to Mislin completions Ext-functors, meaning to completed Ext-functors. As already explained at the end of Section~\ref{sec:constructions}, the naïve constructions does not define Mislin completions for other cohomological functors. To this end, we construct explicit isomorphisms from the terms of the resolution construction to those of the naïve construction. At the heart of this section is the development of explicit formulae for the induced homomorphisms and connecting homomorphisms of the naïve construction. These explicit formulae then turn the above-mentioned isomorphisms to isomorphisms of cohomological functors.  \\

We use the naïve construction from Definition~\ref{defn:naive}, which represents a faithful (absolute homological) generalisation of Benson and Carlson's original construction found in~\cite[p.~109]{ben92}. Both Beligiannis and Reiten in~\cite[p.~37--38/163--165]{bel07} as well as S.\! Guo and L.\! Liang in~\cite[p.~14--15]{guo23} have established relative homological versions of the naïve construction in great generality. The main reason to use a faithful generalisation instead, is to develop explicit formulae for the induced homomorphisms and connecting homomorphisms of the resulting completed Ext-functors. Moreover, this allows us to prove in~\cite{ghe24} that the complete cohomology group $\widehat{H}^0(G, -)$ detects finite cohomological dimension and that Yoneda products exist for completed Ext-functors $\widehat{\mathrm{Ext}}_{\mathcal{C}}^{\bullet}(-, -)$. \\

As the basis to prove that the naïve construction defines completed Ext-functors, we construct a homomorphism from the terms $\widehat{\mathrm{Ext}}_{\mathcal{C}}^n(A, B)$ arising from the resolution construction to those of the naïve construction $BC_{\mathcal{C}}^n(A, B)$ from Definition~\ref{defn:naive}. The idea of the construction and of the proof that it yields an isomorphism can be traced back to Mislin's paper~\cite[p.~299]{mis94}. Again, this differs from S.\! Guo and L.\! Liang's account~\cite[pp.~18--19]{guo23} containing a similar construction and proof. Recall that the terms $BC_{\mathcal{C}}^n(A, B)$ are defined by a bifunctor $[-, -]_{\mathcal{C}}: \mathcal{C}^{\mathrm{op}} \times \mathcal{C} \rightarrow \mathbf{Ab}$ from Equation~\ref{eq:stablehom} and by homomorphisms of the form
\[ t_{\widetilde{A}_{n+k}, \widetilde{B}_k}: [\widetilde{A}_{n+k}, \widetilde{B}_k]_{\mathcal{C}} \rightarrow [\widetilde{A}_{n+k+1}, \widetilde{B}_{k+1}]_{\mathcal{C}} \]
from Equation~\ref{eq:transition}. The first step in constructing the desired homomorphism is the following definition.

\begin{defn}\label{defn:betaprelim}
Let $A$ and $B$ be objects in $\mathcal{C}$ and denote by $(A_n)_{n \in \mathbb{N}_0}$, $(B_n)_{n \in \mathbb{N}_0}$ projective resolutions. Any morphism in the $A_{\bullet}$ can be factorised as $\partial_n: A_n \xrightarrow{\pi_n} \widetilde{A}_n \xrightarrow{\iota_n} A_{n-1}$. Then for every $f \in \mathrm{Ker}(\mathrm{Hom}_{\mathcal{C}}(\partial_{n+1}, B))$ there is a unique morphism $f': \widetilde{A}_n \rightarrow B$ such that $f = f' \circ \pi_n$. Therefore, there is an isomorphisms of abelian groups
\[ \alpha_n(B): \mathrm{Ker}(\mathrm{Hom}_{\mathcal{C}}(\partial_{n+1}, B)) \rightarrow \mathrm{Hom}_{\mathcal{C}}(\widetilde{A}_n, B),\; \; f \mapsto f' \, . \]
In the case where $n = 0$ this yields an isomorphism
\[ \alpha_0(B): \mathrm{Ext}_{\mathcal{C}}^0(A, B) \rightarrow \mathrm{Hom}_{\mathcal{C}}(A, B) \]
since $\mathrm{Ker}(\mathrm{Hom}_{\mathcal{C}}(\partial_0, B)) = \mathrm{Ext}_{\mathcal{C}}^0(A, B)$. If $n \geq 1$ and $f \in \mathrm{Im}(\mathrm{Hom}_{\mathcal{C}}(\partial_n, B))$, then $\alpha_n(B)(f) \in \mathcal{P}_R(\widetilde{A}_n, B)$. Thus $\alpha_n(B)$ descends to a homomorphism of abelian groups
\begin{align*}
\beta_n(B)\! : \, \mathrm{Ext}_{\mathcal{C}}^n(A, B) &\rightarrow [\widetilde{A}_n, B]_{\mathcal{C}} \\
f + \mathrm{Im}(\mathrm{Hom}_{\mathcal{C}}(\partial_n, B)) &\mapsto \alpha_n(B)(f)+ \mathcal{P}_R(\widetilde{A}_n, B) \, .    
\end{align*}
Note that $\beta^n(B)$ is only defined for $n \geq 1$.
\end{defn}

The following lemma constructs the desired homomorphisms and asserts that they are isomorphisms of abelian groups. It can be proved by generalising Mislin's Theorem~4.1 and its proof from~\cite{mis94}.

\begin{lem}\label{lem:betisomorphism}
Let $A$, $B$ be objects in $\mathcal{C}$. Then for every $n \in \mathbb{Z}$ and $k \in \mathbb{N}_0$ with $n+k \geq 1$ the square
\begin{equation}\label{diag:inducedbeta}
\begin{tikzcd}
  \mathrm{Ext}_{\mathcal{C}}^{n+k}(A, \widetilde{B}_k) \arrow[rr, "\beta_{n+k}(\widetilde{B}_k)"] \arrow[d, "\delta^{n+k}"]
    &{}
    &{[}\widetilde{A}_{n+k}{,} \, \widetilde{B}_k{]}_{\mathcal{C}} \arrow[d, "t_{\widetilde{A}_{n+k}{,} \, \widetilde{B}_k}"] \\
  \mathrm{Ext}_{\mathcal{C}}^{n+k+1}(A, \widetilde{B}_{k+1}) \arrow[rr, "\beta_{n+k+1}(\widetilde{B}_{k+1})"]
    &{}
    &{[}\widetilde{A}_{n+k+1}{,} \, \widetilde{B}_{k+1}{]}_{\mathcal{C}}
\end{tikzcd}
\end{equation}
commutes. The homomorphism resulting in the direct limit of the above diagrams
\[ \beta^n(B) := \varinjlim_{k \in \mathbb{N}_0, n+k \geq 1} \beta_{n+k}(\widetilde{B}_k): \widehat{\mathrm{Ext}}_{\mathcal{C}}^n(A, B) \rightarrow BC_{\mathcal{C}}^n(A, B) \]
is an isomorphism of abelian groups.
\end{lem}

Following~\cite{mis94}, we only demonstrate why Diagram~\ref{diag:inducedbeta} commutes. This is relevant for Lemma~\ref{lem:ordinaryhypercohom} by which we establish an isomorphism from the terms of the resolution construction to those of the hypercohomology construction. \\

Let $f + \mathrm{Im}(\mathrm{Hom}_{\mathcal{C}}(\partial_{n+k}, \widetilde{A}_k)$ be an element in $\mathrm{Ext}_{\mathcal{C}}^{n+k}(A, \widetilde{B}_k)$. Then we construct the diagram
\begin{equation}\label{diag:mislinbenson}
\begin{tikzcd}
  A_{n+k+1} \arrow[r, "\pi_{n+k+1}"] \arrow[rd, "F"]
    &\widetilde{A}_{n+k+1} \arrow[r, "\iota_{n+k+1}"] \arrow[d, "\overline{f}^{\ast}"]
    &A_{n+k} \arrow[rr, "\pi_{n+k}"] \arrow[rrd, "f"] \arrow[d, "\overline{f}"]
    & &\widetilde{A}_{n+k} \arrow[d, "\alpha_{n+k}(\widetilde{B}_k)(f)"] \\
  {}
    &\widetilde{B}_{k+1} \arrow[r, "\iota_{k+1}"]
    &B_k \arrow[rr, "\pi_k"]
    & &\widetilde{B}_k
\end{tikzcd}
\end{equation}
as follows. Lift $\alpha_{n+k}(\widetilde{B}_k)(f)$ to $\overline{f}^{\ast}$ as in Diagram~\ref{diag:complexlift}. In particular, the morphism $\overline{f}^{\ast}$ is a representative of
\[ t_{\widetilde{A}_{n+k}, \widetilde{B}_k} \circ \beta_{n+k}(\widetilde{B}_k)(f + \mathrm{Im}(\mathrm{Hom}_{\mathcal{C}}(\partial_{n+k}, \widetilde{B}_k))) \, . \]
If we write $F$ for $\overline{f}^{\ast} \circ \pi_{n+k+1}$, then
\[ \delta^{n+k} \big(f + \mathrm{Im}(\mathrm{Hom}_{\mathcal{C}}(\partial_{n+k}, \widetilde{B}_k)) \big) = F + \mathrm{Im}(\mathrm{Hom}_{\mathcal{C}}(\partial_{n+k+1}, \widetilde{B}_{k+1})) \in \mathrm{Ext}_{\mathcal{C}}^{n+k+1}(A, \widetilde{B}_{k+1}) \]
by the construction of the connecting homomorphism as in~\cite[pp.~11--12]{wei94}. Thus, $\overline{f}^{\ast}$ is also a representative of
\[ \beta_{n+k+1}(\widetilde{B}_{k+1}) \circ \delta^{n+k}(f + \mathrm{Im}(\mathrm{Hom}_{\mathcal{C}}(\partial_{n+k}, \widetilde{B}_k))) \]
and Diagram~\ref{diag:inducedbeta} commutes.

\subsection{Induced homomorphisms and connecting homomorphisms}

To conclude that the naïve construction gives rise to completed Ext-functors, an explicit formula for its induced homomorphisms is needed. This in turn necessitates two preliminary results of whom the first is straightforward and relevant to the construction of Yoneda products in~\cite[Theorem~6.6]{ghe24}.

\begin{prop}\label{prop:biadditive}
The bifunctor $\circ: \mathrm{Hom}_{\mathcal{C}}(B, C) \times \mathrm{Hom}_{\mathcal{C}}(A, B) \rightarrow \mathrm{Hom}_{\mathcal{C}}(A, C)$ descends to a bifunctor $\circ: [B, C]_{\mathcal{C}} \times [A, B]_{\mathcal{C}} \rightarrow [A, C]_{\mathcal{C}}$ that is associative and additive in both variables.
\end{prop}

\begin{prop}\label{prop:transitioning}
One has that $t_{B, C}(-) \circ t_{A, B}(-) = t_{A, C}(- \circ -)$.
\end{prop}

\begin{proof}
The key observation is that $t_{A, B}$ descends from a well defined homomorphism
\[ s_{A, B}: \mathrm{Hom}_{\mathcal{C}}(A, B) \rightarrow [\widetilde{A}_1, \widetilde{B}_1]_{\mathcal{C}}, f \mapsto \widetilde{f_1}+\mathcal{P}_{\mathcal{C}}(\widetilde{A}_1, \widetilde{B}_1) \, . \]
Because $s_{B, C}(-) \circ s_{A, B}(-) = s_{A, C}(- \circ -)$, the respective maps factor through quotient homomorphisms as in the commutative diagram
\begin{center}
\begin{tikzcd}[column sep = scriptsize]
	{\scriptstyle \mathrm{Hom}_{\mathcal{C}}(B{,} \, C) \times \mathrm{Hom}_{\mathcal{C}}(A{,} \, B)} \arrow[r] \arrow[rrrr, bend right = 10, "{\scriptstyle s_{B, C} \times s_{A, B}}"] \arrow[ddd, "{\scriptstyle - \circ -}"] & {\scriptstyle \mathrm{Hom}_{\mathcal{C}}(B{,} \, C) \times {[}A{,} \, B{]}_{\mathcal{C}}} \arrow[r] \arrow[rrr, bend left = 15, "{\scriptstyle s_{B, C} \times t_{A, B}}"] & {\scriptstyle {[}B{,} \, C{]}_{\mathcal{C}} \times {[}A{,} \, B{]}_{\mathcal{C}}} \arrow[rr, "{\scriptscriptstyle t_{B, C} \times t_{A, B}}" near start] & & {\scriptstyle {[}\widetilde{B}_1{,} \, \widetilde{C}_1{]} \times {[}\widetilde{A}_1{,} \, \widetilde{B}_1{]}} \arrow[ddd, "{\scriptstyle - \circ -}"] \\ \\ \\
	{\scriptstyle \mathrm{Hom}_{\mathcal{C}}(A{,} \, C)} \arrow[rr] \arrow[rrrr, bend right = 10, "s_{A, C}"] & & {\scriptstyle {[}A{,} \, C{]}_{\mathcal{C}}} \arrow[rr, "{\scriptstyle t_{A, C}}"] & & {\scriptstyle {[}\widetilde{A}_1{,} \, \widetilde{C}_1{]}}
\end{tikzcd}
\end{center}
Its right-hand side and Proposition~\ref{prop:biadditive} imply that $t_{B, C}(-) \circ t_{A, B}(-) = t_{A, C}(- \circ -)$.
\end{proof}

An explicit formula for induced homomorphisms of the naïve construction is thus given in the following definition.

\begin{defn}\label{defn:inducednaivemorph}
Let $A$, $B$, $C$ be objects in $\mathcal{C}$ and let $A_{\bullet}$, $B_{\bullet}$, $C_{\bullet}$ be projective resolutions. Let $f: B \rightarrow C$ be a morphism and let $(\widetilde{f}_k: \widetilde{B}_k \rightarrow \widetilde{C}_k)_{k \in \mathbb{N}_0}$ be as in Definition~\ref{defn:resolindmorph}, meaning that $\widetilde{f}_0 := f$ and that $\widetilde{f}_{k+1}$ is a lift $\widetilde{f}_k$. By Proposition~\ref{prop:transitioning}, the square of homomorphisms of abelian groups
\begin{equation}\label{diag:inducenaivemorph}
\begin{tikzcd}
    {[}\widetilde{A}_{n+k}, \widetilde{B}_k{]}_{\mathcal{C}} \arrow[rrrrr, "\big(\widetilde{f}_k+\mathcal{P}_{\mathcal{C}}(\widetilde{B}_k{,} \widetilde{C}_k)\big) \circ -"] \arrow[d, "t_{\widetilde{A}_{n+k}, \widetilde{B}_k}"] & & & & & {[}\widetilde{A}_{n+k}, \widetilde{C}_k{]}_{\mathcal{C}} \arrow[d, "t_{\widetilde{A}_{n+k}, \widetilde{C}_k}"] \\
    {[}\widetilde{A}_{n+k+1}, \widetilde{B}_{k+1}{]}_{\mathcal{C}} \arrow[rrrrr, "\big(\widetilde{f}_{k+1}+\mathcal{P}_{\mathcal{C}}(\widetilde{B}_{k+1}{,} \widetilde{C}_{k+1})\big) \circ -"] & & & & & {[}\widetilde{A}_{n+k+1}, \widetilde{C}_{k+1}{]}_{\mathcal{C}}
\end{tikzcd}
\end{equation}
commutes. Then
\[ BC_{\mathcal{C}}^n(A, f) := \varinjlim_{k \in \mathbb{N}_0, n+k \geq 0}\big(\big(\widetilde{f}_k+\mathcal{P}_{\mathcal{C}}(\widetilde{B}_k, \widetilde{C}_k)\big)\circ - \big): BC_{\mathcal{C}}^n(A, B) \rightarrow BC_{\mathcal{C}}^n(A, C) \]
is the induced homomorphism in the naïve construction. \\

Analogously, if $g: C \rightarrow A$ is another morphism with a corresponding sequence of morphisms $(\widetilde{g}_k: \widetilde{B}_k \rightarrow \widetilde{C}_k)_{k \in \mathbb{N}_0}$, then
\[ BC_{\mathcal{C}}^n(g, B) := \varinjlim_{k \in \mathbb{N}_0, n+k \geq 0}\big(- \circ \big(\widetilde{g}_{n+k}+\mathcal{P}_{\mathcal{C}}(\widetilde{C}_{n+k}, \widetilde{A}_{n+k})\big)\big): BC_{\mathcal{C}}^n(A, B) \rightarrow BC_{\mathcal{C}}^n(C, B) \]
defines another induced homomorphism in the naïve construction.
\end{defn}

One can deduce from this definition the following proposition.

\begin{prop}\label{prop:naivebifunctor}
For every $n \in \mathbb{Z}$, $BC_{\mathcal{C}}^n(-, -): \mathcal{C}^{\mathrm{op}} \times \mathcal{C} \rightarrow \mathbf{Ab}$ forms a well defined bifunctor.
\end{prop}

As a step in showing that the naïve construction defines completed Ext-functors, we demonstrate that the isomorphism $\beta^n(-)$ is natural.

\begin{lem}\label{lem:betanatural}
For every $n \in \mathbb{Z}$ the isomorphism $\beta^n(-)$ is natural. That is, for every $f \in \mathrm{Hom}_{\mathcal{C}}(B, C)$ the square
\begin{center}
\begin{tikzcd}
  \widehat{\mathrm{Ext}}_{\mathcal{C}}^n(A{,}\, B) \arrow[r, "\beta^n(B)"] \arrow[d, "\widehat{\mathrm{Ext}}_{\mathcal{C}}^n(A{,}\, f)"] & BC_{\mathcal{C}}^n(A{,}\, B) \arrow[d, "BC_{\mathcal{C}}^n(A{,}\, f)"] \\
  \widehat{\mathrm{Ext}}_{\mathcal{C}}^n(A{,}\, C) \arrow[r, "\beta^n(C)"] & BC_{\mathcal{C}}^n(A{,}\, C)
\end{tikzcd}
\end{center}
commutes.
\end{lem}

\begin{proof}
By construction of $\widehat{\mathrm{Ext}}_{\mathcal{C}}^n(A, f)$ from Definition~\ref{defn:resolindmorph} and of $BC_{\mathcal{C}}^n(A, f)$ from Definition~\ref{defn:inducednaivemorph}, it suffices to consider the cube
\begin{center}
\begin{tikzcd}[row sep = scriptsize, column sep = tiny]
  & \mathrm{Ext}_{\mathcal{C}}^{n+k}(A, \widetilde{B}_k) \arrow[dl, "\mathrm{Ext}_{\mathcal{C}}^{n+k}(A{,} \widetilde{f}_k)"] \arrow[rr, "\beta_{n+k}(\widetilde{B}_k)" near end] \arrow[dd, shift left = 2mm, "\delta^{n+k}" near end] & & {[}\widetilde{A}_{n+k}{,} \, \widetilde{B}_k{]}_{\mathcal{C}} \arrow[dl, "{[}\widetilde{A}_{n+k}{,} \, \widetilde{f}_k{]}_{\mathcal{C}}" near end] \arrow[dd, "t_{\widetilde{A}_{n+k}, \widetilde{B}_k}"] \\
  \mathrm{Ext}_{\mathcal{C}}^{n+k}(A, \widetilde{C}_k) \arrow[rr, crossing over, "\beta_{n+k}(\widetilde{C}_k)" near end] \arrow[dd, "\delta^{n+k}"] & & {[}\widetilde{A}_{n+k}{,} \, \widetilde{C}_k{]}_{\mathcal{C}} \\
  & \mathrm{Ext}_{\mathcal{C}}^{n+k+1}(A, \widetilde{B}_{k+1}) \arrow[dl, "\mathrm{Ext}_{\mathcal{C}}^{n+k+1}(A{,} \widetilde{f}_{k+1})" near start] \arrow[rr, "\beta_{n+k+1}(\widetilde{B}_{k+1})" near end] & & {[}\widetilde{A}_{n+k+1}{,} \, \widetilde{B}_{k+1}{]}_{\mathcal{C}} \arrow[dl, "{[}\widetilde{A}_{n+k+1}{,} \, \widetilde{f}_{k+1}{]}_{\mathcal{C}}"] \\
  \mathrm{Ext}_{\mathcal{C}}^{n+k+1}(A, \widetilde{C}_{k+1}) \arrow[rr, "\beta_{n+k+1}(\widetilde{C}_{k+1})" near end] & & {[}\widetilde{A}_{n+k+1}{,} \, \widetilde{C}_{k+1}{]}_{\mathcal{C}} \arrow[from=uu, crossing over, shift right = 3mm, "t_{\widetilde{A}_{n+k}, \widetilde{C}_k}" near end]
\end{tikzcd}
\end{center}
for any $k \in \mathbb{N}_0$ with $n+k \geq 1$. By definition of the sequence $(\widetilde{f}_k: \widetilde{B}_k \rightarrow \widetilde{C}_k)_{k \in \mathbb{N}_0}$, the left and right hand side commute. The front and back side correspond to Diagram~\ref{diag:inducedbeta}. Lastly, the top and bottom side commute by Definition~\ref{defn:resolindmorph}, Definition~\ref{defn:inducednaivemorph}, Definition~\ref{defn:betaprelim} and form a direct system of commuting squares. The statement of the proposition follows by passing to the direct limit.
\end{proof}

We require the following definition in order to develop an explicit formula for connecting homomorphisms of the naïve construction.

\begin{defn}\label{defn:naiveconnmorph}
Let $0 \rightarrow B \rightarrow C \rightarrow D \rightarrow 0$ be a short exact sequence in $\mathcal{C}$ and denote by $(0 \rightarrow \widetilde{B}_k \rightarrow \widetilde{C}_k \rightarrow \widetilde{D}_k \rightarrow 0)_{k \in \mathbb{N}_0}$ a sequence of short exact sequences as in Definition~\ref{defn:resolconnmorph}, meaning that each subsequent sequence is a lift of the previous along a Horseshoe diagram. Let $A \in \mathrm{obj}(\mathcal{C})$ and denote by $A_{\bullet}$ a projective resolution. If $f: \widetilde{A}_{n+k} \rightarrow \widetilde{D}_k$ is a morphism in $\mathcal{C}$, one can combine Diagram~\ref{diag:satellitemorph} and Diagram~\ref{diag:delta} to the commutative diagram
\begin{center}
\begin{tikzcd}
  0 \arrow[r]
    & \widetilde{A}_{n+k+1} \arrow[r] \arrow[d, "\overline{f}^{\ast}"]
    & A_{n+k} \arrow[r] \arrow[d, "\overline{f}"]
    & \widetilde{A}_{n+k} \arrow[r] \arrow[d, "f"]
    & 0 \\
  0 \arrow[r]
    & \widetilde{D}_{k+1} \arrow[r] \arrow[d, "h^{\ast}"]
    & D_k \arrow[r] \arrow[d, "h"]
    & \widetilde{D}_k \arrow[r] \arrow[d, "\mathrm{id}_{\widetilde{D}_k}"]
    & 0 \\
  0 \arrow[r]
    & \widetilde{B}_k \arrow[r]
    & \widetilde{C}_k \arrow[r]
    & \widetilde{D}_k \arrow[r]
    & 0
\end{tikzcd}
\end{center}
Then, by Proposition~\ref{prop:biadditive} and Proposition~\ref{prop:transitioning}, there is a homomorphism
\begin{align*}
\tau_{n+k, k} := [\widetilde{A}_{n+k+1}, h^{\ast}]_{\mathcal{C}} \circ t_{\widetilde{A}_{n+k}, \widetilde{D}_k}: \; \; [\widetilde{A}_{n+k}, \widetilde{D}_k]_{\mathcal{C}} &\rightarrow [\widetilde{A}_{n+k+1}, \widetilde{B}_k]_{\mathcal{C}} \\
f + \mathcal{P}_{\mathcal{C}}(\widetilde{A}_{n+k}, \widetilde{D}_k) &\mapsto h^{\ast} \circ \overline{f}^{\ast} + \mathcal{P}_{\mathcal{C}}(\widetilde{A}_{n+k+1}, \widetilde{B}_k)
\end{align*}
which is well defined in the sense that it does not depend on the choice of the morphism $h^{\ast}: \widetilde{D}_{k+1} \rightarrow \widetilde{B}_k$.
\end{defn}

An explicit formula for connecting homomorphisms of the naïve construction is contained in the following theorem that also asserts that the natural isomorphisms $\beta^n$ preserve them.

\begin{thm}\label{thm:naiveconnmorph}
For any short exact sequence $0 \rightarrow B \rightarrow C \rightarrow D \rightarrow 0$ and any element $A$ in $\mathcal{C}$, both the squares
\begin{center}
\begin{tikzcd}[column sep = small, row sep = scriptsize]
  {\scriptstyle \mathrm{Ext}_{\mathcal{C}}^{n+k}(A{,} \widetilde{D}_k)} \arrow[rrr, "{\scriptscriptstyle \beta_{n+k}(\widetilde{D}_k)}"] \arrow[dd, "{\scriptscriptstyle \delta^{n+k}}"]
    & & & {\scriptstyle {[}\widetilde{A}_{n+k}{,} \, \widetilde{D}_k{]}_{\mathcal{C}}} \arrow[dd, "{\scriptscriptstyle \tau_{n+k, k}}"] \arrow[ddrr, bend right = 15, white, "\text{{\normalsize and}}" black]
    & & {\scriptstyle {[}\widetilde{A}_{n+k}{,} \, \widetilde{D}_k{]}_{\mathcal{C}}} \arrow[rrrrr, "{\scriptscriptstyle (-1)^k \tau_{n+k, k}}"] \arrow[dd, "{\scriptscriptstyle t_{\widetilde{A}_{n+k}, \widetilde{D}_k}}"]
    & & & & & {\scriptstyle {[}\widetilde{A}_{n+k+1}{,} \, \widetilde{B}_k{]}_{\mathcal{C}}} \arrow[dd, "{\scriptscriptstyle t_{\widetilde{A}_{n+k+1}, \widetilde{B}_k}}"] \\ \\
  {\scriptstyle \mathrm{Ext}_{\mathcal{C}}^{n+k+1}(A{,} \widetilde{B}_k)} \arrow[rrr, "{\scriptscriptstyle \beta_{n+k+1}(\widetilde{B}_k)}"]
    & & & {\scriptstyle {[}\widetilde{A}_{n+k+1}{,} \, \widetilde{B}_k{]}_{\mathcal{C}}}
    & & {\scriptstyle {[}\widetilde{A}_{n+k+1}{,} \, \widetilde{D}_{k+1}{]}_{\mathcal{C}}} \arrow[rrrrr, "{\scriptscriptstyle (-1)^{k+1}\tau_{n+k+1, k+1}}"]
    & & & & & {\scriptstyle {[}\widetilde{A}_{n+k+2}{,} \, \widetilde{B}_{k+1}{]}_{\mathcal{C}}}
\end{tikzcd}
\end{center}
commute for every $k \in \mathbb{N}_0$ with $n+k \geq 1$. The homomorphism in the direct limit of the above right hand square written as
\[ \tau^n := \varinjlim_{k \in \mathbb{N}_0, n+k \geq 1} (-1)^k \tau_{n+k, k}: BC_{\mathcal{C}}^n(A, D) \rightarrow BC_{\mathcal{C}}^{n+1}(A, B) \]
defines the connecting homomorphism in the naïve construction. It fits into the commuting square
\begin{center}
\begin{tikzcd}
     \widehat{\mathrm{Ext}}_{\mathcal{C}}^n(A{,} \, D) \arrow[r, "\beta^n(D)"] \arrow[d, "\widehat{\delta}^n"] & BC_{\mathcal{C}}^n(A{,} \, D) \arrow[d, "\tau^n"] \\
     \widehat{\mathrm{Ext}}_{\mathcal{C}}^{n+1}(A{,} \, B) \arrow[r, "\beta^{n+1}(B)"] & BC_{\mathcal{C}}^{n+1}(A{,} \, B)
\end{tikzcd}
\end{center}
\end{thm}

\begin{proof}
Taking a morphism $h^{\ast}: \widetilde{D}_{k+1} \rightarrow \widetilde{B}_k$ as in Definition~\ref{defn:naiveconnmorph}, we see that the square
\begin{equation}\label{diag:liftfactoring}
\begin{tikzcd}
    \mathrm{Ext}_{\mathcal{C}}^{n+k}(A{,} \, \widetilde{D}_k) \arrow[r, "\delta^{n+k}"] \arrow[d, "\mathrm{id}"] & \mathrm{Ext}_{\mathcal{C}}^{n+k+1}(A{,} \, \widetilde{D}_{k+1}) \arrow[d, "\mathrm{Ext}_{\mathcal{C}}^{n+k+1}(A{,} \, h^{\ast})"] \\
    \mathrm{Ext}_{\mathcal{C}}^{n+k}(A{,} \, \widetilde{D}_k) \arrow[r, "\delta^{n+k}"] & \mathrm{Ext}_{\mathcal{C}}^n(A{,} \, \widetilde{B}_k)
\end{tikzcd}
\end{equation}
commutes. Using this and Definition~\ref{defn:naiveconnmorph}, we can factorise the top left square in the statement of the Lemma as
\begin{center}
\begin{tikzcd}
  \mathrm{Ext}_{\mathcal{C}}^{n+k}(A{,} \, \widetilde{D}_k) \arrow[rrr, "\beta_{n+k}(\widetilde{D}_k)"] \arrow[d, "\delta^{n+k}"] & & & {[}\widetilde{A}_{n+k}{,} \, \widetilde{D}_k{]}_{\mathcal{C}} \arrow[d, "t_{\widetilde{A}_{n+k}, \widetilde{D}_k}"] \\
  \mathrm{Ext}_{\mathcal{C}}^{n+k+1}(A{,} \, \widetilde{D}_{k+1}) \arrow[rrr, "\beta_{n+k+1}(\widetilde{D}_{k+1})"] \arrow[d, "\mathrm{Ext}_{\mathcal{C}}^{n+k+1}(A{,} \, h^{\ast})"] & & & {[}\widetilde{A}_{n+k+1}{,} \, \widetilde{D}_{k+1}{]}_{\mathcal{C}} \arrow[d, "{[}\widetilde{A}_{n+k+1}{,}\, h^{\ast}{]}_{\mathcal{C}}"] \\
  \mathrm{Ext}_{\mathcal{C}}^{n+k+1}(A{,} \, \widetilde{B}_k) \arrow[rrr, "\beta_{n+k+1}(\widetilde{B}_k)"] & & & {[}\widetilde{A}_{n+k+1}{,} \, \widetilde{B}_k{]}_{\mathcal{C}}
\end{tikzcd}
\end{center}
It commutes by Lemma~\ref{lem:betisomorphism} and Lemma~\ref{lem:betanatural}. By this square, Diagram~\ref{diag:anticomm} and Diagram~\ref{diag:inducedbeta}, we conclude that the top right square in the statement of the lemma commutes. The latter squares form a direct system in whose direct limit we obtain the bottom square in the lemma.
\end{proof}

We thus conclude the main result of this section.

\begin{thm}\label{thm:resolbenson}
The naïve construction $(BC_{\mathcal{C}}^{\bullet}(A, -), \tau^{\bullet})$ forms a Mislin completion of $(\mathrm{Ext}_{\mathcal{C}}^{\bullet}(A, -), \delta^{\bullet})$. More specifically, $(BC_{\mathcal{C}}^{\bullet}(A, -), \tau^{\bullet})$ forms a cohomological functor and $\beta^{\bullet}: (\widehat{\mathrm{Ext}}_{\mathcal{C}}^{\bullet}(A, -), \widehat{\delta}^{\bullet}) \rightarrow (BC_{\mathcal{C}}^{\bullet}(A, -), \tau^{\bullet})$ is an isomorphism of cohomological functors.
\end{thm}

\begin{proof}
This follows from Lemma~\ref{lem:betanatural} and Theorem~\ref{thm:naiveconnmorph} together with the formula for induced homomorphisms from Definition~\ref{defn:inducednaivemorph} and for connecting homomorphisms from Theorem~\ref{thm:naiveconnmorph}.
\end{proof}

\begin{rem}\label{rem:mislinsmistake}
In the case of modules over a ring, Mislin constructs in~\cite[p.~299]{mis94} an isomorphism
\[ \varinjlim_{k \in \mathbb{N}_0} S^{-k}\mathrm{Ext}^{n+k}(A, -) \rightarrow BC^n(A, -) \]
for every $n \in \mathbb{Z}$. Then he asserts that one can turn $BC^{\bullet}(A, -)$ into a Mislin completion by transferring the structure of $\widehat{\mathrm{Ext}}^{\bullet}(A, -)$. Since he does not explain how to construct the connecting homomorphism
\[ \tau^n: BC^n(A, D) \rightarrow BC^{n+1}(A, B) \, , \]
he erroneously misses the alternating signs involved in the direct system that gives rise to $\tau^n$ in Theorem~\ref{thm:naiveconnmorph}.
\end{rem}

\section{The hypercohomology and the resolution construction}\label{sec:hypercohom}

We demonstrate in this section that the hypercohomology construction gives rise to completed Ext-functors. At the heart of this is the development of explicit formulae for the induced homomorphisms and for the connecting homomorphisms of this construction. In order to achieve this, we reformulate the cohomology groups of the Vogel complex underlying the construction: modulo chain homotopy they are represented by almost chain maps, meaning chain maps in which finitely at most finitely many morphisms fail to commute with the respective boundary maps. Contrary to what might expect, the connecting homomorphisms of the hypercohomology construction do not arise from the Snake lemma. Instead, they are obtained by shifting the almost chain maps by one degree and composing them with an auxiliary chain map. Hence, after constructing isomorphisms from the terms of the resolution construction to those of the hypercohomology construction, it will not be apparent why these should be isomorphisms of cohomological functors. The latter necessitates a technical result on the auxiliary chain maps used in the construction of the connecting homomorphisms of the hypercohomology construction. This allows us then to prove that cohomological functors resulting from the hypercohomology construction are isomorphic to those from the resolution construction and those from the naïve construction. \\

As the hypercohomology constructions has been generalised previously, let us provide pointers to the literature. Although S.\! Guo and L.\! Liang have turned the cohomology groups of the Vogel complex into a cohomological functor using the Snake lemma in~\cite[Proposition~4.8]{guo23}, we do not know whether their construction yields completed Ext-functors (Question~\ref{ques:comparewithguoliang}). Similarly, J.\! Hu et al.\! introduce the Vogel complex in~\cite[p.~7]{hu21} in a vastly general setting, but we also do not know whether their construction yields completed Ext-functors.

\subsection{Induced homomorphisms and connecting homomorphisms}\label{subsec:hyperindconn}

In this subsection, we develop the explicit formulae for the induced homomorphisms and for the connecting homomorphisms of the hypercohomology construction. To this end, we reformulate the cohomology groups of the Vogel complex underlying this construction in terms of almost chain maps. Benson and Carlson hint in~\cite[p.~109]{ben92} at such a reformulation in the case of modules over a ring. We require the following notation. If $(M_k, \mu_k)_{k \in \mathbb{Z}}$ is a chain complex in $\mathcal{C}$ with boundary maps $\mu_k: M_k \rightarrow M_{k-1}$ and $n \in \mathbb{Z}$, then define $(M[n]_k, \mu[n]_k)_{k \in \mathbb{Z}}$ to be $M[n]_k := M_{k+n}$ and $\mu[n]_k := (-1)^n \mu_{k+n}$~\cite[p.~154]{gel03}, \cite[p.~9--10]{wei94}. If $(N_k, \nu_k)_{k \in \mathbb{Z}}$ denotes another chain complex, then we have constructed the hypercohomoology complex $\mathrm{Hyp}_{\mathcal{C}}(M_{\bullet}, N_{\bullet})_{\bullet}$ in Equation~\ref{eq:hypercohombd}. According to~\cite[p.~62--63]{wei94}, an $n$-cocycle of $\mathrm{Hyp}_{\mathcal{C}}(M_{\bullet}, N_{\bullet})_{\bullet}$ is exactly a chain map of the form $M[n]_{\bullet} \rightarrow N_{\bullet}$ where $n$-coboundaries are nullhomotopic chain maps of this form. In the same way, we observe that an $n$-cocycle of the Vogel complex $\mathrm{Vog}_{\mathcal{C}}(M_{\bullet}, N_{\bullet})_{\bullet}$ is a collection of morphisms $(f_{k+n}: M[n]_k \rightarrow N_k)_{k \in \mathbb{Z}}$ such that all but finitely many of the squares
\begin{center}
\begin{tikzcd}
    M_{k+n} \arrow[r, "f_{k+n}"] \arrow[d, "(-1)^n \mu_{n+k}"] & N_k \arrow[d, "\nu_k"] \\
    M_{k-1+n} \arrow[r, "f_{k-1+n}"]  & N_{k-1}
\end{tikzcd}
\end{center}
commute. We follow the convention in~\cite[p.~109]{ben92} by calling this an almost chain map of degree $n$. We call $f_{\bullet+n}: M[n]_{\bullet} \rightarrow N_{\bullet}$ nullhomotopic if there is a collection of morphisms $(e_{k+n}: M[n]_k \rightarrow N_{k+1})_{k \in \mathbb{Z}}$ such that
\[ f_{k+n} = e_{k-1+n} \circ \mu[n]_k + \mu[n]_{k+1} \circ e_{k+n} \]
for all but finitely many $k \in \mathbb{Z}$. We observe that an $n$-coboundary of $\mathrm{Vog}_{\mathcal{C}}(M_{\bullet}, N_{\bullet})_{\bullet}$ is exactly such a nullhomotopic almost chain map. Denote by $\mathrm{Hom}_{\mathrm{Ch}(\mathcal{C})}(M[n]_{\bullet}, N_{\bullet})$ the set of chain maps $M[n]_{\bullet} \rightarrow N_{\bullet}$ and by $\mathrm{Null}_{\mathrm{Ch}(\mathcal{C})}(M[n]_{\bullet}, N_{\bullet})$ the subset of nullhomotopic maps. Analogously, write $\widehat{\mathrm{Hom}}_{\mathrm{Ch}(\mathcal{C})}(M[n]_{\bullet}, N_{\bullet})$ for the set of almost chain maps $M[n]_{\bullet} \rightarrow N_{\bullet}$ and by $\widehat{\mathrm{Null}}_{\mathrm{Ch}(\mathcal{C})}(M[n]_{\bullet}, N_{\bullet})$ the subset of nullhomotopic almost chain maps. We thus conclude that
\begin{align*}
H^n(\mathrm{Hyp}_{\mathcal{C}}(M_{\bullet}, N_{\bullet})_{\bullet}) &= \mathrm{Hom}_{\mathrm{Ch}(\mathcal{C})}(M[n]_{\bullet}, N_{\bullet})/ \mathrm{Null}_{\mathrm{Ch}(\mathcal{C})}(M[n]_{\bullet}, N_{\bullet}) \text{ and} \\
H^n(\mathrm{Vog}_{\mathcal{C}}(M_{\bullet}, N_{\bullet})_{\bullet}) &= \widehat{\mathrm{Hom}}_{\mathrm{Ch}(\mathcal{C})}(M[n]_{\bullet}, N_{\bullet})/ \widehat{\mathrm{Null}}_{\mathrm{Ch}(\mathcal{C})}(M[n]_{\bullet}, N_{\bullet}) \, .
\end{align*}
Having reformulated the hypercohomology construction, we assume from now on that $M_{\bullet}$, $N_{\bullet}$ are projective resolutions of objects $M, N \in \mathrm{obj}(\mathcal{C})$. \\

The induced homomorphisms and connecting homomorphisms in the hypercohomology construction are modelled on the corresponding homomorphisms for Ext-functors. To substantiate this assertion, we construct an explicit isomorphism from the Ext-group $\mathrm{Ext}_{\mathcal{C}}^n(M, N)$ to the cohomology group $H^n(\mathrm{Hyp}_{\mathcal{C}}(M_{\bullet}, N_{\bullet})_{\bullet})$. It is well known that these terms are isomorphic. The classical setting to prove this are derived categories as can be seen in \cite[Chapter~III]{gel03} and \cite[Sections~2.4--2.5]{wei94}. However, one cannot retrieve an explicit isomorphism from this setting. The construction of our isomorphism is similar to a construction in Weibel's book~\cite[pp.~63--64]{wei94} where he compares the chain complex $\mathrm{Hom}_{\mathcal{C}}(M_{\bullet}, N)$ to $\mathrm{Hyp}_{\mathcal{C}}(M_{\bullet}, N_{\bullet})_{\bullet, \bullet}$ seen as a double complex. As our techniques differ and we have not found this isomorphism in the literature otherwise, we provide full details. First, we need to distinguish the two emerging notions of Ext-functors whence we introduce the following notation.

\begin{notn}\label{notn:chainext}
For the remainder of this section, if $(A_k)_{k \in \mathbb{Z}}$ is a chain complex in $\mathcal{C}$, then it is assumed that $(A_k)_{k \in \mathbb{N}_0}$ is a projective resolution of an object $A$ in $\mathcal{C}$ and $A_k = 0$ for $k < 0$. We write
\[ \mathcal{E}xt_{\mathcal{C}}^n(A, B) := \mathrm{Hom}_{\mathrm{Ch}(\mathcal{C})}(A[n]_{\bullet}, B_{\bullet})/ \mathrm{Null}_{\mathrm{Ch}(\mathcal{C})}(A[n]_{\bullet}, B_{\bullet}) \]
whenever we consider the $n^{\text{th}}$ Ext-group as arising from chain maps modulo chain homotopy. Analogously, we write
\[ \widehat{\mathcal{E}xt}_{\mathcal{C}}^n(A, B) := \widehat{\mathrm{Hom}}_{\mathrm{Ch}(\mathcal{C})}(A[n]_{\bullet}, B_{\bullet})/ \widehat{\mathrm{Null}}_{\mathrm{Ch}(\mathcal{C})}(A[n]_{\bullet}, B_{\bullet}) \, . \]
If $a_k: A_k \rightarrow A_{k-1}$ denote the boundary maps of $A_{\bullet}$, then we continue to write
\[ \mathrm{Ext}_{\mathcal{C}}^n(A, B) = \mathrm{Ker}(\mathrm{Hom}_{\mathcal{C}}(a_{n+1}, B))/ \mathrm{Im}(\mathrm{Hom}_{\mathcal{C}}(a_n, B)) \]
for the $n^{\text{th}}$ Ext-group defined as the $n^{\text{th}}$ derived functor of $\mathrm{Hom}_{\mathcal{C}}(-, B)$.
\end{notn}

The required isomorphism from Ext-groups to the corresponding cohomology groups is given by

\begin{defn}\label{defn:extisoms}
Let $A, B \in \mathrm{obj}(\mathcal{C})$, $(A_{\bullet}, a_{\bullet})$, $(B_{\bullet}, b_{\bullet})$ be projective resolutions and $n \in \mathbb{N}_0$. By Definition~\ref{defn:betaprelim} and the Comparison Theorem~\cite[Theorem~2.2.6]{wei94}, there exists for every $\varphi \in \mathrm{Ker}(\mathrm{Hom}_{\mathcal{C}}(a_{k+1}, B))$ a chain map $(\varphi[n]_k: A[n]_k \rightarrow B_k)_{k \in \mathbb{Z}}$ that is unique up to chain homotopy and that renders the diagram
\begin{equation}\label{diag:lifttochainmap}
\begin{tikzcd}
    {} \dots {} \arrow[r] & A_{n+2} \arrow[rr, "(-1)^n a_{n+2}"] \arrow[d, "\varphi{[}n{]}_2"] & & A_{n+1} \arrow[rr, "(-1)^n a_{n+1}"] \arrow[d, "\varphi{[}n{]}_1"] & & A_n \arrow[rr, "\pi_n"] \arrow[drr, "\varphi"] \arrow[d, "\varphi{[}n{]}_0" near end] & & \widetilde{A}_n \arrow[d, "\alpha_n(B)(\varphi)"] \\
    {} \dots {} \arrow[r] & B_2 \arrow[rr, "b_2"] & & B_1 \arrow[rr, "b_1"] & & B_0 \arrow[rr, "b"] & & B
\end{tikzcd}
\end{equation}
commutative. If $\varphi \in \mathrm{Im}(\mathrm{Hom}_{\mathcal{C}}(a_n, B))$, then one can deduce that $\varphi[n]_{\bullet}$ is chain homotopic to the zero chain map. Therefore, there is a well defined homomorphism
\begin{align*}
\zeta_n: \; \mathrm{Ext}_{\mathcal{C}}^n(A, B) &\rightarrow \mathcal{E}xt_{\mathcal{C}}^n(A, B) \\
\varphi + \mathrm{Im}(\mathrm{Hom}_{\mathcal{C}}(a_n, B)) &\mapsto \varphi[n]_{\bullet} + \mathrm{Null}_{\mathrm{Ch}(\mathcal{C})}(A[n]_{\bullet}, B_{\bullet}) \, .
\end{align*}
Given that Ext-functors of negative degree vanish, we set $\zeta_n := 0$ in the case where $n < 0$.
\end{defn}

The following proposition follows from the previous definition and justifies Notation~\ref{notn:chainext}.

\begin{prop}
For every $n \in \mathbb{Z}$, $\zeta_n: \mathrm{Ext}^n(A, B) \rightarrow \mathcal{E}xt_{\mathcal{C}}^n(A, B)$ is an isomorphism of abelian groups.
\end{prop}

With this notation at hand, we develop an explicit formula for the induced homomorphisms of the hypercohomology construction. In case of $\mathcal{E}xt_{\mathcal{C}}^n(A, B)$ one can see how these induced homomorphisms are based on those of $\mathrm{Ext}^n(A, B)$ via the isomorphism $\zeta^n$.

\begin{defn}\label{defn:extchainmaps}
Let $f: B \rightarrow C$ and $g: C \rightarrow A$ be morphisms in $\mathcal{C}$. Assume that $A_{\bullet}$, $B_{\bullet}$, $C_{\bullet}$ are projective resolutions. By the Comparison Theorem~\cite[Theorem~2.2.6]{wei94}, there is a chain map $f_{\bullet}: B_{\bullet} \rightarrow C_{\bullet}$ that is unique up to chain homotopy and that renders the diagram
\begin{center}
\begin{tikzcd}
    {} \dots {} \arrow[r] & B_2 \arrow[r] \arrow[d, "f_2"] & B_1 \arrow[r] \arrow[d, "f_1"] & B_0 \arrow[r] \arrow[d, "f_0"] & B \arrow[d, "f"] \\
    {} \dots {} \arrow[r] & C_2 \arrow[r] & C_1 \arrow[r] & C_0 \arrow[r] & C
\end{tikzcd}
\end{center}
commutative. Then the homomorphisms
\begin{align*}
\widehat{\mathcal{E}xt}_{\mathcal{C}}^n(A, f): \;  \widehat{\mathcal{E}xt}_{\mathcal{C}}^n(A, B) &\rightarrow \widehat{\mathcal{E}xt}_{\mathcal{C}}^n(A, C) \\
\varphi_{\bullet} + \widehat{\mathrm{Null}}_{\mathrm{Ch}(\mathcal{C})}(A[n]_{\bullet}, B_{\bullet}) &\mapsto f_{\bullet} \circ \varphi_{\bullet} + \widehat{\mathrm{Null}}_{\mathrm{Ch}(\mathcal{C})}(A[n]_{\bullet}, C_{\bullet})
\end{align*}
and
\begin{align*}
\widehat{\mathcal{E}xt}_{\mathcal{C}}^n(g, B): \;  \widehat{\mathcal{E}xt}_{\mathcal{C}}^n(A, B) &\rightarrow \widehat{\mathcal{E}xt}_{\mathcal{C}}^n(C, B) \\
\varphi_{\bullet} + \widehat{\mathrm{Null}}_{\mathrm{Ch}(\mathcal{C})}(A[n]_{\bullet}, B_{\bullet}) &\mapsto \varphi_{\bullet} \circ g[n]_{\bullet} + \widehat{\mathrm{Null}}_{\mathrm{Ch}(\mathcal{C})}(C[n]_{\bullet}, B_{\bullet}) \, .
\end{align*}
are the induced homomorphisms of the hypercohomology construction. The induced homomorphisms
\[ \mathcal{E}xt_{\mathcal{C}}^n(A, f): \mathcal{E}xt_{\mathcal{C}}^n(A, B) \rightarrow \mathcal{E}xt_{\mathcal{C}}^n(A, C) \text{ and } \mathcal{E}xt_{\mathcal{C}}^n(g, B): \mathcal{E}xt_{\mathcal{C}}^n(A, B) \rightarrow \mathcal{E}xt_{\mathcal{C}}^n(C, B) \]
are defined analogously.
\end{defn}

We deduce from this definition

\begin{prop}\label{prop:hyperbifunctor}
For any $n \in \mathbb{Z}$, both
\[ \mathcal{E}xt_{\mathcal{C}}^n(-, -): \mathcal{C}^{\mathrm{op}} \times \mathcal{C} \rightarrow \mathbf{Ab} \text{ and } \, \widehat{\mathcal{E}xt}_{\mathcal{C}}^n(-, -): \mathcal{C}^{\mathrm{op}} \times \mathcal{C} \rightarrow \mathbf{Ab} \]
form well defined bifunctors.
\end{prop}

Finally, we develop an explicit formula for the connecting homomorphisms of the hypercohomology construction. These connecting homomorphisms are based on shifting any given almost chain map by one degree via another chain map $\pi_{\bullet}^n$ and composing it with an auxiliary chain map $h^{\ast}_{\bullet+1}$.

\begin{defn}\label{defn:hyperconnmorph}
Let
\begin{align*}
&0 \rightarrow B \rightarrow C \rightarrow D \rightarrow 0 \quad \text{and} \\
&0 \rightarrow \widetilde{D}_1 \rightarrow D_0 \rightarrow D \rightarrow 0
\end{align*}
be short exact sequences in $\mathcal{C}$ where $D_0$ is projective. Take the latter to be part of a  projective resolution $(D_{\bullet}, d_{\bullet})$ of $D$. Define the chain complex $(D_{\bullet}^{+1}, d_{\bullet}^{+1})$ by setting
\begin{align*}
&(D_m^{+1}, d_m^{+1})_{m \in \mathbb{N}} := (D_{m+1}, d_{m+1})_{m \in \mathbb{N}} \\
&D_0^{+1} := D_1 \quad \text{and} \quad D_m^{+1} := 0 \text{ for } m < 0 \, .
\end{align*}
Note that $D_{\bullet}^{+1}$ is a projective resolution of $\widetilde{D}_1$. Define the chain map $\pi_{\bullet}^n: D_{\bullet} \rightarrow D^{+1}[-1]_{\bullet}$ by setting
\begin{equation}
\pi_m^n := \begin{cases} (-1)^{m-1+n}\mathrm{id}_{D_m}: D_m \rightarrow D_m &\text{if } m \geq 1 \\ {} \; \; 0 &\text{if } m \leq 0 \end{cases}
\end{equation}
According to~\cite[p.~166]{gel03},
\[ \mathcal{E}xt_{\mathcal{C}}^{n+1}(A, D) = \mathrm{Hom}_{\mathrm{Ch}(\mathcal{C})}(A[n]_{\bullet}, D[-1]_{\bullet})/ \mathrm{Null}_{\mathrm{Ch}(\mathcal{C})}(A[n]_{\bullet}, D[-1]_{\bullet}) \]
where an analogous statement pertains to $\widehat{\mathcal{E}xt}_{\mathcal{C}}^{n+1}(A, D)$. Then
\begin{align*}
\Delta^n: \mathcal{E}xt_{\mathcal{C}}^n(A, D) &\rightarrow \mathcal{E}xt_{\mathcal{C}}^{n+1}(A, \widetilde{D}_1) \\
\varphi_{\bullet+n}+\mathrm{Null}_{\mathrm{Ch}(\mathcal{C})}(A[n]_{\bullet}, D_{\bullet}) &\mapsto \pi_{\bullet}^n \circ \varphi_{\bullet+n}+\mathrm{Null}_{\mathrm{Ch}(\mathcal{C})}(A[n]_{\bullet}, D^{+1}[-1]_{\bullet})
\end{align*}
is the connecting homomorphism associated to $0 \rightarrow \widetilde{D}_1 \rightarrow D_0 \rightarrow D \rightarrow 0$. Define the connecting homomorphism $\widehat{\Delta}^n: \widehat{\mathcal{E}xt}_{\mathcal{C}}^n(A, D) \rightarrow \widehat{\mathcal{E}xt}_{\mathcal{C}}^{n+1}(A, \widetilde{D}_1)$ analogously. \\

If $h^{\ast}: \widetilde{D}_1 \rightarrow B$ is a lift of $\mathrm{id}_D: D \rightarrow D$ as in Diagram~\ref{diag:delta} and $h^{\ast}_{\bullet+1}: D_{\bullet}^{+1} \rightarrow B_{\bullet}$ a chain map as in the Comparison Theorem, then
\begin{align*}
\Delta^n: \mathcal{E}xt_{\mathcal{C}}^n(A, D) &\rightarrow \mathcal{E}xt_{\mathcal{C}}^{n+1}(A, B) \\
\varphi_{\bullet+n}+\mathrm{Null}_{\mathrm{Ch}(\mathcal{C})}(A[n]_{\bullet}, D_{\bullet}) &\mapsto h^{\ast}[-1]_{\bullet+1} \circ \pi_{\bullet}^n \circ \varphi_{\bullet+n}+\mathrm{Null}_{\mathrm{Ch}(\mathcal{C})}(A[n]_{\bullet}, B[-1]_{\bullet})
\end{align*}
is the connecting homomorphism associated to $0 \rightarrow B \rightarrow C \rightarrow D \rightarrow 0$. Define the connecting homomorphism $\widehat{\Delta}^n: \widehat{\mathcal{E}xt}_{\mathcal{C}}^n(A, D) \rightarrow \widehat{\mathcal{E}xt}_{\mathcal{C}}^{n+1}(A, B)$ analogously.
\end{defn}

\subsection{Termwise isomorphisms}

The goal of this subsection is to construct explicit isomorphisms from the terms of the resolution construction to those of the hypercohomology construction. As a first step, we prove that $\zeta_{\bullet}$ from Definition~\ref{defn:extisoms} is an isomorphism of cohomological functors, meaning that the cohomology groups $\mathcal{E}xt_{\mathcal{C}}^{\bullet}(A, -)$ with the choice of connecting homomorphisms $\Delta^{\bullet}$ are Ext-functors. We then replicate the resolution construction using direct limits of the terms $\mathcal{E}xt_{\mathcal{C}}^{\bullet}(A, -)$ instead of the terms $\mathrm{Ext}_{\mathcal{C}}^{\bullet}(A, -)$. This yields isomorphisms from the terms of the resolution construction $\widehat{\mathrm{Ext}}_{\mathcal{C}}^{\bullet}(A, -)$ to those of the hypercohomology construction $\widehat{\mathcal{E}xt}_{\mathcal{C}}^{\bullet}(A, -)$.

\begin{lem}\label{lem:ordinaryhypercohom}
The isomorphisms $\zeta_{\bullet}: (\mathrm{Ext}_{\mathcal{C}}^{\bullet}(A, -), \delta^{\bullet}) \rightarrow (\mathcal{E}xt_{\mathcal{C}}^{\bullet}(A, -), \Delta^{\bullet})$ from Definition~\ref{defn:extisoms} form an isomorphism of cohomological functors. In particular, $(\mathcal{E}xt_{\mathcal{C}}^{\bullet}(A, -), \Delta^{\bullet})$ forms a cohomological functor.
\end{lem}

\begin{proof}
It follows from Definition~\ref{defn:extchainmaps} that the isomorphisms $\zeta_n: \mathrm{Ext}_{\mathcal{C}}^n(A, -) \rightarrow \mathcal{E}xt_{\mathcal{C}}^n(A, -)$ are natural. Hence, it suffices to prove that $\zeta_n$ commutes with the connecting homomorphisms. Let $0 \rightarrow B \rightarrow C \rightarrow D \rightarrow 0$ and $0 \rightarrow \widetilde{D}_1 \rightarrow D_0 \rightarrow D \rightarrow 0$ be short exact sequences in $\mathcal{C}$ where $D_0$ is projective. Denote by $h^{\ast}: \widetilde{D}_1 \rightarrow B$ a lift of $\mathrm{id}_D$ as in Definition~\ref{defn:hyperconnmorph}. By Diagram~\ref{diag:liftfactoring}, the connecting homomorphism associated to the former short exact sequence factors through the one associated to the latter as
\[ \delta^n: \mathrm{Ext}_{\mathcal{C}}^n(A, D) \xrightarrow{\delta^n} \mathrm{Ext}_{\mathcal{C}}^{n+1}(A, \widetilde{D}_1) \xrightarrow{\mathrm{Ext}_{\mathcal{C}}^{n+1}(A, h^{\ast})} \mathrm{Ext}_{\mathcal{C}}^{n+1}(A, B) \, . \]
Thus, the left hand triangle in the diagram
\begin{equation}\label{diag:trapezia}
\begin{tikzcd}
  \mathrm{Ext}_{\mathcal{C}}^n(A{,} \, D) \arrow[rrr, "\zeta_n"] \arrow[dr, "\delta^n"] \arrow[dd, "\delta^n"]
    & & & \mathcal{E}xt_{\mathcal{C}}^n(A{,} \, D) \arrow[dl, "\Delta^n"] \arrow[dd, shift left = 2mm, "\Delta^n"] \\
  & \mathrm{Ext}_{\mathcal{C}}^{n+1}(A{,} \, \widetilde{D}_1) \arrow[r, "\zeta_{n+1}"] \arrow[dl, "\mathrm{Ext}_{\mathcal{C}}^{n+1}(A{,} \, h^{\ast})" near start]
    & \mathcal{E}xt_{\mathcal{C}}^{n+1}(A{,} \, \widetilde{D}_1) \arrow[dr, "\mathcal{E}xt_{\mathcal{C}}^{n+1}(A{,} \, h^{\ast})" near start] & \\
  \mathrm{Ext}_{\mathcal{C}}^{n+1}(A{,} \, B) \arrow[rrr, "\zeta_{n+1}"]
    & & & \mathcal{E}xt_{\mathcal{C}}^{n+1}(A{,} \, B)
\end{tikzcd}
\end{equation}
commutes. By Definition~\ref{defn:extchainmaps} and Definition~\ref{defn:hyperconnmorph}, the bottom trapezium and the right hand triangle commute. The top trapzium can be seen as a generalisation of \cite[Proposition~1.1]{emm17} where we use a different method to prove that it commutes. Let $\varphi \in \mathrm{Ker}(\mathrm{Hom}_{\mathcal{C}}(a_n, D))$. We explain how one can combine Diagram~\ref{diag:mislinbenson} and Diagram~\ref{diag:lifttochainmap} as
\begin{equation}\label{diag:zetanddelta}
\begin{tikzcd}
    {} \dots {} \arrow[r]
    & A_{n+2} \arrow[rr, "(-1)^n a_{n+2}"] \arrow[d, "\varphi{[}n{]}_2"]
    & & A_{n+1} \arrow[r, "\pi_{n+1}" near end] \arrow[rr, bend left = 26, "a_{n+1}"] \arrow[d, "\varphi{[}n{]}_1" near end] \arrow[dr, "\Phi"]
    & \widetilde{A}_{n+1} \arrow[r, "\iota_{n+1}" near start] \arrow[d]
    & A_n \arrow[r, "\pi_n"] \arrow[dr, bend left = 15, "\varphi"] \arrow[d, "\varphi{[}n{]}_0"]
    & \widetilde{A}_n \arrow[d, "\alpha_{n+k}(B)(\varphi)"] \\
    {} \dots {} \arrow[r]
    & D_2 \arrow[rr, "d_2"]
    & & D_1 \arrow[r, "\pi_1"] \arrow[rr, bend right = 28, "d_1"]
    & \widetilde{D}_1 \arrow[r, "\iota_1"]
    & D_0 \arrow[r, "\pi_0"]
    & D
\end{tikzcd}
\end{equation}
The right most morphisms up to $\Phi: A_{n+1} \rightarrow \widetilde{B}_1$ are taken from Diagram~\ref{diag:mislinbenson} while $\varphi: A_n \rightarrow D$ and $\varphi[n]_{\bullet}: A[n]_{\bullet} \rightarrow D_{\bullet}$ are the morphisms as they occur in Diagram~\ref{diag:lifttochainmap}. In particular,
\[ \delta^n \big(\varphi+\mathrm{Im}(\mathrm{Hom}_{\mathcal{C}}(a_n, D))\big) = \Phi+\big(\varphi+\mathrm{Im}(\mathrm{Hom}_{\mathcal{C}}(a_n, \widetilde{D}_1)) \, . \]
However, $\iota_{n+1} \circ \pi_{n+1} = a_{n+1}$ might not agree with $(-1)^n a_{n+1}: A_{n+1} \rightarrow A_n$. This together Diagram~\ref{diag:mislinbenson} and the Comparison Theorem implies that $\Phi = (-1)^n \pi_1 \circ \varphi[n]_1$. Therefore, the chain map $\Phi[n+1]_{\bullet}: A[n+1] \rightarrow \widetilde{D}_{\bullet}^{+1}$ defined by
\[ \Phi[n+1]_m := \begin{cases} (-1)^{n+m} \varphi[n]_{m+1}: A_{m+n+1} \rightarrow D_{m+1} &\text{if } m \geq 0 \\ \; \; 0 &\text{if } m < 0 \end{cases} \]
is a lift of $\Phi$ as in Diagram~\ref{diag:lifttochainmap}. The alternating signs are due to the fact that the boundary maps of $A[n]_{\bullet}$ and $A[n+1]_{\bullet}$ have opposite signs. We conclude by Diagram~\ref{diag:zetanddelta} and Definition~\ref{defn:hyperconnmorph} that
\[ \zeta_{n+1}\Big(\delta^n \big(\varphi+\mathrm{Im}(\mathrm{Hom}_{\mathcal{C}}(a_n, D))\big)\Big) = \Delta^n \Big(\zeta^n \big(\varphi+\mathrm{Im}(\mathrm{Hom}_{\mathcal{C}}(a_n, D))\big)\Big) \, . \]
The top trapezium and thus all of Diagram~\ref{diag:trapezia} commutes, proving that $\zeta_{\bullet}$ commutes with the connecting homomorphisms.
\end{proof}

We reformulate the resolution construction using direct limits of the terms $\mathcal{E}xt_{\mathcal{C}}^{\bullet}(A, -)$ as another step in constructing an isomorphism from the terms of the resolution construction $\mathrm{Ext}_{Res, \mathcal{C}}^n(A, B)$ to those of the hypercohomology construction $\widehat{\mathcal{E}xt}_{\mathcal{C}}^n(A, B)$.

\begin{defn}
Let $B \in \mathrm{obj}(\mathcal{C})$ and denote by $B_{\bullet}$ a projective resolution. For any $n \in \mathbb{Z}$ define
\[ \mathcal{E}xt_{Res, \mathcal{C}}^n(A, B) := \varinjlim_{k \in \mathbb{N}_0} (\mathcal{E}xt_{\mathcal{C}}^{n+k}(A, \widetilde{B}_k), \Delta^{n+k}) \, . \]
Let $f: B \rightarrow C$ be a morphism in $\mathcal{C}$ and $(\widetilde{f}_k: \widetilde{B}_k \rightarrow \widetilde{C}_k)_{k \in \mathbb{N}_0}$ be a sequence of morphisms as in Definition~\ref{defn:resolindmorph}, meaning that $\widetilde{f}_{k+1}$ is a lift of $\widetilde{f}_k$. Then define the induced homomorphism by
\[ \mathcal{E}xt_{Res, \mathcal{C}}^n(A, f) := \varinjlim_{k \in \mathbb{N}_0} \mathcal{E}xt_{\mathcal{C}}^{n+k}(A, \widetilde{f}_k): \mathcal{E}xt_{\mathcal{C}}^n(A, B) \rightarrow \mathcal{E}xt_{\mathcal{C}}^n(A, C) \, . \]
Let $0 \rightarrow B \rightarrow C \rightarrow D \rightarrow 0$ be a short exact sequence and $(0 \rightarrow \widetilde{A}_k \rightarrow \widetilde{B}_k \rightarrow \widetilde{C}_k \rightarrow 0)_{k \in \mathbb{N}_0}$ be a sequence of short exact sequences as in Definition~\ref{defn:resolconnmorph}, meaning that each subsequent sequence is a lift of the previous along a Horseshoe diagram. Then define the connecting homomorphism by
\[ \widetilde{\Delta}^n := \varinjlim_{k \in \mathbb{N}_0} (-1)^k \Delta^{n+k}: \mathcal{E}xt_{Res, \mathcal{C}}^n(A, D) \rightarrow \mathcal{E}xt_{Res, \mathcal{C}}^{n+1}(A, B) \]
where $\Delta^{n+k}: \mathcal{E}xt_{\mathcal{C}}^{n+k}(A, \widetilde{D}_k) \rightarrow \mathcal{E}xt_{\mathcal{C}}^{n+k+1}(A, \widetilde{B}_k)$.
\end{defn}

We thus obtain an isomorphism from the terms of the resolution construction $\mathrm{Ext}_{Res, \mathcal{C}}^n(A, B)$.

\begin{defn}\label{defn:resolinchains}
By virtue of the isomorphisms of cohomological functors 
\[ \zeta_{n+k}: \mathrm{Ext}_{\mathcal{C}}^{n+k}(A, -) \rightarrow \mathcal{E}xt_{\mathcal{C}}^{n+k}(A, -) \]
define the natural isomorphisms from the terms of the resolution construction
\[ \zeta^n := \varinjlim_{k \in \mathbb{N}_0} \zeta_{n+k}: \mathrm{Ext}_{Res, \mathcal{C}}^n(A, -) \rightarrow \mathcal{E}xt_{Res, \mathcal{C}}^n(A, -) \, . \]
Then
\[ \zeta^{\bullet}: (\mathrm{Ext}_{Res, \mathcal{C}}^{\bullet}(A, -), \widetilde{\delta}^{\bullet}) \rightarrow (\mathcal{E}xt_{Res, \mathcal{C}}^{\bullet}(A, -), \widetilde{\Delta}^{\bullet}) \]
forms an isomorphism of cohomological functors.
\end{defn}

We introduce the following terminology in order ease notation when constructing a homomorphism from $\mathcal{E}xt_{Res, \mathcal{C}}^n(A, B)$ to the terms of the hypercohomology construction $\widehat{\mathcal{E}xt}_{\mathcal{C}}^n(A, B)$.

\begin{conv}
For the remainder of this section, we consider the cofinal system
\[ (\mathcal{E}xt_{\mathcal{C}}^{n+2k}(A, \widetilde{B}_{2k}), \Delta^{n+2k+1} \circ \Delta^{n+2k})_{k \in \mathbb{N}_0} \]
and specifically form $\mathcal{E}xt_{Res, \mathcal{C}}^n(A, B)$ as its direct limit.
\end{conv}

\begin{notn}
If $(B_{\bullet}, b_{\bullet})$ is a projective resolution of $B$ and $k \in \mathbb{N}_0$, then we define the projective resolution $(B_{\bullet}^{+k}, b_{\bullet}^{+k})_{m \in \mathbb{Z}}$ of $\widetilde{B}_k$ by setting $(B_m^{+k}, b_m^{+k})_{m \in \mathbb{N}} := (B_{m+k}, b_{m+k})_{m \in \mathbb{N}}$, $B_0^{+k} := B_k$ and $B_m^{+k} := 0$ for $m < 0$.
\end{notn}

To establish the required isomorphism from $\mathcal{E}xt_{Res, \mathcal{C}}^n(A, B)$ to $\widehat{\mathcal{E}xt}_{\mathcal{C}}^n(A, B)$, we require the following fact about abelian groups.

\begin{prop}\label{prop:liminab}
(\cite[p.~261]{osb00}) In the category $\mathbf{Ab}$ of abelian groups, the direct limit of a direct system $\lbrace M_i, \psi_{i, j} \rbrace_{i \leq j \in I}$ can be given as $\bigoplus_{i \in I} M_i/\sim$ where $m_i \in M_i \sim m_j \in M_j$ if there is $k \in I$ such that $\psi_{i, k}(m_i) = \psi_{j, k}(m_j) \in M_k$. In particular, if $\lbrace m_i \in M_i \rbrace_{i \in I}$ is a collection of nonzero elements such that $\psi_{i, j}(m_i) = m_j$ for any $i \leq j$, then the element $(m_i)_{i \in I} \in \varinjlim_{i \in I} M_i$ is nonzero.
\end{prop}

The required isomorphisms are thus given in

\begin{lem}\label{lem:vogelresol}
There is an isomorphism $\vartheta^n(B): \mathcal{E}xt_{Res, \mathcal{C}}^n(A, B) \rightarrow \widehat{\mathcal{E}xt}_{\mathcal{C}}^n(A, B)$ for any $B \in \mathrm{obj}(\mathcal{C})$ and $n \in \mathbb{Z}$.
\end{lem}

\begin{proof}
For any $k \in \mathbb{N}_0$ we construct a homomorphism
\begin{align*}
\vartheta_{n}^{2k}(B): \mathcal{E}xt_{\mathcal{C}}^{n+2k}(A, \widetilde{B}_{2k}) &\rightarrow \widehat{\mathcal{E}xt}_{\mathcal{C}}^n(A, B) \\
\varphi_{\bullet+n+2k} + \mathrm{Null}_{\mathrm{Ch}(\mathcal{C})}(A[n+2k]_{\bullet}, B_{\bullet}^{+2k}) &\mapsto \Phi_{\bullet+n} + \widehat{\mathrm{Null}}_{\mathrm{Ch}(\mathcal{C})}(A[n]_{\bullet}, B_{\bullet})
\end{align*}
as follows. For a representative
\[ \varphi_{\bullet+n+2k} \in \mathrm{Hom}_{\mathrm{Ch}(\mathcal{C})}(A[n+2k]_{\bullet}, B_{\bullet}^{+2k}) \]
of an element in $\mathcal{E}xt_{\mathcal{C}}^{n+2k}(A, \widetilde{B}_{2k})$ we define $\Phi_{\bullet+n}: A[n]_{\bullet} \rightarrow B_{\bullet}$ by
\[ \Phi_{m+n} := \begin{cases} \varphi_{m+n}: A_{m+n} \rightarrow B_m &\text{if } m \geq 2k \\ \; \; 0 &\text{if } m < 2k \end{cases} \]
As the boundary maps of $A[n]_{\bullet}$ and $A[n+2k]_{\bullet}$ agree modulo a shift of $2k$ in their indices, all squares of the form
\begin{center}
\begin{tikzcd}
  A_{m+n} \arrow[rr, "\Phi_{m+n}"] \arrow[d, "(-1)^n a_{m+n}"]
    & & B_m \arrow[d, "b_m"] \\
  A_{m-1+n} \arrow[rr, "\Phi_{n+m-1}"]
    & & B_{m-1}
\end{tikzcd}
\end{center}
commute for $m \in \mathbb{Z} \setminus \lbrace 2k \rbrace$. We conclude that $\Phi_{\bullet+n} \in \widehat{\mathrm{Hom}}_{\mathrm{Ch}(\mathcal{C})}(A[n]_{\bullet}, B_{\bullet})$. If the chain map $\varphi_{\bullet+n+2k}$ is nullhomotopic, then $\Phi_{\bullet+n}$ is genuinely nullhomotopic in any degree $m \in \mathbb{Z} \setminus \lbrace 2k-1 \rbrace$, hence nullhomotopic as an almost chain map. Thus, $\vartheta_{n}^{2k}$ is a well defined homomorphism. \\

By definition of $\vartheta_{n}^{2k}$ and by Definition~\ref{defn:hyperconnmorph}, the triangle
\begin{equation}\label{diag:thetadeltatriangle}
\begin{tikzcd}
  \mathcal{E}xt_{\mathcal{C}}^{n+2k}(A{,} \, \widetilde{B}_{2k}) \arrow[ddrr, "\vartheta_{n}^{2k}(B)"] \arrow[dddd, shift right = 6mm, "\Delta^{n+2k+1} \circ \Delta^{n+2k}"] & & {} \\ \\
  {} & & \widehat{\mathcal{E}xt}_{\mathcal{C}}^n(A{,} \, B) \\ \\
  \mathcal{E}xt_{\mathcal{C}}^{n+2k+2}(A{,} \, \widetilde{B}_{2k+2)}) \arrow[uurr, "\vartheta_{n}^{2k+2)}(B)" near end] & & {}
\end{tikzcd}
\end{equation}
commutes. We write the resulting homomorphism in the direct limit of the above triangles as
\[ \vartheta^n(B) := \varinjlim_{k \in \mathbb{N}_0} \vartheta_{n}^{2k}(B): \; \mathcal{E}xt_{Res, \mathcal{C}}^n(A, B) \rightarrow \widehat{\mathcal{E}xt}_{\mathcal{C}}^n(A, B) \, . \]
To show that it is surjective, let $\Phi_{\bullet+n}+\widehat{\mathrm{Null}}_{\mathrm{Ch}(\mathcal{C})}(A[n]_{\bullet}, B_{\bullet})$ be in $\widehat{\mathcal{E}xt}_{\mathcal{C}}^n(A, B)$. Since it is an almost chain map, there is $k \in \mathbb{N}_0$ such that $(\Phi_{m+n}: A_{m+n} \rightarrow B_m)_{m \geq 2k}$ is a chain map. Defining $\varphi_{\bullet+n+2k}: A[n+2k]_{\bullet} \rightarrow B_{\bullet}^{+2k}$ by setting
\[ \varphi_{m+n+2k} := \begin{cases}\Phi_{m+n+2k}: A_{m+n+2k} \rightarrow B_{m+2k} &\text{if } m \geq 0 \\ \; \; 0 &\text{if } m < 0 \end{cases} \]
gives an element in $\mathcal{E}xt_{\mathcal{C}}^{n+2k}(A, \widetilde{B}_{2k})$ that $\vartheta_{n}^{2k}(B)$ maps to $\Phi_{\bullet+n} + \widehat{\mathrm{Null}}_{\mathrm{Ch}(\mathcal{C})}(A[n]_{\bullet}, B_{\bullet})$. To prove that $\vartheta^n(B)$ is injective, let $x \in \varinjlim  \mathcal{E}xt_{\mathcal{C}}^{n+2k}(A, \widetilde{B}_{2k})$ with $\vartheta^n(B)(x) = 0$. According to Proposition~\ref{prop:liminab}, there exists $k \in \mathbb{N}_0$ and an element
\[  \psi_{\bullet+n+2k} \in \mathrm{Hom}_{\mathrm{Ch}(\mathcal{C})}(A[n+2k]_{\bullet}, B_{\bullet}^{+2k}) \]
that maps to $x$ in the direct limit and that lies in $\mathrm{Ker}(\vartheta_n^{2k}(B))$. One can choose the natural number $k$ such that the representative $\psi_{\bullet+n+2k}$ is nullhomotopic. As $x = 0$, $\vartheta^n(B)$ is injective and thus an isomorphism.
\end{proof}

We finally assemble our isomorphisms into the desired isomorphisms from the terms of the resolution construction $\mathrm{Ext}_{Res, \mathcal{C}}^n(A, B)$ to those of the hypercohomology construction $\widehat{\mathcal{E}xt}_{\mathcal{C}}^n(A, B)$.

\begin{defn}\label{defn:hyperresolcohom}
For every $n \in \mathbb{Z}$ there is an isomorphism
\[ \sigma^n := \vartheta^n \circ \zeta^n: \widehat{\mathrm{Ext}}_{\mathcal{C}}^n(A, -) \rightarrow \mathcal{E}xt_{Res, \mathcal{C}}^n(A, -) \rightarrow \widehat{\mathcal{E}xt}_{\mathcal{C}}^n(A, -) \]
where $\zeta^n$ is taken from Definition~\ref{defn:resolinchains} and $\vartheta^n$ is taken from Lemma~\ref{lem:vogelresol}.
\end{defn}

\subsection{Preliminaries to the isomorphisms of cohomological functors}

We require two preliminary results to demonstrate that the isomorphisms 
\[ \vartheta^n(-): \mathcal{E}xt_{Res, \mathcal{C}}^n(A, -) \rightarrow \widehat{\mathcal{E}xt}_{\mathcal{C}}^n(A, -) \quad \text{and} \quad \sigma^n: \widehat{\mathrm{Ext}}_{\mathcal{C}}^n(A, -) \rightarrow \widehat{\mathcal{E}xt}_{\mathcal{C}}^n(A, -) \]
from Lemma~\ref{lem:vogelresol} and Definition~\ref{defn:hyperresolcohom} constitute isomorphisms of cohomological functors. On the one hand, we need to to prove that $\vartheta^n(-)$ is natural. On the other hand, a technical result is needed for the next subsection in order to prove that $\vartheta^n$ and $\sigma^n$ commute with connecting homomorphisms. To provide some detail, $\vartheta^n$ is the direct limit of homomorphisms
\[ \vartheta_{n+2k}: \mathcal{E}xt_{\mathcal{C}}^{n+k}(A, \widetilde{B}_{2k}) \rightarrow \widehat{\mathcal{E}xt}_{\mathcal{C}}^n(A, B) \, . \]
Each $\vartheta_{n+2k}$ has to commute with the connecting homomorphisms $\Delta^{n+2k}$ and $\widehat{\Delta}^n$ from Definition~\ref{defn:hyperconnmorph}. According to this definition, these connecting homomorphisms are constructed using auxiliary chain maps lifted from identity morphisms. The technical result asserts that each $\vartheta_{n+2k}$ commutes with these auxiliary maps.

\begin{prop}\label{prop:thetanatural}
For every $n \in \mathbb{Z}$ the isomorphism $\vartheta^n(-)$ from Lemma~\ref{lem:vogelresol} is natural. That is, for every morphism $f: B \rightarrow C$ in $\mathcal{C}$ the square
\begin{center}
\begin{tikzcd}
  \mathcal{E}xt_{Res, \mathcal{C}}^n(A{,}\, B) \arrow[r, "\vartheta^n(B)"] \arrow[d, "\mathcal{E}xt_{Res, \mathcal{C}}^n(A{,}\, f)"] & \widehat{\mathcal{E}xt}_{\mathcal{C}}^n(A{,}\, B) \arrow[d, "\widehat{\mathcal{E}xt}_{\mathcal{C}}^n(A{,}\, f)"] \\
  \mathcal{E}xt_{Res, \mathcal{C}}^n(A{,}\, C) \arrow[r, "\vartheta^n(C)"] & \widehat{\mathcal{E}xt}_{\mathcal{C}}^n(A{,}\, C)
\end{tikzcd}
\end{center}
commutes.
\end{prop}

\begin{proof}
Let $(\widetilde{f}_k: \widetilde{M}_k \rightarrow \widetilde{N}_k)_{k \in \mathbb{N}_0}$ be a sequence of morphism as in Definition~\ref{defn:resolindmorph}, meaning that $\widehat{f}_{k+1}$ is a lift of $\widehat{f}_k$. For any $k \in \mathbb{N}_0$ consider the square
\begin{center}
\begin{tikzcd}
    \mathcal{E}xt_{\mathcal{C}}^{n+2k}(A{,} \, \widetilde{B}_{2k}) \arrow[rr, "\mathcal{E}xt_{\mathcal{C}}^{n+2k}(A{,} \, \widetilde{f}_{2k})"] \arrow[d, "\vartheta_n^{2k}(B)"] & & \mathcal{E}xt_{\mathcal{C}}^{n+2k}(A{,} \, \widetilde{C}_{2k}) \arrow[d, "\vartheta_n^{2k}(C)"] \\
    \widehat{\mathcal{E}xt}_{\mathcal{C}}^n(A{,} \, B) \arrow[rr, "\widehat{\mathcal{E}xt}_{\mathcal{C}}^n(A{,} \, f)"] & & \widehat{\mathcal{E}xt}_{\mathcal{C}}^n(A{,} \, C)
\end{tikzcd}
\end{center}
It commutes as we can lift $\widetilde{f}_{2k}: \widetilde{B}_{2k} \rightarrow \widetilde{C}_{2k}$ to a chain map $f_{\bullet}^{+2k}: B_{\bullet}^{+2k} \rightarrow C_{\bullet}^{+2k}$ given by
\[ f_m^{+2k} := \begin{cases} f_{m+2k}: B_{m+2k} \rightarrow C_{m+2k} &\text{if } m \geq 0 \\ \; \; 0 &\text{if } m < 0  \end{cases} \]
These commutative squares form a direct system because the connecting homomorphisms
\[ \Delta^{\bullet}: \mathcal{E}xt_{\mathcal{C}}^{\bullet}(A, -) \rightarrow \mathcal{E}xt_{\mathcal{C}}^{\bullet+1}(A, -) \]
are natural and Diagram~\ref{diag:thetadeltatriangle} commutes. In the direct limit, we thus obtain the commutative square in the statement of the proposition.
\end{proof}

The technical result needed to prove later that $\vartheta^n$ and $\sigma^n$ commute with connecting homomorphisms takes the following form.

\begin{prop}\label{prop:syzygyislift}
Let $0 \rightarrow B \xrightarrow{f} C \xrightarrow{g} D \rightarrow 0$ be a short exact sequence in $\mathcal{C}$ and let $B_{\bullet}$, $D_{\bullet}$ be projective resolutions. Let $h^{\ast}: \widetilde{D}_1 \rightarrow B$ be a morphism lifting $\mathrm{id}_D: D \rightarrow D$ as in Diagram~\ref{diag:delta} and let $h_{\bullet+1}: D_{\bullet}^{+1} \rightarrow B_{\bullet}$ be a chain map lifting $h^{\ast}$ as in the Comparison Theorem. Denote by $\widetilde{h}_{k+1}^{\ast}: \widetilde{D}_{k+1} \rightarrow \widetilde{B}_k$ the induced homomorphism between the corresponding syzygies where $\widetilde{h}_1^{\ast} = h^{\ast}$. \\

Then there is a projective resolution $\underline{D}_{\bullet}'$ of $D$ such that for any $k \in \mathbb{N}_0$, $\widetilde{D}_k = \widetilde{\underline{D}}_k$ and there is a commutative diagram
\begin{equation}\label{diag:syzygyislift}
\begin{tikzcd}
    0 \arrow[r] & \widetilde{D}_{k+1} \arrow[r, "\iota_{k+1}"] \arrow[d, "\widetilde{h}_{k+1}^{\ast}"] & \underline{D}_k \arrow[r, "\pi_k"] \arrow[d, "h_k"] & \widetilde{D}_k \arrow[r] \arrow[d, "\mathrm{id}"] & 0 \\
    0 \arrow[r] & \widetilde{B}_k \arrow[r, "\widetilde{f}_k"] & \widetilde{C}_k \arrow[r, "\widetilde{g}_k"] & \widetilde{D}_k \arrow[r] & 0
\end{tikzcd}
\end{equation}
where the morphisms $\iota_{k+1}$, $\pi_k$ stem from $\underline{D}_{\bullet}$. In particular, any $\widetilde{h}_{k+1}^{\ast}$ is a lift of $\mathrm{id}_{\widetilde{D}_k}$ as in Diagram~\ref{diag:delta}.
\end{prop}

We explain how one can conclude from this result that $\vartheta^n$ commutes with connecting homomorphisms. Let $0 \rightarrow \widetilde{D}_1 \rightarrow D_0 \rightarrow D \rightarrow 0$ be a short exact sequence with $D_0$ projective and denote its associated connecting homomorphism by
\[ \widehat{\Delta}^n: \widehat{\mathcal{E}xt}_{\mathcal{C}}^n(A, D) \rightarrow \widehat{\mathcal{E}xt}_{\mathcal{C}}^{n+1}(A, \widetilde{D}_1) \, . \]
According to Definition~\ref{defn:hyperconnmorph}, the connecting homomorphism associated to the short exact sequence $0 \rightarrow B \rightarrow C \rightarrow D \rightarrow 0$ is given by
\[ \widehat{\mathcal{E}xt}_{\mathcal{C}}^n(A, D) \xrightarrow{\widehat{\Delta}^n} \widehat{\mathcal{E}xt}_{\mathcal{C}}^{n+1}(A, \widetilde{D}_1) \xrightarrow{\widehat{\mathcal{E}xt}_{\mathcal{C}}^{n+1}(A, h^{\ast})} \widehat{\mathcal{E}xt}_{\mathcal{C}}^{n+1}(A, B) \, . \]
To conclude that $\vartheta^n$ commutes with this connecting homomorphism, this needs to hold for the homomorphisms
\[ \vartheta_{n+2k}: \mathcal{E}xt_{\mathcal{C}}^{n+k}(A, \widetilde{B}_{2k}) \rightarrow \widehat{\mathcal{E}xt}_{\mathcal{C}}^n(A, B) \, . \]
This is the case if the connecting homomorphism associated to the short exact sequence $0 \rightarrow \widetilde{B}_{2k} \rightarrow \widetilde{C}_{2k} \rightarrow \widetilde{D}_{2k} \rightarrow 0$ is
\[ \mathcal{E}xt_{\mathcal{C}}^{n+2k}(A, \widetilde{D}_{2k}) \xrightarrow{\Delta^n} \mathcal{E}xt_{\mathcal{C}}^{n+1+2k}(A, \widetilde{D}_{2k+1}) \xrightarrow{\mathcal{E}xt_{\mathcal{C}}^{n+1+2k}(A, \widetilde{h}_{2k+1}^{\ast})} \mathcal{E}xt_{\mathcal{C}}^{n+1+2k}(A, \widetilde{B}_{2k}) \, . \]
As this is only dependent on whether the morphism $\widetilde{h}_{2k+1}^{\ast}$ is of the appropriate form, this is where Proposition~\ref{prop:syzygyislift} comes into play.

\begin{proof}[Proposition~\ref{prop:syzygyislift}]
We proceed by induction over $k \in \mathbb{N}_0$. If $k = 0$, then we take $\underline{D}_0 := D_0$ where Diagram~\ref{diag:delta} yields the desired diagram. Assume by induction that $\underline{D}_0, \dots, \underline{D}_k$ have been defined and that there is a diagram of the required form for a number $k \geq 0$. In order to construct $\underline{D}_{k+1}$ and establish Diagram~\ref{diag:syzygyislift} in the inductive step, we consider two instance of Diagram~\ref{diag:horseshoe} from the Horseshoe Lemma. Adhering to the construction found in~\cite[p.~37--38]{wei94}, we explain how to obtain the first diagram:
\begin{equation}\label{diag:horseshoeback}
\begin{tikzcd}[row sep = scriptsize]
  {}
    & 0 \arrow[d]
    & & 0 \arrow[d]
    & & 0 \arrow[d]
    & {} \\
  0 \arrow[r]
    & \widetilde{D}_{k+2} \arrow[rr, "\iota_{k+2}''"] \arrow[dd, "\iota_{k+2}'"]
    & & \mathrm{Ker}(\varepsilon) \arrow[rr, "\pi_{k+1}''"] \arrow[dd]
    & & \widetilde{D}_{k+1} \arrow[r] \arrow[dd, "\iota_{k+1}"]
    & 0 \\ \\
  0 \arrow[r]
    & D_{k+1} \arrow[rr, "i"] \arrow[ddrr, "\iota_{k+1} \circ \pi_{k+1}'"] \arrow[dd, "\pi_{k+1}'"]
    & & D_{k+1} \oplus \underline{D}_k \arrow[rr, "p"] \arrow[dd, "\epsilon"]
    & & \underline{D}_k \arrow[r] \arrow[ddll, "\mathrm{id}"] \arrow[dd, "\pi_k"]
    & 0 \\ \\
  0 \arrow[r]
    & \widetilde{D}_{k+1} \arrow[rr, "\iota_{k+1}"] \arrow[d]
    & & \underline{D}_k \arrow[rr, "\pi_k"] \arrow[d]
    & & \widetilde{D}_k \arrow[r] \arrow[d]
    & 0 \\
  {}
    & 0
    & & 0
    & & 0
    & {}
\end{tikzcd}
\end{equation}
The morphisms $\pi_k$, $\iota_{k+1}$ stem from the projective resolution $(\underline{D}_m)_{m = 1}^k$ and the morphisms $\pi_{k+1}'$, $\iota_{k+2}'$ from $D_{\bullet}$. The morphisms of the middle row are the ones of a direct product and induce the morphisms of the top row. All rows and columns are exact. Because $\underline{D}_k$ and $D_{k+1} \oplus \underline{D}_k$ are projective, the middle column is a split exact sequence and $\mathrm{Ker}(\varepsilon)$ is also projective. Analogously, we construct with the morphisms $\widetilde{f}_k$ and $\widetilde{g}_k$ from Diagram~\ref{diag:syzygyislift} the second diagram of this form:
\begin{equation}\label{diag:horseshoefront}
\begin{tikzcd}[row sep = scriptsize]
  {}
    & 0 \arrow[d]
    & & 0 \arrow[d]
    & & 0 \arrow[d]
    & {} \\
  0 \arrow[r]
    & \widetilde{B}_{k+1} \arrow[rr, "\widetilde{f}_{k+1}"] \arrow[dd]
    & & \widetilde{C}_{k+1} \arrow[rr, "\widetilde{g}_{k+1}"] \arrow[dd]
    & & \widetilde{D}_{k+1} \arrow[r] \arrow[dd, "\iota_{k+1}"]
    & 0 \\ \\
  0 \arrow[r]
    & B_k \arrow[rr, "i = f_k"] \arrow[ddrr, "\widetilde{f}_k \circ \check{\pi}_k"] \arrow[dd, "\check{\pi}_k"]
    & & B_k \oplus \underline{D}_k \arrow[rr, "p = g_k"] \arrow[dd, "\eta"]
    & & \underline{D}_k \arrow[r] \arrow[ddll, "h_k"] \arrow[dd, "\pi_k"]
    & 0 \\ \\
  0 \arrow[r]
    & \widetilde{B}_k \arrow[rr, "\widetilde{f}_k"] \arrow[d]
    & & C_k \arrow[rr, "\widetilde{g}_k"] \arrow[d]
    & & \widetilde{D}_k \arrow[r] \arrow[d]
    & 0 \\
  {}
    & 0
    & & 0
    & & 0
    & {}
\end{tikzcd}
\end{equation}
The triangle at the bottom right corner commutes by Diagram~\ref{diag:syzygyislift}. We incorporate the lower halves of Diagram~\ref{diag:horseshoefront} and Diagram~\ref{diag:horseshoeback} into the front and back side of the diagram
\begin{equation}\label{diag:doesntcommingen}
\begin{tikzcd}
  & {\scriptstyle D_{k+1}} \arrow[rr, "{\scriptstyle i_1^1}"] \arrow[ddrr, bend right = 15, "{\scriptstyle \iota_{k+1} \circ \pi_{k+1}'}" near end] \arrow[dl, "{\scriptstyle h_{k+1}^{\ast}}"] \arrow[dd, "{\scriptstyle \pi_{k+1}'}" near end]
    & & {\scriptstyle D_{k+1} \oplus \underline{D}_k} \arrow[rr, "{\scriptstyle p_2^1}"] \arrow[dl, "{\scriptstyle h_{k+1}^{\ast} \oplus \mathrm{id}}"] \arrow[dd, shift left = 2mm, "{\scriptstyle \epsilon}" near start]
    & & {\scriptstyle \underline{D}_k} \arrow[dl, "{\scriptstyle \mathrm{id}}" near end] \arrow[ddll, bend left = 32, "{\scriptstyle \mathrm{id}}" near start] \arrow[dd, "{\scriptstyle \pi_k}"] \\
  {\scriptstyle B_k} \arrow[rr, crossing over, "{\scriptstyle i_1^2}" near end] \arrow[dd, "{\scriptstyle \check{\pi}_k}"]
    & & {\scriptstyle B_k \oplus \underline{D}_k} \arrow[rr, crossing over, "{\scriptstyle p_2^2}" near end]
    & & {\scriptstyle \underline{D}_k} & \\
  & {\scriptstyle \widetilde{D}_{k+1}} \arrow[rr, "{\scriptstyle \iota_{k+1}}" near start] \arrow[dl, "{\scriptstyle \widetilde{h}_{k+1}^{\ast}}"]
    & & {\scriptstyle \underline{D}_k} \arrow[rr, "{\scriptstyle \pi_k}" near end] \arrow[dl, "{\scriptstyle h_k}" near start]
    & & {\scriptstyle \widetilde{D}_k} \arrow[dl, "{\scriptstyle \mathrm{id}}"] \\
  {\scriptstyle \widetilde{B}_k} \arrow[rr, "{\scriptstyle \widetilde{f}_k}"]
    & & {\scriptstyle C_k} \arrow[rr, "{\scriptstyle \widetilde{g}_k}" near end] \arrow[from = uull, bend right = 20, crossing over, "{\scriptstyle \widetilde{f}_k \circ \check{\pi}_k}" near end] \arrow[from = uu, crossing over, "{\scriptstyle \eta}" near end] \arrow[from = uurr, bend left = 25, crossing over, "{\scriptstyle h_k}"]
    & & {\scriptstyle \widetilde{D}_k} \arrow[from = uu, crossing over, "{\scriptstyle \pi_k}" near end] &
\end{tikzcd}
\end{equation}
The bottom side stems from Diagram~\ref{diag:syzygyislift}, which commutes by our induction hypothesis. The left hand side arises from the chain map $h_{\bullet}^{\ast}: D_{\bullet}^{+1} \rightarrow B_{\bullet}$. The morphism $h_{k+1}^{\ast} \oplus \mathrm{id}$ is constructed as in~\cite[p.~251]{mac95} and renders the top side commutative. In particular, all squares on the outside of Diagram~\ref{diag:doesntcommingen} commute. The slanted squares in the interior give rise to the commuting diagram
\begin{center}
\begin{tikzcd}
    D_{k+1} \arrow[rr, "\iota_{k+1} \circ \pi_{k+1}'"] \arrow[dd, "h_{k+1}^{\ast}"] & & \underline{D}_k \arrow[dd, "h_k"] & & \underline{D}_k  \arrow[dd, "\mathrm{id}"] \arrow[ll, "\mathrm{id}"] \\ \\
    B_k \arrow[rr, "\widetilde{f}_k \circ \check{\pi}_k"] & & C_k & & \underline{D}_k \arrow[ll, "h_k"]
\end{tikzcd}
\end{center}
We can use it together with Diagram~\ref{diag:horseshoeback}, Diagram~\ref{diag:horseshoefront} and the content of~\cite[p.~251]{mac95} to prove that that the middle square in the interior of Diagram~\ref{diag:doesntcommingen} also commutes. Therefore, all squares in Diagram~\ref{diag:doesntcommingen} commute. Its top side is connected to its bottom side by epimorphisms according to Diagram~\ref{diag:horseshoeback} and~\ref{diag:horseshoefront}. We take kernels to obtain the commuting diagram
\begin{center}
\begin{tikzcd}
    0 \arrow[r] & \widetilde{D}_{k+2} \arrow[r, "\iota_{k+2}''"] \arrow[d, "\widetilde{h}_{k+2}^{\ast}"] & \mathrm{Ker}(\varepsilon) \arrow[r, "\pi_{k+1}''"] \arrow[d, "H"] & \widetilde{D}_{k+1} \arrow[r] \arrow[d, "\mathrm{id}"] & 0 \\
    0 \arrow[r] & \widetilde{B}_{k+1} \arrow[r, "\widetilde{f}_{k+1}"] & \widetilde{C}_{k+1} \arrow[r, "\widetilde{g}_{k+1}"] & \widetilde{D}_{k+1} \arrow[r] & 0
\end{tikzcd}
\end{center}
where the rows are short exact sequences. Having seen that $\mathrm{Ker}(\varepsilon)$ is projective, we complete the inductive step by setting $\underline{D}_{k+1} := \mathrm{Ker}(\varepsilon)$, $\iota_{k+2} : = \iota_{k+2}''$, $\pi_{k+1} := \pi_{k+1}''$ and $h_{k+1} := H$.
\end{proof}

\subsection{Isomorphisms of cohomological functors}

One goal of this subsection is to prove that the isomorphisms from the terms of the resolution construction to those of the hypercohomology construction
\[ \sigma^{\bullet}: \widehat{\mathrm{Ext}}_{\mathcal{C}}^{\bullet}(A, -) \rightarrow \widehat{\mathcal{E}xt}_{\mathcal{C}}^{\bullet}(A, -) \]
form an isomorphism of cohomological functors. This necessitates to demonstrate by the results of the previous subsection that
\[ \vartheta^{\bullet}: \mathcal{E}xt_{Res, \mathcal{C}}^{\bullet}(A, -) \rightarrow \widehat{\mathcal{E}xt}_{\mathcal{C}}^{\bullet}(A, -) \]
is an isomorphism of cohomological functors. The other goal of this subsection is to establish an isomorphism of cohomological functors from the terms of the hypercohomology construction to those of the naïve construction.

\begin{lem}\label{lem:hyperresol}
For any short exact sequence $0 \rightarrow B \rightarrow C \rightarrow D \rightarrow 0$ in $\mathcal{C}$ the square
\begin{center}
\begin{tikzcd}
    \mathcal{E}xt_{Res, \mathcal{C}}^n(A{,} \, D) \arrow[r, "\vartheta^n"] \arrow[d, "\widetilde{\Delta}^n"] & \widehat{\mathcal{E}xt}_{\mathcal{C}}^n(A{,} \, D) \arrow[d, "\widehat{\Delta}^n"] \\
    \mathcal{E}xt_{Res, \mathcal{C}}^{n+1}(A{,} \, B) \arrow[r, "\vartheta^{n+1}"] & \widehat{\mathcal{E}xt}_{\mathcal{C}}^{n+1}(A{,} \, B)
\end{tikzcd}
\end{center}
commutes for every $n \in \mathbb{Z}$. In particular, $\vartheta^{\bullet}: \mathcal{E}xt_{Res, \mathcal{C}}^{\bullet}(A, -) \rightarrow \widehat{\mathcal{E}xt}_{\mathcal{C}}^{\bullet}(A, -)$ is an isomorphism of cohomological functors and $(\widehat{\mathcal{E}xt}_{\mathcal{C}}^{\bullet}(A, -), \widehat{\Delta}^{\bullet})$ forms a cohomological functor.
\end{lem}

\begin{proof}
It suffices to prove that the square in the statement of the theorem commutes since $\vartheta^n$ is already a natural isomorphism by Proposition~\ref{prop:thetanatural}. Let $k \in \mathbb{N}_0$, $N := n+2k$ and $K := 2k$. The proof is based on the diagram
\begin{equation}\label{diag:unfoldedprism}
\begin{tikzcd}
    {\scriptscriptstyle \mathcal{E}xt_{\mathcal{C}}^N(A{,} \, \widetilde{D}_K)} \arrow[rr, "{\scriptscriptstyle \Delta^N}" near end] \arrow[rrrrr, bend right = 12, "\Delta^N"] \arrow[dd, shift right = 4mm, "{\scriptscriptstyle \Delta^N}"] & & {\scriptscriptstyle \mathcal{E}xt_{\mathcal{C}}^{N+1}(A{,} \, \widetilde{D}_{K+1})} \arrow[rrr, "{\scriptscriptstyle \mathcal{E}xt_{\mathcal{C}}^{N+1}(A{,} \, \widetilde{h}_{K+1}^{\ast})}" near end] \arrow[dd, shift right = 4mm, crossing over, "{\scriptscriptstyle \Delta^{N+1}}" near end] & & & {\scriptscriptstyle \mathcal{E}xt_{\mathcal{C}}^{N+1}(A{,} \, \widetilde{B}_K)} \arrow[dd, shift right = 4mm, "{\scriptscriptstyle \Delta^{N+1}}"] \arrow[dddddd, shift left = 6mm, bend left = 25, "{\scriptscriptstyle \vartheta_{n+1}^K(B)}"] \\ \\
    {\scriptscriptstyle \mathcal{E}xt_{\mathcal{C}}^{N+1}(A{,} \, \widetilde{D}_{K+1})} \arrow[rr, "{\scriptscriptstyle -\Delta^{N+1}}" near end] \arrow[rrrrr, bend right = 12, "-\Delta^{N+1}"] \arrow[dd, shift right = 4mm, "{\scriptscriptstyle \Delta^{N+1}}"] & & {\scriptscriptstyle \mathcal{E}xt_{\mathcal{C}}^{N+2}(A{,} \, \widetilde{D}_{K+2})} \arrow[rrr, "{\scriptscriptstyle \mathcal{E}xt_{\mathcal{C}}^{N+2}(A{,} \, \widetilde{h}_{K+2}^{\ast})}" near end] \arrow[dd, shift right = 4mm, crossing over, "{\scriptscriptstyle \Delta^{N+2}}" near end] & & & {\scriptscriptstyle \mathcal{E}xt_{\mathcal{C}}^{N+2}(A{,} \, \widetilde{B}_{K+1})} \arrow[dd, shift right = 4mm, "{\scriptscriptstyle \Delta^{N+2}}"] \\ \\
    {\scriptscriptstyle \mathcal{E}xt_{\mathcal{C}}^{N+2}(A{,} \, \widetilde{D}_{K+2})} \arrow[rr, "{\scriptscriptstyle \Delta^{N+2}}" near end] \arrow[rrrrr, bend right = 12, "\Delta^{N+2}"] \arrow[dd, shift right = 4mm, "{\scriptscriptstyle \vartheta_n^{K+2}(D)}"] & & {\scriptscriptstyle \mathcal{E}xt_{\mathcal{C}}^{N+3}(A{,} \, \widetilde{D}_{K+3})} \arrow[rrr, "{\scriptscriptstyle \mathcal{E}xt_{\mathcal{C}}^{N+2}(A{,} \, \widetilde{h}_{K+3}^{\ast})}" near end] \arrow[dd, shift right = 4mm, crossing over, "{\scriptscriptstyle \vartheta_{n+1}^{K+2}(\widetilde{D}_1)}" near end] & & & {\scriptscriptstyle \mathcal{E}xt_{\mathcal{C}}^{N+3}(A{,} \, \widetilde{B}_{K+2})} \arrow[dd, shift right = 4mm, "{\scriptscriptstyle \vartheta_{n+1}^{K+2}(B)}"] \\ \\
    {\scriptscriptstyle \widehat{\mathcal{E}xt}_{\mathcal{C}}^n(A{,} \, D)} \arrow[rr, "{\scriptscriptstyle \widehat{\Delta}^n}" near end] \arrow[rrrrr, bend right = 12, "\widehat{\Delta}^n"] \arrow[from = uuuuuu, shift left = 6mm, bend left = 25, crossing over, "{\scriptscriptstyle \vartheta_n^K(D)}"] & & {\scriptscriptstyle \widehat{\mathcal{E}xt}_{\mathcal{C}}^{n+1}(A{,} \, \widetilde{D}_1)} \arrow[rrr, "{\scriptscriptstyle \widehat{\mathcal{E}xt}_{\mathcal{C}}^{n+1}(A{,} \, h^{\ast})}" near end] \arrow[from = uuuuuu, shift left = 7mm, bend left = 25, crossing over, "{\scriptscriptstyle \vartheta_{n+1}^K(\widetilde{D}_1)}"] & & & {\scriptscriptstyle \widehat{\mathcal{E}xt}_{\mathcal{C}}^{n+1}(A{,} \, B)}
\end{tikzcd}
\end{equation}
that we construct in the following. The very bottom row renders the connecting homomorphism
\[ \widehat{\Delta}^n: \widehat{\mathcal{E}xt}_{\mathcal{C}}^n(A, D) \rightarrow \widehat{\mathcal{E}xt}_{\mathcal{C}}^{n+1}(A, B) \]
associated to $0 \rightarrow B \rightarrow C \rightarrow D \rightarrow 0$ as found in Definition~\ref{defn:hyperconnmorph}. The columns at the margins of the diagram are taken from Diagram~\ref{diag:thetadeltatriangle} and give rise to the isomorphisms
\[ \vartheta^n: \mathcal{E}xt_{Res, \mathcal{C}}^n(A, D) \rightarrow \widehat{\mathcal{E}xt}_{\mathcal{C}}^n(A, D)  \quad \text{and} \quad \vartheta^{n+1}: \mathcal{E}xt_{Res, \mathcal{C}}^{n+1}(A, B) \rightarrow \widehat{\mathcal{E}xt}_{\mathcal{C}}^{n+1}(A, B) \, . \]
On the one hand, we shall demonstrate that the very top row renders the connecting homomorphism $\Delta^N$ associated to the short exact sequence $0 \rightarrow \widetilde{B}_K \rightarrow \widetilde{C}_K \rightarrow \widetilde{D}_K \rightarrow 0$. We shall see that the middle row to the bottom renders the connecting homomorphism $\Delta^{N+2}$ associated to an analogous short exact sequence. On the other hand, if Diagram~\ref{diag:unfoldedprism} commutes, then the homomorphisms in its lower third form a direct system of commuting squares whose direct limit is the square in statement of the theorem. Therefore, it remains to prove that the diagram commutes and that the rows render the appropriate connecting homomorphisms. \\

We first show commutativity. The top left and middle left square arise from Diagram~\ref{diag:anticomm} and Lemma~\ref{lem:ordinaryhypercohom}. More specifically, their connecting homomorphisms are associated to the short exact sequences at the margins of Diagram~\ref{diag:horseshoeback}. The bottom left square and, thus, the entire left hand side commutes due to Definition~\ref{defn:hyperconnmorph} and Lemma~\ref{lem:vogelresol}. The induced homomorphisms in the right hand side are taken from Proposition~\ref{prop:syzygyislift}. More specifically, the homomorphism $h^{\ast}: \widetilde{D}_1 \rightarrow B$ is a lift of $\mathrm{id}_D: D \rightarrow D$ and the homomorphisms $\widetilde{h}_{K+1}^{\ast}$ stem from the chain map $h_{\bullet+1}^{\ast}: D_{\bullet}^{+1} \rightarrow B_{\bullet}$ lifting $h^{\ast}$ as in the Comparison Theorem. By Proposition~\ref{prop:thetanatural}, the entire right hand side commutes. \\

To conclude that Diagram~\ref{diag:unfoldedprism} commutes and to complete the proof, we explain how the rows render the connecting homomorphisms associated to their supposed short exact sequences. In the last subsection, we proved a technical result that we labelled Proposition~\ref{prop:syzygyislift}. According to it, the homomorphism at the very top of the diagram
\[ \mathcal{E}xt_{\mathcal{C}}^N(A, \widetilde{D}_K) \xrightarrow{\Delta^N} \mathcal{E}xt_{\mathcal{C}}^{N+1}(A, \widetilde{D}_{K+1}) \xrightarrow{\mathcal{E}xt_{\mathcal{C}}^{N+1}(A, \widetilde{h}_{K+1}^{\ast})} \mathcal{E}xt_{\mathcal{C}}^{N+1}(A, \widetilde{B}_K) \]
is the connecting homomorphism associated to $0 \rightarrow \widetilde{B}_K \rightarrow \widetilde{C}_K \rightarrow \widetilde{D}_K \rightarrow 0$. Analogously, the middle row to the bottom renders the connecting homomorphism associated to the short exact sequence $0 \rightarrow \widetilde{B}_{K+2} \rightarrow \widetilde{C}_{K+2} \rightarrow \widetilde{D}_{K+2} \rightarrow 0$.
\end{proof}

\begin{rem}
For any short exact sequence $0 \rightarrow B \rightarrow C \rightarrow D \rightarrow 0$ in $\mathcal{C}$ and any $n \in \mathbb{Z}$, Lemma~\ref{lem:hyperresol} implies that $\widehat{\Delta}^n: \widehat{\mathcal{E}xt}_{\mathcal{C}}^n(A, D) \rightarrow \widehat{\mathcal{E}xt}_{\mathcal{C}}^{n+1}(A, B)$ is well defined in the sense that it does not depend on a choice of a morphism $h^{\ast}: \widetilde{D}_1 \rightarrow B$ as in Definition~\ref{defn:hyperconnmorph}. We have not been able to prove this directly.
\end{rem}

We now accomplish the first goal of this subsection by Definition~\ref{defn:resolinchains} and Lemma~\ref{lem:hyperresol}:

\begin{thm}\label{thm:hyperresol}
The hypercohomology construction $(\widehat{\mathcal{E}xt}_{\mathcal{C}}^{\bullet}(A, -), \widehat{\Delta}^{\bullet})$ forms a Mislin completion of $(\mathrm{Ext}_{\mathcal{C}}^{\bullet}(A, -), \delta^{\bullet})$. More specifically,
\[ \sigma^{\bullet} = \vartheta^{\bullet} \circ \zeta^{\bullet}: \big(\widehat{\mathrm{Ext}}_{\mathcal{C}}^{\bullet}(A, -), \widehat{\delta}^{\bullet}\big) \rightarrow \big(\mathcal{E}xt_{Res, \mathcal{C}}^{\bullet}(A, -), \widetilde{\Delta}^{\bullet}\big) \rightarrow \big(\widehat{\mathcal{E}xt}_{\mathcal{C}}^{\bullet}(A, -), \widehat{\Delta}^{\bullet}\big) \]
is an isomorphism of cohomological functors where the isomorphisms $\sigma^n$ are taken from Definition~\ref{defn:hyperresolcohom}
\end{thm}

Shifting our attention to the second goal, we construct in the following an isomorphism of cohomological functors from the hypercohomology construction to the naïve construction. More specifically, in~\cite[p.~21--22]{guo23}, S.\! Guo and L.\! Liang establish an isomorphism of abelian groups $\widehat{\mathcal{E}xt}_{\mathcal{C}}^n(A, B) \rightarrow BC_{\mathcal{C}}^n(A, B)$ for every $n \in \mathbb{Z}$. As it is used in the construction of Yoneda products for complete Ext-functors in~\cite[Theorem~6.6]{ghe24}, we present it below and prove that it constitutes an isomorphism of cohomological functors.

\begin{defn}\label{defn:hypernaive}
(\cite[p.~21--22]{guo23}) For any $A, B \in \mathrm{obj}(\mathcal{C})$ and $n \in \mathbb{Z}$ define the map
\[ \rho^n: \widehat{\mathcal{E}xt}_{\mathcal{C}}^n(A, B) \rightarrow BC_{\mathcal{C}}^n(A, B) \]
as follows. If
\[ x = \varphi_{\bullet+n}+\widehat{\mathrm{Null}}_{\mathrm{Ch}(\mathcal{C})}(A[n]_{\bullet}, B_{\bullet}) \in \widehat{\mathcal{E}xt}_{\mathcal{C}}^n(A, B) \, , \]
then there exists $K \in \mathbb{N}_0$ such that $(\varphi_{m+n}: A_{m+n} \rightarrow B_m)_{m \geq 2K}$ is a chain map. Denote by $\widetilde{\varphi}_{2K+n}: \widetilde{A}_{2K+n} \rightarrow \widetilde{B}_{2K}$ the morphism induced between the corresponding syzygies. Define $\rho^n(x) \in BC_{\mathcal{C}}^n(A, B)$ to be the element to which
\[ \widetilde{\varphi}_{2K}+\mathcal{P}_{\mathcal{C}}(\widetilde{A}_{2K+n}, \widetilde{B}_{2K}) \in [\widetilde{A}_{2K+n}, B_{2K}]_{\mathcal{C}} \]
is mapped in the direct limit.
\end{defn}

\begin{lem}\label{lem:hypernaive}
The maps $\rho^n$ form an isomorphism of cohomological functors such that
\[ \rho^{\bullet} = \beta^{\bullet} \circ (\sigma^{\bullet})^{-1}: (\widehat{\mathcal{E}xt}_{\mathcal{C}}^{\bullet}(A, -), \widehat{\Delta}^{\bullet}) \rightarrow (\mathrm{Ext}_{Res, \mathcal{C}}^n(A, -), \widetilde{\delta}^{\bullet}) \rightarrow (BC_{\mathcal{C}}^{\bullet}(A, -), \tau^{\bullet}) \]
where $\beta^{\bullet}$ is taken from Theorem~\ref{thm:resolbenson} and $\sigma^{\bullet}$ from Theorem~\ref{thm:hyperresol}.
\end{lem}

\begin{proof}
Since $\beta^{\bullet}$ and $\sigma^{\bullet}$ are already isomorphisms of cohomological functors, it suffices to show that $\rho^n = \beta^n \circ (\sigma^{n})^{-1}$ for every $n \in \mathbb{Z}$. This follows from Definition~\ref{defn:hyperresolcohom} and Definition~\ref{defn:hypernaive}.
\end{proof}

We close our account on constructions of Mislin completions by posing a question. At the start of Subsection~\ref{subsec:hyperindconn} we have seen that any completed Ext-group $\widehat{\mathcal{E}xt}_{\mathcal{C}}^n(A, B)$ can be defined as the $n^{\text{th}}$ cohomology group of the Vogel complex $\mathrm{Vog}_{\mathcal{C}}^n(A_{\bullet}, B_{\bullet})_{\bullet}$. Therefore, one can define connecting homomorphisms for the hypercohomology construction in a different manner as the following lemma demonstrates.

\begin{lem}
(\cite[Proposition~4.8]{guo23}) Let $0 \rightarrow B \rightarrow C \rightarrow D \rightarrow 0$ be a short exact sequence in $\mathcal{C}$ and let $0 \rightarrow B_{\bullet} \rightarrow C_{\bullet} \rightarrow D_{\bullet} \rightarrow 0$ be a short exact sequence of chain complexes obtained by the Horseshoe Lemma~\cite[p.~37]{wei94}. According to the Snake Lemma~\cite[p.~11--12]{wei94}, there are associated connecting homomorphisms
\[ \overline{\delta}^n: \widehat{\mathcal{E}xt}_{\mathcal{C}}^n(A, D) \rightarrow \widehat{\mathcal{E}xt}_{\mathcal{C}}^{n+1}(A, B) \quad \text{and} \quad \overline{\delta}^n: \widehat{\mathcal{E}xt}_{\mathcal{C}}^n(B, A) \rightarrow \widehat{\mathcal{E}xt}_{\mathcal{C}}^{n+1}(D, A) \, . \]
In particular, $\big(\widehat{\mathcal{E}xt}_{\mathcal{C}}^{\bullet}(A, -), \overline{\delta}^{\bullet}\big)$ and $\big(\widehat{\mathcal{E}xt}_{\mathcal{C}}^{\bullet}(-, A), \overline{\delta}^{\bullet}\big)$ form cohomological functors.
\end{lem}

\begin{ques}\label{ques:comparewithguoliang}
Does the connecting homomorphism $\widehat{\Delta}^{\bullet}$ constructed in Definition~\ref{defn:hyperconnmorph} agree with $\overline{\delta}^{\bullet}$? More specifically, is $\big(\widehat{\mathcal{E}xt}_{\mathcal{C}}^{\bullet}(A, -), \overline{\delta}^{\bullet}\big)$ a Mislin completion of the Ext-functors $\big(\mathrm{Ext}_{\mathcal{C}}^{\bullet}(A, -), \delta^{\bullet}\big)$?
\end{ques}

\section{Complete cohomology in condensed mathematics}\label{sec:condmaths}

In this last section, we develop complete cohomology for any $T1$ topological group via the uniform theory from this paper and condensed mathematics. For this purpose, we first provide an outline of condensed mathematics. By establishing complete cohomology for $T1$ topological groups, we also establish an example of completed Ext-functors that cannot be developed in terms of any previous framework. Thus far, cohomology for a topological group $G$ has been usually constructed via continuous homological algebra, meaning via classes of topological modules with a continuous $G$-action. Hence, it is of interest how continuous cohomology of $G$ compares to its condensed cohomology. Although one can turn every $T1$ topological module into a condensed module, this does not automatically transfer the required continuous homological algebra for $G$ into condensed mathematics. The issue here is that projective continuous modules do not result in projective condensed modules in general. We end this paper by proving that complete continuous cohomology for profinite $G$ is isomorphic as a cohomological functor to its more general, complete condensed cohomology. \\

The building blocks of condensed mathematics are condensed sets that we define following Clausen and Scholze's account~\cite[Lecture~I and~II]{sch22}. In short, condensed sets can be thought as sheaves of sets. The reason for using sheaves is because they have excellent category-theoretic properties that are fundamental for group cohomology. Further below we provide a down-to-earth and straightforward definition of $\kappa$-condensed sets after introducing the relevant Grothendieck sites of our sheaves. A profinite space is an inverse limit of finite discrete spaces, meaning that it can be assembled from finite sets in a particular manner. Let $\mathbf{Pro}$ denote the category of profinite spaces. A finite collection of (continuous) maps $\lbrace f_i: S_i \rightarrow S \rbrace_{i = 1}^n$ in $\mathbf{Pro}$ is called a covering if the induced map $\bigsqcup_{i = 1}^n S_i \rightarrow S$ is surjective. Although this turns $\mathbf{Pro}$ into a Grothendieck site, it is not well suited for sheaf-theoretic purposes as it is not essentially small. A cardinal number $\kappa$ is called a strong limit cardinal if for every $\lambda < \kappa$ we also have $2^{\lambda} < \kappa$. For an uncountable strong limit cardinal $\kappa$, we define the site of $\kappa$-small profinite spaces $\mathbf{Pro}_{\kappa}$ as the full subcategory of $\mathbf{Pro}$ consisting of all profinite spaces with less than $\kappa$ clopen subsets. Note that the Grothendieck site $\mathbf{Pro}_{\kappa}$ is essentially small and thus suited for sheaf-theoretic purposes. Then a $\kappa$-condensed set is a sheaf of sets on $\mathbf{Pro}_{\kappa}$. In more down-to-earth terms, a (contravariant) functor
\[ X: \mathbf{Pro}_{\kappa}^{\mathrm{op}} \rightarrow \mathbf{Set} \]
is a $\kappa$-condensed set if the following two conditions are satisfied.
\begin{enumerate}
    \item For every finite collection $\lbrace S_i \rbrace_{i = 1}^n$ of objects in the category $\mathbf{Pro}_{\kappa}$ there is a bijection $X(\bigsqcup_{i = 1}^n S_i) \rightarrow \prod_{i = 1}^n X(S_i)$.
    \item Let $f: T \rightarrow S$ be a (continuous) surjection in $\mathbf{Pro}_{\kappa}$ and denote the projections from the fiber product by $p_1, p_2: T \times_S T \rightarrow T$. Then the function $X(f): X(S) \rightarrow X(T)$ is injective with image
\[ \mathrm{Im}(X(f)) = \lbrace x \in X(T) \mid X(p_1)(x) = X(p_2)(x) \in X(T \times_S T) \rbrace \, . \]
\end{enumerate}
One can think of $X$ as encoding a topological space and of $X(S)$ as the continuous maps ``$S \rightarrow X$'' from a $\kappa$-profinite space. If $\kappa < \kappa'$ is another strong limit cardinal, then there is a functor
\[ \mathrm{Cond}_{\kappa}(\mathbf{Set}) \rightarrow \mathrm{Cond}_{\kappa'}(\mathbf{Set}) \, . \]
It is fully faithful because we have used strong limit cardinals. This means that one can extend any $\kappa$-condensed set to a $\kappa'$-condensed set. If $K$ denotes the class of uncountable strong limit cardinals, then the category of condensed sets is defined as the colimit category
\[ \mathrm{Cond}(\mathbf{Set}) := \varinjlim_{\kappa \in K} \mathrm{Cond}_{\kappa}(\mathbf{Set}) \, . \]
Hence, a condensed set is an equivalence class of $\kappa$-condensed sets and thus, an equivalence class of sheaves of sets. Although $\mathrm{Cond}(\mathbf{Set})$ does not form a category of sheaves over any site~\cite[Remark~2.12]{sch19}, it nevertheless inherits good category-theoretic properties from the categories $\mathrm{Cond}_{\kappa}(\mathbf{Set})$~\cite[p.~13]{sch22}. \\

We now define condensed groups and modules, demonstrate how topological groups and modules can be turned into condensed objects and present relevant category-theoretic properties. By~\cite[p.~7]{sch19} condensed groups are defined as equivalence classes of $\kappa$-condensed groups, meaning sheaves of groups on the site $\mathbf{Pro}_{\kappa}$. One defines other condensed algebraic structures such as condensed rings and condensed modules analogously. According to~\cite[p.~15/16 and Proposition~1.7]{sch19} we can establish the following relation with topological spaces. If $T1\text{-}\mathbf{Top}$ denotes the category of $T1$ topological spaces, then there is a faithful functor $T1\text{-}\mathbf{Top} \rightarrow \mathrm{Cond}(\mathbf{Set})$ taking any $Y$ to its condensate
\[ \underline{Y}: \mathbf{Pro}_{\kappa}^{\mathrm{op}} \rightarrow \mathbf{Set}, \; S \mapsto \lbrace f: S \rightarrow Y \text{ continuous} \rbrace \]
where $\kappa$ is chosen to be a sufficiently large strong limit cardinal. This functor does not depend on $\kappa$ as evidenced by the proofs of Proposition~2.9 and Proposition~2.15 in~\cite{sch19}. Thus, one can replace any $T1$ topological space by its condensate and only work with the latter. Condensates of $T1$ topological groups (resp.\! rings, modules etc) are condensed groups (resp.\! rings, modules etc)~\cite[p.~8]{sch19}. In the category of condensed abelian groups $\mathrm{Cond}(\mathbf{Ab})$ all limits and colimits exist where arbitrary products, arbitrary direct sums and direct limits are exact~\cite[Theorem~1.10]{sch19}. According to~\cite[p.~12]{sch19}, $\mathrm{Cond}(\mathbf{Ab})$ has enough projectives. Following the source, these assertions also hold for the category of condensed $\mathcal{R}$-modules $\mathrm{Cond}(Mod(\mathcal{R}))$ where $\mathcal{R}$ denotes a condensed ring. This is relevant to our theory as it constructs Mislin completions for functors $T^{\bullet}: \mathcal{C} \rightarrow \mathcal{D}$ where $\mathcal{C}$ has enough projectives and in $\mathcal{D}$ all countable direct limits exists and are exact. Therefore, one can construct Mislin completions using condensed modules. \\

We construct free condensed modules in order to define the above mentioned completed Ext-functors that cannot be developed in terms of any previous framework. Let $S$ be a projective object in $\mathbf{Pro}$. Then the free condensed $\mathcal{R}$-module $\mathcal{R}[\underline{S}]$ is the sheafification of the presheaf of $\mathcal{R}$-modules
\[ \mathbf{Pro}_{\kappa}^{\mathrm{op}} \rightarrow \mathbf{Ab}, T \mapsto \mathcal{R}(T)[\underline{S}(T)] \]
where $\mathcal{R}(T)[\underline{S}(T)]$ denotes the free $\mathcal{R}(T)$-module over the set $\underline{S}(T)$ and $\kappa$ is a sufficiently large strong limit cardinal~\cite[p.~12]{sch19}. Since the Grothendieck site $\mathbf{Pro}_{\kappa}$ is essentially small, sheafification exists and the construction of $\mathcal{R}[\underline{S}]$ is independent of $\kappa$ by the proof of~\cite[Proposition~2.9]{sch19}. The condensed modules $\mathcal{R}[\underline{S}]$ are compact projective generators, meaning that for every condensed $\mathcal{R}$-module $A$ there is a collection of projective profinite spaces $S_i$ such that there is an epimorphism $\bigoplus_{i \in I} \mathcal{R}[\underline{S_i}] \rightarrow A$~\cite[p.~12]{sch19}. For condensed $\mathcal{R}$-modules $A$, $B$ one does not only have the unenriched Hom-set
\[ \mathrm{Hom}_{\mathrm{Cond}(Mod(\mathcal{R}))}(A, B) \in \mathrm{obj}(\mathbf{Ab}) \, , \]
but also an internal Hom-set which is constructed in~\cite[p.~13]{sch19}. Namely, for any condensed $\mathcal{R}$-module $C$ the condensed abelian group $C \otimes_{\mathcal{R}} A$ is the sheafification of
\[ \mathbf{Pro}_{\kappa}^{\mathrm{op}} \rightarrow \mathbf{Ab}, T \mapsto C(T) \otimes_{\mathcal{R}(T)} A(T) \, . \]
Then define the internal Hom-set
\[ \underline{\mathrm{Hom}}_{\mathrm{Cond}(Mod(\mathcal{R}))}(A, B) \in \mathrm{obj}(\mathrm{Cond}(\mathbf{Ab})) \]
by
\[ \mathbf{Pro}_{\kappa}^{\mathrm{op}} \rightarrow \mathbf{Ab}, T \mapsto \mathrm{Hom}_{\mathrm{Cond}(Mod(\mathcal{R}))}(A \otimes_{\mathcal{R}} \mathcal{R}[\underline{T}], B) \, . \]
By definition, it satisfies the tensor-Hom adjunction
\[ \mathrm{Hom}_{\mathrm{Cond}(Mod(\mathcal{R}))}(C, \underline{\mathrm{Hom}}_{\mathrm{Cond}(Mod(\mathcal{R}))}(A, B)) \cong \mathrm{Hom}_{\mathrm{Cond}(Mod(\mathcal{R}))}(C \otimes_{\mathcal{R}} A, B) \, . \]
Because the category $\mathrm{Cond}(Mod(\mathcal{R}))$ has enough projectives, we can define the unenriched Ext-functors
\[ \mathrm{Ext}_{\mathrm{Cond}(Mod(\mathcal{R}))}^{\bullet}(A, -): \mathrm{Cond}(Mod(\mathcal{R})) \rightarrow \mathbf{Ab} \]
and the internal Ext-functors
\[ \underline{\mathrm{Ext}}_{\mathrm{Cond}(Mod(\mathcal{R}))}^{\bullet}(A, -): \mathrm{Cond}(Mod(\mathcal{R})) \rightarrow \mathrm{Cond}(\mathbf{Ab}) \]
to be the derived functors of $\mathrm{Hom}_{\mathrm{Cond}(Mod(\mathcal{R}))}(A, -)$ and $\underline{\mathrm{Hom}}_{\mathrm{Cond}(Mod(\mathcal{R}))}(A, -)$. As we shall see shortly, the internal completed Ext-functors cannot be established in any previous framework. \\

One can define two forms of condensed group cohomology for any $T1$ topological group using the above Ext-functors. More specifically, for any condensed group $\mathcal{G}$ one can form a condensed group ring $\mathcal{R}[\mathcal{G}]$. Then the category $\mathrm{Cond}(Mod(\mathcal{R}[\mathcal{G}]))$ is equivalent to the subcategory of $\mathrm{Cond}(Mod(\mathcal{R}))$ consisting of all objects with an $\mathcal{G}$-action~\cite[p.~2--3]{ans20}. By~\cite[p.~5/8]{ans20}, one can define condensed unenriched group cohomology $H_{\mathcal{R}}^{\bullet}(\mathcal{G}, -)$ and the condensed internal group cohomology $\underline{H}_{\mathcal{R}}^{\bullet}(\mathcal{G}, -)$ as the Ext-functors $\mathrm{Ext}_{\mathrm{Cond}(Mod(\mathcal{R}[\mathcal{G}]))}^{\bullet}(\mathcal{R}, -)$ and $\underline{\mathrm{Ext}}_{\mathrm{Cond}(Mod(\mathcal{R}[\mathcal{G}]))}^{\bullet}(\mathcal{R}, -)$. We have thus developed complete cohomology for $T1$ topological groups and a novel example of completed Ext-functors:

\begin{thm}\label{thm:condgroupcohom}
Let $\mathcal{R}$ be a condensed ring and $A$ a condensed $\mathcal{R}$-module. Then there exist completed condensed unenriched Ext-functors
\[ \widehat{\mathrm{Ext}}_{\mathrm{Cond}(Mod(\mathcal{R}))}^{\bullet}(A, -): \mathrm{Cond}(Mod(\mathcal{R})) \rightarrow \mathbf{Ab} \]
and completed condensed internal Ext-functors
\[ \widehat{\underline{\mathrm{Ext}}}_{\mathrm{Cond}(Mod(\mathcal{R}))}^{\bullet}(A, -): \mathrm{Cond}(Mod(\mathcal{R})) \rightarrow \mathrm{Cond}(\mathbf{Ab}) \, . \]
If $\mathcal{G}$ is a condensed group, then there exists complete condensed unenriched group cohomology
\[ \widehat{H}_{\mathcal{R}}^{\bullet}(\mathcal{G}, -): \mathrm{Cond}(Mod(\mathcal{R}[\mathcal{G}])) \rightarrow \mathbf{Ab} \]
and complete condensed internal group cohomology
\[ \widehat{\underline{H}}_{\mathcal{R}}^{\bullet}(\mathcal{G}, -): \mathrm{Cond}(Mod(\mathcal{R}[\mathcal{G}])) \rightarrow \mathrm{Cond}(\mathbf{Ab}) \, . \]
For any $T1$ topological group $G$, there exists complete condensed cohomology $\widehat{H}_{\mathcal{R}}^{\bullet}(\underline{G}, -)$ and complete internal condensed cohomology $\widehat{\underline{H}}_{\mathcal{R}}^{\bullet}(\underline{G}, -)$ defined over its condensate $\underline{G}$.
\end{thm}

\begin{rem}\label{rem:differencetoguoliang}
The cohomological functors $\widehat{\underline{\mathrm{Ext}}}_{\mathrm{Cond}(Mod(\mathcal{R}))}^{\bullet}(A, -)$ and $\widehat{\underline{H}}_{\mathcal{R}}^{\bullet}(\mathcal{G}, -)$ cannot be developed in terms of previous frameworks such as the one by Beligiannis and Reiten in~\cite{bel07}, the one by S.\! Guo and L.\! Liang's in~\cite{guo23} and the one by J.\! Hu et al.\! in~\cite{hu21}. This is because they have $\mathrm{Cond}(\mathbf{Ab})$ as codomain category and not a category of modules.
\end{rem}

We end this paper by proving that complete continuous cohomology for a profinite group $G$ is isomorphic as a cohomological functor to its complete condensed cohomology. For this purpose, we introduce a special class of condensed modules. Namely, in as a completed group ring $R \llbracket G \rrbracket$ of a profinite group $G$ over a profinite ring $R$ is a completion of an ordinary group ring $R[G]$~\cite[Proposition~7.1.2]{wil98}, one can introduce a notion of completion for condensed modules over a condensed ring $\mathcal{R}$. We call a condensed ring admitting such a form of completion an analytic ring and a condensed module satisfying this completion a solid module. More formally, let $\mathbf{ProjPro}$ denote the full subcategory of $\mathbf{Pro}$ consisting of projective objects. Then an analytic ring $\mathcal{A}$ comes with a particular functor
\[ c_{\mathcal{A}}: \mathbf{ProjPro} \rightarrow \mathrm{Cond}(Mod(\mathcal{A})), \; S \mapsto \mathcal{A}[S]_{\blacksquare}  \, . \]
and a condensed $\mathcal{A}$-module homomorphism $C_{\mathcal{A}}(S): \mathcal{A}[\underline{S}] \rightarrow \mathcal{A}[\underline{S}]_{\blacksquare}$ for every projective profinite space $S$~\cite[p.~44]{sch19}. One can think of $c_{\mathcal{A}}$ as assigning to every projective profinite space $S$ a canonical choice of free solid $\mathcal{A}$-module $\mathcal{A}[S]_{\blacksquare}$. Let $\mathrm{Solid}(\mathcal{A})$ denote the full subcategory of $\mathrm{Cond}(Mod(\mathcal{A}))$ such that for all its objects $M$ and for every $S \in \mathbf{ProjPro}$,
\[ \mathrm{Hom}_{\mathrm{Cond}(Mod(\mathcal{A}))}(C_{\mathcal{A}}(S), M): \mathrm{Hom}_{\mathrm{Cond}(Mod(\mathcal{A}))}(\mathcal{A}[S]_{\blacksquare}, M) \rightarrow M(S) \]
is an isomorphism. According to~\cite[Proposition~7.5]{sch19}, $\mathrm{Solid}(\mathcal{A})$ is an abelian category stable under limits, colimits and extensions in which the objects $\mathcal{A}[S]_{\blacksquare}$ are compact projective generators. In particular, we can define the Ext functors
\[ \mathrm{Ext}_{\mathrm{Solid}(\mathcal{A})}^{\bullet}(M, -): \mathrm{Solid}(\mathcal{A}) \rightarrow \mathbf{Ab} \]
as the derived functors of $\mathrm{Hom}_{\mathrm{Solid}(\mathcal{A})}(M, -)$. Consequently, there exist completed solid Ext-functors
\[ \widehat{\mathrm{Ext}}_{\mathrm{Solid}(\mathcal{A})}^{\bullet}(M, -): \mathrm{Solid}(\mathcal{A}) \rightarrow \mathbf{Ab} \]
as the derived func
Moreover, there exists an extension of the morphisms $C_{\mathcal{A}}(S): \mathcal{A}[\underline{S}] \rightarrow \mathcal{A}[\underline{S}]_{\blacksquare}$ to a functor $\mathrm{Cond}(Mod(\mathcal{A})) \rightarrow \mathrm{Solid}(\mathcal{A})$ according to~\cite[Proposition~7.5]{sch19}. We denote this functor as solidification. We need it and the following example to establish the isomorphism from the complete continuous cohomology of a profinite group $G$ to the complete condensed cohomology of its condensate $\underline{G}$. 

\begin{exl}\label{exl:analyticrings}
Let $S = \varprojlim_{i \in I} S_i$ be a profinite space and define the condensed $\underline{\mathbb{Z}}$-module $\mathbb{Z}[S]^{\blacksquare} := \varprojlim_{i \in I} \underline{\mathbb{Z}}[\underline{S_i}]$. Then the condensed ring $\underline{\mathbb{Z}}$ together with the assignment $S \mapsto \mathbb{Z}[S]^{\blacksquare}$ forms an analytic ring~\cite[p.~46]{sch19}. Denote by $(-)^{\blacksquare}: \mathrm{Cond}(\mathbf{Ab}) \rightarrow \mathrm{Solid}(\underline{\mathbb{Z}})$ the associated solidification functor. If $\mathcal{R}$ is a condensed ring, then $(\mathcal{R}, S \mapsto (\mathcal{R}[\underline{S}])^{\blacksquare})$ is an analytic ring~\cite[Proposition~3.6]{tan24}.
\end{exl}

The desired and final result is thus

\begin{lem}\label{lem:solidandctscohom}
Let $G$ be a profinite group and $R$ be a profinite commutative ring. Denote by $R \llbracket G \rrbracket$ the completed group ring and by $PMod(R\llbracket G \rrbracket)$ denote the category of profinite $R \llbracket G \rrbracket$-modules. Endow the condensed ring $\underline{R \llbracket G \rrbracket}$ with an analytic ring structure as in Example~\ref{exl:analyticrings}. Then
\begin{enumerate}
\item There exists complete solid solid group cohomology
\[ \widehat{H}_{\mathrm{Solid}(R)}^{\bullet}(\underline{G}, -): \mathrm{Solid}(\underline{R \llbracket G \rrbracket}) \rightarrow \mathbf{Ab} \, . \]
\item This cohomology is isomorphic as a cohomological functor to complete continuous cohomology
\[ \widehat{H}_R^{\bullet}(G, -): PMod(R \llbracket G \rrbracket) \rightarrow \mathbf{Ab} \, . \]
\end{enumerate}
\end{lem}

\begin{proof}
The first assertion holds because the $\underline{R \llbracket G \rrbracket}$-module $\underline{R}$ is solid by~\cite[pp.~14--15]{tan24} and solid group cohomology $H_{\mathrm{Solid}(R)}^{\bullet}(\underline{G}, -) = \mathrm{Ext}_{\mathrm{\underline{R \llbracket G \rrbracket}}}^{\bullet}(\underline{R}, -)$ is thus well defined. The second assertion follows from~\cite[Theorem~3.14]{tan24} because taking condensates constitutes an exact functor $PMod(R \llbracket G \rrbracket) \rightarrow \mathrm{Solid}(\underline{R \llbracket G \rrbracket})$ that preserves projectives.
\end{proof}

\section*{Acknowledgements}

A special thanks goes to Peter H. Kropholler for sharing his expertise with me. He has discussed examples of complete cohomology groups of discrete groups with me and pointed out his joint paper with Jonathan Cornick ``On Complete Resolutions''. In particular, he introduced me to condensed mathematics. I would also like to express my gratitude to Emma Brink for her great help in condensed mathematics. \\
%I am indebted to the anonymous referees whose invaluable feedback has substantially improved the paper. \\

I would like to acknowledge Alejandro Adem, Jon F.\! Carlson, Joel Friedman, Kalle Karu, Zinovy Reichstein and Ben Williams who provided invaluable suggestions as part of assessing my PhD thesis. Moreover, I am thankful to Andrew Fisher for pointing out the paper ``Complete Homology over Associative Rings'' by Olgur Celikbas et al. Lastly, I am grateful to Ged Corob Cook, Julian Feuerpfeil and Sofiya Yatsyna for insightful discussions.

%next line adds the Bibliography to the contents page
\addcontentsline{toc}{section}{References}
%change bibliography name to references
\renewcommand{\bibname}{References}
\bibliographystyle{plain}  %use the plain bibliography style
\bibliography{refs}        %use a bibtex bibliography file refs.bib

\end{document}